\documentclass[12pt]{amsart}
\usepackage{accents}
\usepackage{appendix}
\usepackage{amsfonts}
\usepackage{amsmath}
\usepackage{amssymb}	
\usepackage{blkarray}
\usepackage{mathrsfs}
\usepackage{amsthm}
\usepackage{array,booktabs,multirow}
\usepackage{braket}
\usepackage{cite}
\usepackage{centernot}
\usepackage{dsfont}
\usepackage{kbordermatrix}
\usepackage{mathalfa}
\usepackage[shortlabels]{enumitem} 
\usepackage{etoolbox}
\usepackage{float}
\usepackage[hang, flushmargin]{footmisc}
\usepackage{latexsym}
\usepackage{lipsum}
\usepackage{needspace}
\usepackage{tikz}
\usepackage{mathrsfs}
\usepackage{hyperref}
\usetikzlibrary{matrix,arrows}  

\usepackage{tikz-cd}

\theoremstyle{plain}
\newtheorem{thm}{Theorem}[section]
\newtheorem{cor}[thm]{Corollary}
\newtheorem{lem}[thm]{Lemma}
\newtheorem{prop}[thm]{Proposition}

\theoremstyle{definition}

\newtheorem{defn}[thm]{Definition}

\theoremstyle{remark}

\setlist[enumerate,1]{leftmargin=2em}

\def\BI{\mathfrak{BI}}
\def\C{\mathbb C}
\def\Cl{{\rm Cl}}
\def\D{\mathbf D}

\def\R{\mathbb R}
\def\M{\mathscr M}
\def\N{\mathscr N}

\def\W{\mathcal W}

\def\Z{\mathbb Z}

\def\Sym{\mathfrak{S}}
\def\rank{{\rm rank}}

\def\e{\varepsilon}

\newcommand{\rvline}{\hspace*{-\arraycolsep}\vline\hspace*{-\arraycolsep}}
\renewcommand{\kbldelim}{(}
\renewcommand{\kbrdelim}{)}

\title[The space $\M_n$ and the algebra $\BI$]
{The space of Dunkl monogenics associated with $\Z_2^3$ and the universal Bannai--Ito algebra}

\author{Hau-Wen Huang}
\address{
Hau-Wen Huang\\
Department of Mathematics\\
National Central University\\
Chung-Li 32001 Taiwan
}
\email{hauwenh@math.ncu.edu.tw}

\begin{document}

\begin{abstract}
Let $n\geq 0$ denote an integer.  
Let $\M_n$ denote the space of Dunkl monogenics of degree $n$ associated with the reflection group $\Z_2^3$. 
The universal Bannai--Ito algebra $\BI$ is a unital associative algebra over $\C$ generated by $X,Y,Z$ and the relations assert that each of 
\begin{gather*}
\{X,Y\}-Z,
\qquad 
\{Y,Z\}-X,
\qquad 
\{Z,X\}-Y
\end{gather*}
commutes with $X,Y,Z$. 
When the multiplicity function $k$ is real-valued the space $\M_n$ supports a $\BI$-module in terms of the symmetries of the spherical Dirac--Dunkl operator.  Under the assumption that $k$ is nonnegative, it was shown that $\dim \M_n=2(n+1)$ and $\M_n$ is isomorphic to a direct sum of two copies of an $(n+1)$-dimensional irreducible $\BI$-module. 
In this paper, we improve the aforementioned result.
\end{abstract}

\maketitle



\section{Introduction}

The idea of Dunkl operators was to use the finite reflection groups to generalize the standard partial derivatives \cite{Dunkl1989}. Since that time, the connections between the Dunkl operators and many other areas \cite{Dunkl,Cherednik:1991,Dunkl&Harmonic,Dunkl&Fourier,Dunkltransform} have been explored such as Fourier transforms, quantum-mechanical integrable systems and double affine Hecke algebras. 
The goal of this paper is to improve the main result of \cite{BI2016} on the Dunkl operators associated 
with the finite reflection group $\Z_2^3$ and the universal Bannai--Ito algebra. 
For notational convenience, we adopt the following conventions: 
The unadorned tensor product $\otimes$ is meant to be over the real number field $\R$.  The notation ``$\dim$'' is meant to be the dimension over the complex number field $\C$. 
An algebra is meant to be a unital associative algebra. For any integer $n\geq 1$ the notation ${\rm Mat}_n$ stands for the algebra consisting of all $n\times n$ matrices.
In an algebra the curly bracket $\{\, , \}$ stands for the anticommutator. Given a polynomial ring $R$ and an integer $n\geq 0$, let $R_n$ denote the subspace of $R$ spanned by the homogeneous polynomials in $R$ of degree $n$.

Consider the root system
$$
\Phi=\{\pm(1,0,0), \pm(0,1,0), \pm(0,0,1)\}
$$
in the standard three-dimensional Euclidean space $\R^3$. 
The finite reflection group $\W$ of $\Phi$ is isomorphic to $\Z_2^3$. A {\it multiplicity function} $k:\Phi\to \C$ is a $\W$-invariant complex-valued function on $\Phi$. In this paper we assume that $k$ is real-valued and write
\begin{align*}
k_1 =k (1,0,0),
\qquad 
k_2 =k(0,1,0),
\qquad 
k_3 =k(0,0,1).
\end{align*}
The {\it Dunkl operators} $T_1,T_2,T_3$ associated to $(1,0,0),(0,1,0),(0,0,1)$  are given as follows:
\begin{align*}
(T_1f)(x_1,x_2,x_3)
&=
\frac{\partial f}{\partial{x_1}}(x_1,x_2,x_3)
+
k_1\frac{f(x_1,x_2,x_3)-f(-x_1,x_2,x_3)}{x_1},
\\
(T_2f)(x_1,x_2,x_3)
&=
\frac{\partial f}{\partial{x_2}}(x_1,x_2,x_3)
+
k_2\frac{f(x_1,x_2,x_3)-f(x_1,-x_2,x_3)}{x_2},
\\
(T_3f)(x_1,x_2,x_3)
&=
\frac{\partial f}{\partial{x_3}}(x_1,x_2,x_3)
+
k_3\frac{f(x_1,x_2,x_3)-f(x_1,x_2,-x_3)}{x_3}
\end{align*}
for all smooth functions $f$ on $\R^3$. 
The {\it Clifford algebra} $\Cl(\R^3)$ of $\R^3$ is an algebra over $\R$ generated by $e_1,e_2,e_3$ with the relations
\begin{gather*}
e_1^2=1,\qquad e_2^2=1,\qquad e_3^2=1,
\\
\{e_1,e_2\}=\{e_2,e_3\}=\{e_3,e_1\}=0.
\end{gather*}
Recall that the {\it Pauli matrices}  $\sigma_1,\sigma_2,\sigma_3$ are given by 
$$
\sigma_1
=\begin{pmatrix}
0 &1
\\
1 &0
\end{pmatrix},
\qquad 
\sigma_2
=\begin{pmatrix}
0 &(-1)^{\frac{3}{2}}
\\
(-1)^{\frac{1}{2}} &0
\end{pmatrix},
\qquad 
\sigma_3
=\begin{pmatrix}
1 &0
\\
0 &-1
\end{pmatrix}.
$$
The vector space $\C^2$ supports a $\Cl(\R^3)$-module via the representation $\Cl(\R^3)\to {\rm Mat}_2(\C)$ given by 
\begin{eqnarray*}
e_1 &\mapsto & \sigma_1,
\qquad 
e_2 \;\;\mapsto\;\; \sigma_2,
\qquad 
e_3 \;\;\mapsto\;\; \sigma_3.
\end{eqnarray*}
The {\it Dirac--Dunkl operator} $\D$ is defined as
$$
\D=e_1 \otimes T_1+e_2\otimes T_2+e_3\otimes T_3.
$$
Given an integer $n\geq 0$ the {\it space $\M_n$ of Dunkl monogenics of degree $n$} is defined as
$$
\M_n=\ker (\D|_{\C^2\otimes \R[x_1,x_2,x_3]_n}).
$$

The {\it universal Bannai--Ito algebra} $\BI$  \cite{Huang:R<BImodules,R&BI2015,Vinet2019,BI&NW2016,BI2015,BI2014,BI&osp2018,BI2014-2,BI2015-2,BI2017,Huang:R<BI,Huang:BImodule} is an algebra over $\C$ generated by $X,Y,Z$ and the relations assert that each of the following elements commutes with $X,Y,Z$:
\begin{align}
&\{X,Y\}-Z, 
\label{kappa}
\\
&\{Y,Z\}-X,
\label{lambda}
\\ 
&\{Z,X\}-Y.
\label{mu}
\end{align}
Let $\kappa,\lambda,\mu$ denote the central elements (\ref{kappa})--(\ref{mu}) of $\BI$, respectively. Note that the universal Bannai--Ito algebra is the case $q=-1$ of the universal Askey--Wilson algebra \cite{Huang:AW&DAHAmodule,Huang:2021,centerAW:2016,Huang:RW,uaw&equit2011,DAHA2013,ter2018,uaw2011,Askeyscheme,lp&awrelation,lp2001,hidden_sym,BannaiIto1984}.

\begin{thm}
[\S 4, \cite{BI2016}]
\label{thm:BImodule_Mn}
Let $n\geq 0$ denote an integer. Then the following hold: 
\begin{enumerate}
\item The space $\M_n$ is a $\BI$-module given by 
\begin{align*}
(X f)(x_1,x_2,x_3) &=
\textstyle(
(-1)^{\frac{1}{2}}e_1\otimes (x_3 T_2
-x_2 T_3))f(x_1,-x_2,-x_3)
\\
&\qquad+\,
\textstyle
\frac{1}{2}f(x_1,-x_2,-x_3)
+
k_2f(x_1,x_2,-x_3)+k_3f(x_1,-x_2,x_3),
\\
(Yf)(x_1,x_2,x_3) &=
\textstyle(
(-1)^{\frac{1}{2}}e_2\otimes (x_1 T_3- x_3 T_1)
) f(-x_1,x_2,-x_3)
\\
&\qquad+\,
\textstyle
\frac{1}{2} f(-x_1,x_2,-x_3)
+
k_3 f(-x_1,x_2,x_3)+k_1 f(x_1,x_2,-x_3),
\\
(Zf)(x_1,x_2,x_3)&=
\textstyle(
(-1)^{\frac{1}{2}}e_3\otimes (x_2 T_1- x_1 T_2)
) f(-x_1,-x_2,x_3)
\\
&\qquad+\,
\textstyle
\frac{1}{2} f(-x_1,-x_2,x_3)
+
k_1f(x_1,-x_2,x_3)+k_2 f(-x_1,x_2,x_3)
\end{align*}
for all $f\in \M_n$. 

\item The elements $\kappa,\lambda,\mu$ act on $\M_n$ as scalar multiplication by 
\begin{align*}
2(k_1k_2+(-1)^n(k_1+k_2+k_3+n+1)k_3),
\\
2(k_2k_3+(-1)^n(k_1+k_2+k_3+n+1)k_1),
\\
2(k_3k_1+(-1)^n(k_1+k_2+k_3+n+1)k_2),
\end{align*}
respectively.
\end{enumerate}
\end{thm}

\noindent In \cite{BI2016}, under the assumption that $k_1,k_2,k_3$ are nonnegative, it was shown that $\dim \M_n=2(n+1)$ and $\M_n$ is isomorphic to a direct sum of two copies of an $(n+1)$-dimensional irreducible $\BI$-module. 
In \cite{HBI2016} it was studied a higher rank version of Theorem \ref{thm:BImodule_Mn}. 
The aim of this paper is to generalize the results of \cite{BI2016} as follows:

Define 
\begin{align*}
t_1&=\left\{
\begin{array}{ll}
-2k_1 \qquad &\hbox{if $2k_1$ is an odd negative integer},
\\
\infty \qquad &\hbox{otherwise};
\end{array}
\right.
\\
t_2&=\left\{
\begin{array}{ll}
-2k_2 \qquad &\hbox{if $2k_2$ is an odd negative integer},
\\
\infty \qquad &\hbox{otherwise};
\end{array}
\right.
\\
t_3&=\left\{
\begin{array}{ll}
-2k_3 \qquad &\hbox{if $2k_3$ is an odd negative integer},
\\
\infty \qquad &\hbox{otherwise}.
\end{array}
\right.
\end{align*}
%We shall show the following statements:

\begin{thm}\label{thm:dimM=2(n+1)}
Let $n\geq 0$ denote an integer. If $n<\max\{t_1,t_2,t_3\}$ or $n+1\geq t_1+t_2+t_3$ then $\dim\M_n=2(n+1)$.
\end{thm}

\noindent For each of the following cases, it follows from Theorem \ref{thm:dimM=2(n+1)} that $\dim\M_n=2(n+1)$:
\begin{enumerate}
\item[(I)] $n<\min\{t_1,t_2,t_3\}$.

\item[(II)] 
%Two of $t_1,t_2,t_3$ is greater than $n$ and the rest one is less than or equal to $n%$. 
$t_2\leq n<\min\{t_1,t_3\}$, $t_3\leq n<\min\{t_1,t_2\}$ or $t_1\leq n<\min\{t_2,t_3\}$.

\item[(III)] 
%One of $t_1,t_2,t_3$ is greater than $n$ and the sum of the other two is less than or equal to $n$.
$t_1+t_3\leq n<t_2$, $t_1+t_2\leq n<t_3$ or $t_2+t_3\leq n<t_1$.

\item[(IV)] $n\geq t_1+t_2+t_3$.
\end{enumerate}
In this paper we explicitly describe the $\BI$-modules $\M_n$ of types  (I)--(IV).

The outline of this paper is as follows. In \S\ref{s:BImodule} we review the preliminaries on the finite-dimensional $\BI$-modules. 
In \S\ref{s:submodule_Mn} we introduce three $\BI$-submodules of $\M_n$ and give some results concerning these $\BI$-modules. 
In \S\ref{s:decMn} we study three decompositions of $\M_n$ relative to the three $\BI$-submodules of $\M_n$ given in \S \ref{s:submodule_Mn}. 
In \S\ref{s:dimMn=2n+2} we give a proof of Theorem \ref{thm:dimM=2(n+1)}.
In \S\ref{s:case1} we represent the structures of the $\BI$-modules $\M_n$ of type (I) (Theorems \ref{thm:n<mint123(odd)} and \ref{thm:n<mint123(even)}).
In \S\ref{s:case2} we represent the structures of the $\BI$-modules $\M_n$ of type (II) (Theorems  \ref{thm:n<mint13(odd)}--\ref{thm:n<mint23(odd)} and  \ref{thm:n<mint13(even)}--\ref{thm:n<mint23(even)}).
In \S\ref{s:case3} we display the structures of the $\BI$-modules $\M_n$ of type (III) (Theorems \ref{thm:t1+t3<=n<t2(odd)}--\ref{thm:t2+t3<=n<t1(odd)} and  \ref{thm:t1+t3<=n<t2(even)}--\ref{thm:t2+t3<=n<t1(even)}). 
In \S\ref{s:case4} we display the structures of the $\BI$-modules $\M_n$ of type (IV) (Theorems \ref{thm:n+1>=t1+t2+t3(odd)} and \ref{thm:n+1>=t1+t2+t3(even)}).

\section{Finite-dimensional $\BI$-modules}\label{s:BImodule}

Let $V$ denote a $\BI$-module. For any algebra automorphism $\e$ of $\BI$, the notation 
$
V^\e
$ 
stands for an alternate $\BI$-module structure on $V$ given by
$$
xv :=\e(x) v 
\qquad \hbox{for all $x\in \BI$ and all $v\in V$}.
$$

Recall that $\{\pm 1\}$ is a group under multiplication and the group $\{\pm 1\}^2$ is isomorphic to the Klein $4$-group. Recall that the symmetric group $\Sym_3$ of degree three is generated by the transpositions $(1\,2)$ and $(2\,3)$. Observe that there exists a unique group homomorphism $\Sym_3\to {\rm Aut}(\{\pm 1\}^2)$ given in the following way:
 
\begin{table}[H]
\centering
\extrarowheight=3pt
\begin{tabular}{c|cc}
$\Sym_3$ &${{\rm Aut}(\{\pm 1\}^2)}$
\\
\midrule[1pt]
\multirow{2}{*}{$(1\,2)$}  &$(1,-1)\mapsto(-1,1)$
\\
&$(-1,1)\mapsto(1,-1)$
\\
\hline
\multirow{2}{*}{$(2\,3)$}  &$(1,-1)\mapsto (1,-1)$
\\
&$(-1,1)\mapsto (-1,-1)$
\end{tabular}
\caption{The homomorphism $\Sym_3\to {\rm Aut}(\{\pm 1\}^2)$}\label{pm1-semidirect}
\end{table}
\noindent Let $\{\pm 1\}^2\rtimes \Sym_3$ denote the semidirect product of $\{\pm 1\}^2$ and $\Sym_3$ with respect to the above homomorphism $\Sym_3\to {\rm Aut}(\{\pm 1\}^2)$. Using Table \ref{pm1-semidirect} yields that there exists a unique $\{\pm 1\}^2\rtimes \Sym_3$-action on $\BI$ such that each element of $\{\pm 1\}^2\rtimes \Sym_3$ acts on $\BI$ as an algebra automorphism in the following way:

\begin{table}[H]
\centering
\extrarowheight=3pt
\begin{tabular}{c|rrrrrr}
$u$  &$X$ &$Y$ &$Z$ 
&$\kappa$ &$\lambda$ &$\mu$
\\

\midrule[1pt]
${(1,-1)}(u)$ &$X$  &$-Y$ &$-Z$ 
&$-\kappa$ &$\lambda$ &$-\mu$
\\
${(-1,1)}(u)$ &$-X$  &$Y$ &$-Z$
&$-\kappa$ &$-\lambda$ &$\mu$
\\
${(1\,2)}(u)$  &$Y$  &$X$ &$Z$
&$\kappa$ &$\mu$ &$\lambda$ 
\\
${(2\,3)}(u)$  &$X$  &$Z$ &$Y$
&$\mu$ &$\lambda$ &$\kappa$
\end{tabular}
\caption{The $\{\pm 1\}^2\rtimes \Sym_3$-action on $\BI$}\label{pm1-action}
\end{table}

The classification of even-dimensional irreducible $\BI$-modules is as follows:

\begin{prop}
[Proposition 2.4, \cite{Huang:BImodule}]
\label{prop:Ed} 
For any scalars $a,b,c\in \C$ and any odd integer $d\geq 1$, there exists a $(d+1)$-dimensional $\BI$-module $E_d(a,b,c)$ satisfying the following conditions: 
\begin{enumerate}
\item There exists a basis for $E_d(a,b,c)$ with respect to which the matrices representing $X$ and $Y$ are 
$$
\begin{pmatrix}
\theta_0 & & &  &{\bf 0}
\\ 
1 &\theta_1 
\\
&1 &\theta_2
 \\
& &\ddots &\ddots
 \\
{\bf 0} & & &1 &\theta_d
\end{pmatrix},
\qquad 
\begin{pmatrix}
\theta_0^* &\varphi_1 &  & &{\bf 0}
\\ 
 &\theta_1^* &\varphi_2
\\
 &  &\theta_2^* &\ddots
 \\
 & & &\ddots &\varphi_d
 \\
{\bf 0}  & & & &\theta_d^*
\end{pmatrix},
$$
respectively, where 
\allowdisplaybreaks
\begin{align*}
\theta_i &=\frac{(-1)^i(2a-d+2i)}{2}  \qquad (0\leq i\leq d),
\\
\theta^*_i &=\frac{(-1)^i(2b-d+2i)}{2} \qquad (0\leq i\leq d),
\\
\varphi_i &=
\left\{
\begin{array}{ll}
\displaystyle{i (d-i+1)}
\qquad &\hbox{if $i$ is even},
\\
\displaystyle{c^2-\frac{(2a+2b-d+2i-1)^2}{4}}
\qquad &\hbox{if $i$ is odd}
\end{array}
\right.
\qquad (1\leq i\leq d).
\end{align*}

\item The elements $\kappa,\lambda,\mu$ act on $E_d(a,b,c)$ as scalar multiplication by
\begin{align*}
c^2-a^2-b^2+\frac{(d+1)^2}{4},
\\
a^2-b^2-c^2+\frac{(d+1)^2}{4},
\\
b^2-c^2-a^2+\frac{(d+1)^2}{4},
\end{align*}
respectively.
\end{enumerate}
\end{prop}

Note that the traces of $X,Y,Z$ on $E_d(a,b,c)$ are equal to $-\frac{d+1}{2}$.

\begin{thm}
[Theorem 4.5, \cite{Huang:BImodule}]
\label{thm:irr_E}
For any scalars $a,b,c\in \C$ and any odd integer $d\geq 1$, the $\BI$-module $E_d(a,b,c)$ is irreducible if and only if
$$
a+b+c,-a+b+c,a-b+c,a+b-c
\not\in 
\textstyle
\{
\frac{d-1}{2}-i\,|\, i=0,2,\ldots,d-1
\}.
$$
\end{thm}

\begin{thm}
[Theorem 5.3, \cite{Huang:BImodule}]
\label{lem:onto2_E}
Let $a,b,c\in \C$ and $d\geq 1$ denote an odd integer. If the $\BI$-module $E_d(a,b,c)$ is irreducible then the $\BI$-module $E_d(a,b,c)$ is isomorphic to 
$E_d(-a,b,c)$, $E_d(a,-b,c)$ and $E_d(a,b,-c)$.
\end{thm}

\begin{thm}
[Theorem 6.3, \cite{Huang:BImodule}]
\label{thm:onto_E}
Let $d\geq 1$ denote an odd integer. Suppose that $V$ is a $(d+1)$-dimensional irreducible $\BI$-module. Then there exist $a,b,c\in \C$ and $\e\in \{\pm 1\}^2$ such that the $\BI$-module $E_d(a,b,c)^\e$ is isomorphic to $V$.
\end{thm}

\begin{thm}
\label{thm:onto2_E}
Let $d\geq 1$ denote an odd integer. Suppose that $V$ is a $(d+1)$-dimensional irreducible $\BI$-module. For any scalars $a,b,c \in \C$ and any $\e\in \{\pm 1\}^{2}$ the following are equivalent:
\begin{enumerate}
\item The $\BI$-module $E_d(a,b,c)^\e$ is isomorphic to $V$.
\item The traces of $X$ and $Y$ on $V^\e$ are equal to $-\frac{d+1}{2}$ and $\kappa+\mu,\lambda+\kappa,\mu+\lambda$ act on $V^\e$ as scalar multiplication by 
\begin{gather*}
-2\left(a+\frac{d+1}{2}\right)\left(a-\frac{d+1}{2}\right),
\\
-2\left(b+\frac{d+1}{2}\right)\left(b-\frac{d+1}{2}\right),
\\
-2\left(c+\frac{d+1}{2}\right)\left(c-\frac{d+1}{2}\right),
\end{gather*} 
respectively.
\end{enumerate}
\end{thm}
\begin{proof}
Immediate from Proposition \ref{prop:Ed} and Theorems \ref{lem:onto2_E}, \ref{thm:onto_E}.
\end{proof}

The classification of odd-dimensional irreducible $\BI$-modules is as follows:

\begin{prop}
[Proposition 2.6, \cite{Huang:BImodule}]
\label{prop:Od} 
For any scalars $a,b,c\in \C$ and any even integer $d\geq 0$,  there exists a $(d+1)$-dimensional $\BI$-module $O_d(a,b,c)$ satisfying the following conditions:
\begin{enumerate}
\item There exists a basis for $O_d(a,b,c)$ with respect to which the matrices representing $X$ and $Y$ are 
$$
\begin{pmatrix}
\theta_0 & & &  &{\bf 0}
\\ 
1 &\theta_1 
\\
&1 &\theta_2
 \\
& &\ddots &\ddots
 \\
{\bf 0} & & &1 &\theta_d
\end{pmatrix},
\qquad 
\begin{pmatrix}
\theta_0^* &\varphi_1 &  & &{\bf 0}
\\ 
 &\theta_1^* &\varphi_2
\\
 &  &\theta_2^* &\ddots
 \\
 & & &\ddots &\varphi_d
 \\
{\bf 0}  & & & &\theta_d^*
\end{pmatrix},
$$
respectively, where 
\begin{align*}
\theta_i &=\frac{(-1)^i(2a-d+2i)}{2}  \qquad (0\leq i\leq d),
\\
\theta^*_i &=\frac{(-1)^i(2b-d+2i)}{2} \qquad (0\leq i\leq d),
\\
\varphi_i &=
\left\{
\begin{array}{ll}
\displaystyle{\frac{i(d+1-2i-2a-2b-2c)}{2}}
\qquad &\hbox{if $i$ is even},
\\
\displaystyle{\frac{(i-d-1)(d+1-2i-2a-2b+2c)}{2}}
\qquad &\hbox{if $i$ is odd}
\end{array}
\right.
\qquad (1\leq i\leq d).
\end{align*}

\item The elements $\kappa,\lambda,\mu$ act on $O_d(a,b,c)$ as scalar multiplication by
\allowdisplaybreaks
\begin{align*}
2ab-c(d+1),
\\
2bc-a(d+1),
\\
2ca-b(d+1),
\end{align*}
respectively.
\end{enumerate}
\end{prop}

Note that the traces of $X,Y,Z$ on $O_d(a,b,c)$ are equal to $a,b,c$ respectively.

\begin{thm}
[Theorem 2.8, \cite{Huang:BImodule}]
\label{thm:irr_O}
For any $a, b, c\in \C$ and any even integer $d\geq 0$ the $\BI$-module $O_d(a,b,c)$ is irreducible if and only if 
$$
a+b+c, a-b-c, -a+b-c, -a-b+c\not\in 
\{
\textstyle
\frac{d+1}{2}-i \,|\, i =2,4,\ldots,d
\}.
$$
\end{thm}

\begin{thm}
[Theorem 2.8, \cite{Huang:BImodule}]
\label{thm:onto_O}
Let $d\geq 0$ denote an even integer.  Suppose that $V$ is a $(d+1)$-dimensional irreducible $\BI$-module. Then there exist unique $a,b,c\in \C$ such that the $\BI$-module $O_d(a,b,c)$ is isomorphic to $V$.
\end{thm}

\begin{thm}
%[Theorem 6.4, \cite{Huang:additiveDAHA&LT}]
\label{thm:onto2_O}
Let $d\geq 0$ denote an even integer. Suppose that $V$ is a $(d+1)$-dimensional irreducible $\BI$-module. For any scalars $a,b,c\in \C$ the following are equivalent:
\begin{enumerate}
\item The $\BI$-module $O_d(a,b,c)$ is isomorphic to $V$. 

\item The traces of $X,Y,Z$ on $V$ are equal to $a,b,c$ respectively. 
\end{enumerate}
\end{thm}
\begin{proof}
Immediate from Proposition \ref{prop:Od} and Theorem \ref{thm:onto_O}.
\end{proof}

\section{Three $\BI$-submodules of $\M_n$}
\label{s:submodule_Mn}

In this section we consider three subspaces of $\M_n$ given as follows:

\begin{defn}\label{defn:M(x1)}
Let $n\geq 0$ denote an integer. Define 
\begin{align*}
\M_n(x_1)&=
\M_n\cap 
\bigoplus_{i=t_1}^n 
 \C^2\otimes (x_1^i\cdot \R[x_2,x_3]_{n-i}),
\\
\M_n(x_2)&=
\M_n\cap 
\bigoplus_{i=t_2}^n 
 \C^2\otimes (x_2^i\cdot \R[x_1,x_3]_{n-i}),
\\
\M_n(x_3)&=
\M_n\cap 
\bigoplus_{i=t_3}^n 
\C^2\otimes (x_3^i\cdot \R[x_1,x_2]_{n-i}).
\end{align*}
The space $\M_n(x_1)$ (resp. $\M_n(x_2)$) (resp. $\M_n(x_3)$) is interpreted as the zero subspace of $\M_n$ when $n<t_1$ (resp. $n<t_2$) (resp. $n<t_3$). Observe that 
each of $\M_n(x_1), \M_n(x_2), \M_n(x_3)$ is a $\BI$-submodule of $\M_n$. 
\end{defn}

For convenience we use the following notations. Let 
\begin{align*}
\D(x_1) &=e_2\otimes T_2+e_3\otimes T_3,
\\
\D(x_2) &=e_3\otimes T_3+e_1\otimes T_1,
\\
\D(x_3) &=e_1\otimes T_1+e_2\otimes T_2.
\end{align*}
For any integer $n\geq 0$ we set 
\begin{align*}
m^{(n)}_1
&=
n+(1-(-1)^n)k_1,
\\
m^{(n)}_2
&=
n+(1-(-1)^n)k_2,
\\
m^{(n)}_3
&=
n+(1-(-1)^n)k_3.
\end{align*}
Note that for any integer $n\geq 0$, 
\begin{align*}
T_1(x_1^n)
&=
m^{(n)}_1 x_1^{n-1},
\\
T_2(x_2^n)
&=
m^{(n)}_2 x_2^{n-1},
\\
T_3(x_3^n)
&=
m^{(n)}_3 x_3^{n-1}.
\end{align*}
For any integers $i,j$ with $0\leq i\leq j$ we set 
\begin{align*}
[x_1]^j_i &=
\prod_{h=i+1}^j m_1^{(h)} 
x_1^{i},
\\
[x_2]^j_i &=
\prod_{h=i+1}^j m_2^{(h)} 
x_2^{i},
\\
[x_3]^j_i &=
\prod_{h=i+1}^j m_3^{(h)} 
x_3^{i}.
\end{align*}

\begin{lem}\label{lem:dimM1}
Let $n\geq 0$ denote an integer. Then the following hold: 
\begin{enumerate}
\item Suppose that $n\geq t_1$. Then given any integer $k$ with $t_1\leq k\leq n$ and any $f_{n-i}\in \C^2\otimes \R[x_2,x_3]_{n-i}$ for $i=t_1,t_1+1,\ldots,k$ the polynomial 
\begin{gather}\label{e:dimM1}
\sum_{i=t_1}^k 
(-1)^{\lceil \frac{i}{2}\rceil}
e_1^i
\otimes
[x_1]^k_i(f_{n-i})\in \M_n
\end{gather}
if and only if $f_{n-i}=\D(x_1)^{i-t_1} (f_{n-t_1})$ for all $i=t_1,t_1+1,\ldots,k$.

\item Suppose that $n\geq t_2$. Then given any integer $k$ with $t_2\leq k\leq n$ and any $f_{n-i}\in \C^2\otimes \R[x_1,x_3]_{n-i}$ for all $i=t_2,t_2+1,\ldots,k$  the polynomial 
$$
\sum_{i=t_2}^k
(-1)^{\lceil \frac{i}{2}\rceil}
e_2^i
\otimes
[x_2]^k_i(f_{n-i})\in \M_n 
$$
if and only if $f_{n-i}=\D(x_2)^{i-t_2} (f_{n-t_2})$ for all $i=t_2,t_2+1,\ldots,k$.

\item Suppose that $n\geq t_3$. Then given any integer $k$ with $t_3\leq k\leq n$ and any $f_{n-i}\in \C^2\otimes \R[x_1,x_2]_{n-i}$ for all $i=t_3,t_3+1,\ldots,k$ the polynomial   
$$
\sum_{i=t_3}^k 
(-1)^{\lceil \frac{i}{2}\rceil}
e_3^i
\otimes
[x_3]^k_i(f_{n-i})\in \M_n
$$
if and only if $f_{n-i}=\D(x_3)^{i-t_3} (f_{n-t_3})$ for all $i=t_3,t_3+1,\ldots,k$.
\end{enumerate}
\end{lem}
\begin{proof}
(i): Let  $f$ denote the left-hand side of (\ref{e:dimM1}).
Using $\D=\D(x_1)+e_1\otimes T_1$ yields that  
\begin{align*}
\D(f)&=\sum_{i=t_1}^{k-1}
(-1)^{\lceil \frac{3i}{2}\rceil}
e_1^i
\otimes
[x_1]^k_i
(\D(x_1)(f_{n-i}))
+
\sum_{i=t_1+1}^k
(-1)^{\lceil \frac{i}{2}\rceil}
e_1^{i+1} 
\otimes
[x_1]^k_{i-1}
(f_{n-i})
\\
&=
\sum_{i=t_1+1}^k
(-1)^{\lceil \frac{3i-3}{2}\rceil}
e_1^{i-1}
\otimes
[x_1]^k_{i-1}
(
\D(x_1)(f_{n-i+1})
)
+
\sum_{i=t_1+1}^k
(-1)^{\lceil \frac{i}{2}\rceil}
e_1^{i+1} 
\otimes 
[x_1]^k_{i-1}
(f_{n-i}).
\end{align*}
Since $e_1^2=1$ and $(-1)^{\lceil \frac{3i-3}{2}\rceil}=(-1)^{\lceil \frac{i}{2}\rceil+1}$ for all integers $i$, it follows that 
$$
\D(f)=
\sum_{i=t_1+1}^k
(-1)^{\lceil \frac{i}{2}\rceil}
e_1^{i+1}
\otimes 
[x_1]^k_{i-1} (f_{n-i}-\D(x_1)(f_{n-i+1})).
$$
Since $k\geq t_1$ the term $[x_1]^k_i\not=0$ for all $i=t_1,t_1+1,\ldots,k$. Hence $f\in \M_n$ if and only if $f_{n-i}=\D(x_1)(f_{n-i+1})$ for all $i=t_1+1,t_1+2,\ldots,k$. The statement (i) follows.

(ii), (iii): By similar arguments the statements (ii), (iii) hold.
\end{proof}

\begin{prop}\label{prop:dimM(x1)}
Let $n\geq 0$ denote an integer. Then the following hold:
\begin{enumerate}
\item $
\dim \M_n(x_1)
=
\left\{
\begin{array}{ll}
0
\qquad 
&\hbox{if $n< t_1$},
\\
2(n-t_1+1)
\qquad 
&\hbox{if $n\geq t_1$}.
\end{array}
\right.
$ 

\item $
\dim \M_n(x_2)
=
\left\{
\begin{array}{ll}
0
\qquad 
&\hbox{if $n< t_2$},
\\
2(n-t_2+1)
\qquad 
&\hbox{if $n\geq t_2$}.
\end{array}
\right.
$  

\item $
\dim \M_n(x_3)
=
\left\{
\begin{array}{ll}
0
\qquad 
&\hbox{if $n< t_3$},
\\
2(n-t_3+1)
\qquad 
&\hbox{if $n\geq t_3$}.
\end{array}
\right.
$ 
\end{enumerate}
\end{prop}
\begin{proof}
(i): By Definition \ref{defn:M(x1)} the statement (i) is true for $n<t_1$. Suppose that $n\geq t_1$. 
By Lemma \ref{lem:dimM1}(i) there exists a linear isomorphism $\C^2\otimes \R[x_2,x_3]_{n-t_1}\to \M_n(x_1)$ that sends $f$ to 
\begin{eqnarray*}
f &\mapsto &
\sum_{i=t_1}^n
(-1)^{\lceil \frac{i}{2}\rceil}
e_1^i
\otimes
[x_1]^{n}_{i}
( 
\D(x_1)^{i-t_1} (f))
\end{eqnarray*} 
for all $f\in \C^2\otimes \R[x_2,x_3]_{n-t_1}$. Therefore (i) follows.

(ii), (iii): By similar arguments the statements (ii), (iii) hold.
\end{proof}

\begin{prop}\label{prop:dimM(x12)}
Let $n\geq 0$ denote an integer. Then the following hold:
\begin{enumerate}
\item $
\dim \M_n(x_1)\cap \M_n(x_2)=
\left\{
\begin{array}{ll}
0
\qquad 
&\hbox{if $n<t_1+t_2$},
\\
2(n-t_1-t_2+1)
\qquad 
&\hbox{if $n\geq t_1+t_2$}.
\end{array}
\right.
$

\item $
\dim \M_n(x_2)\cap \M_n(x_3)=
\left\{
\begin{array}{ll}
0
\qquad 
&\hbox{if $n<t_2+t_3$},
\\
2(n-t_2-t_3+1)
\qquad 
&\hbox{if $n\geq t_2+t_3$}.
\end{array}
\right.
$

\item $
\dim \M_n(x_3)\cap \M_n(x_1)=
\left\{
\begin{array}{ll}
0
\qquad 
&\hbox{if $n<t_3+t_1$},
\\
2(n-t_3-t_1+1)
\qquad 
&\hbox{if $n\geq t_3+t_1$}.
\end{array}
\right.
$
\end{enumerate}
\end{prop}
\begin{proof}
(i): By Definition \ref{defn:M(x1)} the statement (i) is true for $n<t_1+t_2$. Suppose that $n\geq t_1+t_2$. Let 
$
f\in
\bigoplus_{i=t_1}^n\C^2\otimes (x_1^i\cdot \R[x_2,x_3]_{n-i})
\cap 
\bigoplus_{i=t_2}^n\C^2\otimes (x_2^i\cdot \R[x_1,x_3]_{n-i})$. 
Observe that there are unique $f_{n-t_2-i}\in \C^2\otimes \R[x_2,x_3]_{n-t_2-i}$ for all $i=t_1,t_1+1,\ldots,n-t_2$ such that 
$$
f=\sum_{i=t_1}^{n-t_2}
(-1)^{\lceil \frac{i}{2}\rceil}
e_1^i \otimes [x_1]_i^{n-t_2}x_2^{t_2} (f_{n-t_2-i}). 
$$
Using Lemma \ref{lem:dimM1}(i) yields that $f\in \M_n$ if and only if $f_{n-t_2-i}=\D(x_1)^{i-t_1}(f_{n-t_1-t_2})$ for all $i=t_1,t_1+1,\ldots,n-t_2$. 
By the above comment there exists a linear isomorphism $\C^2\otimes \R[x_2,x_3]_{n-t_1-t_2}\to \M_n(x_1)\cap \M_n(x_2)$ that sends 
\begin{eqnarray*}
f &\mapsto& 
\sum_{i=t_1}^{n-t_2}
(-1)^{\lceil \frac{i}{2}\rceil}
e_1^i \otimes [x_1]_i^{n-t_2}x_2^{t_2} (\D(x_1)^{i-t_1}(f)) 
\end{eqnarray*}
for all $f\in \C^2\otimes \R[x_2,x_3]_{n-t_1-t_2}$. Therefore (i) follows.

(ii), (iii): By similar arguments the statements (ii), (iii) hold.
\end{proof}

\begin{prop}\label{prop:dimM(x123)}
Let $n\geq 0$ denote an integer. Then 
$$
\dim \M_n(x_1)\cap \M_n(x_2)\cap \M_n(x_3)=
\left\{
\begin{array}{ll}
0
\qquad 
&\hbox{if $n<t_1+t_2+t_3$},
\\
2 (n-t_1-t_2-t_3+1)
\qquad 
&\hbox{if $n\geq t_1+t_2+t_3$}.
\end{array}
\right.
$$
\end{prop}
\begin{proof}
By Definition \ref{defn:M(x1)} the statement is true for $n<t_1+t_2+t_3$. Suppose that $n\geq t_1+t_2+t_3$. Let $f$ lie in $
\bigoplus_{i=t_1}^n\C^2\otimes (x_1^i\cdot \R[x_2,x_3]_{n-i})
\cap
\bigoplus_{i=t_2}^n\C^2\otimes (x_2^i\cdot \R[x_1,x_3]_{n-i})
\cap
\bigoplus_{i=t_3}^n\C^2\otimes (x_3^i\cdot \R[x_1,x_2]_{n-i})$. 
Observe that there are unique $f_{n-t_2-t_3-i}\in \C^2\otimes \R[x_2,x_3]_{n-t_2-t_3-i}$ for all $i=t_1,t_1+1,\ldots,n-t_2-t_3$ such that 
$$
f=\sum_{i=t_1}^{n-t_2-t_3}
(-1)^{\lceil \frac{i}{2}\rceil}
e_1^i \otimes [x_1]_i^{n-t_2-t_3} x_2^{t_2} x_3^{t_3} (f_{n-t_2-t_3-i}). 
$$
Using Lemma \ref{lem:dimM1}(i) yields that $f\in \M_n$ if and only if $f_{n-t_2-t_3-i}=\D(x_1)^{i-t_1}(f_{n-t_1-t_2-t_3})$ for all $i=t_1,t_1+1,\ldots,n-t_2-t_3$. 
By the above comment there exists a linear isomorphism $\C^2\otimes \R[x_2,x_3]_{n-t_1-t_2-t_3}\to \M_n(x_1)\cap \M_n(x_2)\cap \M_n(x_3)$ that sends 
\begin{eqnarray*}
f &\mapsto& 
\sum_{i=t_1}^{n-t_2-t_3}
(-1)^{\lceil \frac{i}{2}\rceil}
e_1^i \otimes [x_1]_i^{n-t_2-t_3} x_2^{t_2} x_3^{t_3} (\D(x_1)^{i-t_1}(f)) 
\end{eqnarray*}
for all $f\in \C^2\otimes \R[x_2,x_3]_{n-t_1-t_2-t_3}$. The proposition follows.
\end{proof}

\section{Three decompositions of $\M_n$}\label{s:decMn}

In this section we study the decompositions of $\M_n$ relative to $\M_n(x_1),\M_n(x_2),\M_n(x_3)$.

\begin{defn}\label{defn:N123}
Let $n\geq 0$ denote an integer. Define 
\begin{align*}
\N_n(x_1)&=
\M_n\cap 
\bigoplus_{i=0}^{\min\{n,t_1-1\}}
 \C^2\otimes (x_1^i\cdot \R[x_2,x_3]_{n-i}),
\\
\N_n(x_2)&=
\M_n\cap 
\bigoplus_{i=0}^{\min\{n,t_2-1\}}
 \C^2\otimes (x_2^i\cdot \R[x_1,x_3]_{n-i}),
\\
\N_n(x_3)&=
\M_n\cap 
\bigoplus_{i=0}^{\min\{n,t_3-1\}}
\C^2\otimes (x_3^i\cdot \R[x_1,x_2]_{n-i}).
\end{align*}
\end{defn}

\begin{thm}\label{thm:decM}
Let $n\geq 0$ denote an integer. Then the following hold:
\begin{enumerate}
\item $\M_n=\M_n(x_1)\oplus\N_n(x_1)$.

\item $\M_n=\M_n(x_2)\oplus\N_n(x_2)$.

\item $\M_n=\M_n(x_3)\oplus\N_n(x_3)$.
\end{enumerate}
\end{thm}
\begin{proof}
(i): Suppose that $n<t_1$. Then $\M_n(x_1)=\{0\}$ and $\N_n(x_1)=\M_n$ by Definitions \ref{defn:M(x1)} and \ref{defn:N123}. The statement (i) holds for $n<t_1$. 

Suppose that $n\geq t_1$. %Observe that 
%$
%\M_n(x_1)\cap\N_n(x_1)=\{0\}.
%$
%We now show that 
%$\M_n= \M_n(x_1)+\N_n(x_1)$. 
Let $f\in \M_n$ be given. Since $f\in \C^2\otimes \R[x_1,x_2,x_3]_n$ there are unique 
$$
f_1\in \bigoplus\limits_{i=t_1}^n \C^2\otimes (x_1^i\cdot \R[x_2,x_3]_{n-i}),\qquad 
f_2\in \bigoplus\limits_{i=0}^{t_1-1} \C^2\otimes (x_1^i\cdot \R[x_2,x_3]_{n-i})
$$ such that 
$
f=f_1+f_2.
$
Since $\D(f)=0$ it follows that 
$\D(f_1)=-\D(f_2)$. 
Since 
$
\D(f_1)\in \bigoplus\limits_{i=t_1}^{n-1} \C^2\otimes (x_1^i\cdot \R[x_2,x_3]_{n-i-1})
$
and 
$
\D(f_2)\in
\bigoplus\limits_{i=0}^{t_1-2} \C^2\otimes (x_1^i\cdot \R[x_2,x_3]_{n-i-1})$, 
it follows that $\D(f_1)=\D(f_2)=0$. Hence $f_1\in \M_n(x_1)$ and $f_2\in \N_n(x_1)$. The statement (i) holds for $n\geq t_1$.

(ii), (iii): By similar arguments (ii) and (iii) follow.
\end{proof}

\begin{lem}\label{lem1:dimN1}
Let $n\geq 0$ denote an integer. Then the following hold:
\begin{enumerate}
\item Suppose that $n< t_1$. Then given any $f_{n-i}\in \C^2\otimes \R[x_2,x_3]_{n-i}$ for all $i=0,1,\ldots, n$ the polynomial 
\begin{gather}\label{e:dimN1}
\sum_{i=0}^n (-1)^{\lceil \frac{i}{2}\rceil}
e_1^i\otimes
[x_1]^n_{i}
(f_{n-i})\in \M_n
\end{gather}
if and only if $f_{n-i}=\D(x_1)^i (f_n)$ for all $i=0,1,\ldots,n$.

\item Suppose that $n< t_2$. Then given any $f_{n-i}\in \C^2\otimes \R[x_1,x_3]_{n-i}$ for all $i=0,1,\ldots, n$ the polynomial 
$$
\sum_{i=0}^n (-1)^{\lceil \frac{i}{2}\rceil}
e_2^i\otimes
[x_2]^n_{i}
(f_{n-i})\in \M_n
$$
if and only if $f_{n-i}=\D(x_2)^i (f_n)$ for all $i=0,1,\ldots,n$.

\item Suppose that $n< t_3$. Then given any $f_{n-i}\in \C^2\otimes \R[x_1,x_2]_{n-i}$ for all $i=0,1,\ldots, n$ the polynomial 
$$
\sum_{i=0}^n (-1)^{\lceil \frac{i}{2}\rceil}
e_3^i\otimes
[x_3]^n_{i}
(f_{n-i})\in \M_n
$$
if and only if $f_{n-i}=\D(x_3)^i (f_n)$ for all $i=0,1,\ldots,n$.
\end{enumerate}
\end{lem}
\begin{proof}
(i): Let  $f$ denote the left-hand side of (\ref{e:dimN1}).
Using $\D=\D(x_1)+e_1\otimes T_1$ yields that  
\begin{align*}
\D(f)&=\sum_{i=0}^{n-1}
(-1)^{\lceil \frac{3i}{2}\rceil}
e_1^i
\otimes
[x_1]^n_i
(\D(x_1)(f_{n-i}))
+
\sum_{i=1}^n
(-1)^{\lceil \frac{i}{2}\rceil}
e_1^{i+1} 
\otimes
[x_1]^n_{i-1}
(f_{n-i})
\\
&=
\sum_{i=1}^n
(-1)^{\lceil \frac{3i-3}{2}\rceil}
e_1^{i-1}
\otimes
[x_1]^n_{i-1}
(
\D(x_1)(f_{n-i+1})
)
+
\sum_{i=1}^n
(-1)^{\lceil \frac{i}{2}\rceil}
e_1^{i+1} 
\otimes 
[x_1]^n_{i-1}
(f_{n-i}).
\end{align*}
Since $e_1^2=1$ and $(-1)^{\lceil \frac{3i-3}{2}\rceil}=(-1)^{\lceil \frac{i}{2}\rceil+1}$ for all integers $i$, it follows that 
$$
\D(f)=
\sum_{i=1}^n
(-1)^{\lceil \frac{i}{2}\rceil}
e_1^{i+1}
\otimes 
[x_1]^n_{i-1} (f_{n-i}-\D(x_1)(f_{n-i+1})).
$$
Since $n< t_1$ the term $[x_1]^n_{i-1}\not=0$ for all $i=1,2,\ldots,n$. Hence $f\in \M_n$ if and only if $f_{n-i}=\D(x_1)(f_{n-i+1})$ for all $i=1,2,\ldots,n$. The statement (i) follows.

(ii), (iii): By similar arguments the statements (ii), (iii) hold.
\end{proof}

\begin{lem}\label{lem2:dimN1}
Let $n\geq 0$ denote an integer. Then the following hold:
\begin{enumerate}
\item Suppose that $n\geq t_1$. Then given any $f_{n-i}\in \C^2\otimes \R[x_2,x_3]_{n-i}$ for all $i=0,1,\ldots, t_1-1$ the polynomial 
\begin{gather}\label{e2:dimN1}
\sum_{i=0}^{t_1-1} (-1)^{\lceil \frac{i}{2}\rceil}
e_1^i\otimes
[x_1]^{t_1-1}_{i}
(f_{n-i})\in \M_n
\end{gather}
if and only if $f_{n-i}=\D(x_1)^i (f_n)$ for all $i=0,1,\ldots,t_1-1$ and $\D(x_1)^{t_1}(f_n)=0$.

\item Suppose that $n\geq t_2$. Then given any $f_{n-i}\in \C^2\otimes \R[x_1,x_3]_{n-i}$ for all $i=0,1,\ldots, t_2-1$ the polynomial 
$$
\sum_{i=0}^{t_2-1} (-1)^{\lceil \frac{i}{2}\rceil}
e_2^i\otimes
[x_2]^{t_2-1}_{i}
(f_{n-i})\in \M_n
$$
if and only if $f_{n-i}=\D(x_2)^i (f_n)$ for all $i=0,1,\ldots,t_2-1$ and $\D(x_2)^{t_2}(f_n)=0$.

\item Suppose that $n\geq t_3$. Then given any $f_{n-i}\in \C^2\otimes \R[x_1,x_2]_{n-i}$ for all $i=0,1,\ldots, t_3-1$ the polynomial 
$$
\sum_{i=0}^{t_3-1} (-1)^{\lceil \frac{i}{2}\rceil}
e_3^i\otimes
[x_3]^{t_3-1}_{i}
(f_{n-i})\in \M_n
$$
if and only if $f_{n-i}=\D(x_3)^i (f_n)$ for all $i=0,1,\ldots,t_3-1$ and $\D(x_3)^{t_3}(f_n)=0$. 
\end{enumerate}
\end{lem}
\begin{proof}
(i): Let  $f$ denote the left-hand side of (\ref{e2:dimN1}).
Using $\D=\D(x_1)+e_1\otimes T_1$ yields that  
\begin{align*}
\D(f)&=\sum_{i=0}^{t_1-1}
(-1)^{\lceil \frac{3i}{2}\rceil}
e_1^i
\otimes
[x_1]^{t_1-1}_i
(\D(x_1)(f_{n-i}))
+
\sum_{i=1}^{t_1-1}
(-1)^{\lceil \frac{i}{2}\rceil}
e_1^{i+1} 
\otimes
[x_1]^{t_1-1}_{i-1}
(f_{n-i})
\\
&=
\sum_{i=1}^{t_1}
(-1)^{\lceil \frac{3i-3}{2}\rceil}
e_1^{i-1}
\otimes
[x_1]^{t_1-1}_{i-1}
(
\D(x_1)(f_{n-i+1})
)
+
\sum_{i=1}^{t_1-1}
(-1)^{\lceil \frac{i}{2}\rceil}
e_1^{i+1} 
\otimes 
[x_1]^{t_1-1}_{i-1}
(f_{n-i}).
\end{align*}
Since $e_1^2=1$ and $(-1)^{\lceil \frac{3i-3}{2}\rceil}=(-1)^{\lceil \frac{i}{2}\rceil+1}$ for all integers $i$, it follows that $\D(f)$ is equal to 
$$
(-1)^{\frac{t_1-1}{2}} 
(1\otimes x_1^{t_1-1})
(\D(x_1)(f_{n-t_1+1}))
+
\sum_{i=1}^{t_1-1}
(-1)^{\lceil \frac{i}{2}\rceil}
e_1^{i+1}
\otimes 
[x_1]^{t_1-1}_{i-1} (f_{n-i}-\D(x_1)(f_{n-i+1})).
$$
Since $[x_1]^{t_1-1}_{i-1}\not=0$ for all $i=1,2,\ldots,t_1-1$. Hence $f\in \M_n$ if and only if $f_{n-i}=\D(x_1)(f_{n-i+1})$ for all $i=1,2,\ldots,t_1-1$ and $\D(x_1)(f_{n-t_1+1})=0$. The statement (i) follows.

(ii), (iii): By similar arguments the statements (ii), (iii) hold.
\end{proof}

\begin{prop}\label{prop:dimN(x1)}
Let $n\geq 0$ denote an integer. Then the following hold:
\begin{enumerate}
\item 
$
\dim \N_n(x_1)=
\left\{
\begin{array}{ll}
2(n+1)
\qquad 
&\hbox{if $n< t_1$},
\\
{\rm nullity} 
\left(\D(x_1)^{t_1}
\big|_{\C^2\otimes \R[x_2,x_3]_n}
\right)
\qquad 
&\hbox{if $n\geq t_1$}.
\end{array}
\right.
$ 

\item 
$
\dim \N_n(x_2)=
\left\{
\begin{array}{ll}
2(n+1)
\qquad 
&\hbox{if $n< t_2$},
\\
{\rm nullity} 
\left(\D(x_2)^{t_2}
\big|_{\C^2\otimes \R[x_1,x_3]_n}
\right)
\qquad 
&\hbox{if $n\geq t_2$}.
\end{array}
\right.
$ 

\item 
$
\dim \N_n(x_3)=
\left\{
\begin{array}{ll}
2(n+1)
\qquad 
&\hbox{if $n< t_3$},
\\
{\rm nullity} 
\left(\D(x_3)^{t_3}
\big|_{\C^2\otimes \R[x_1,x_2]_n}
\right)
\qquad 
&\hbox{if $n\geq t_3$}.
\end{array}
\right.
$ 
\end{enumerate}
\end{prop}
\begin{proof}
(i): Suppose that $n<t_1$. By Lemma \ref{lem1:dimN1}(i) there is a linear isomorphism $\C^2\otimes \R[x_2,x_3]_n \to \N_n(x_1)$ that sends 
\begin{eqnarray*}
f &\mapsto & 
\sum_{i=0}^n
(-1)^{\lceil \frac{i}{2}\rceil}
e_1^i
\otimes 
[x_1]^{n}_{i}
(\D(x_1)^i (f))
\end{eqnarray*}
for all $f\in \C^2\otimes \R[x_2,x_3]_n$. The statement (i) holds for $n<t_1$. Suppose that $n\geq t_1$. 
By Lemma \ref{lem2:dimN1}(i) there is a linear isomorphism $\ker\left(\D(x_1)^{t_1}\big|_{\C^2\otimes \R[x_2,x_3]_n}\right)\to\N_n(x_1)$ that sends 
\begin{eqnarray*}
f 
&\mapsto &
\sum_{i=0}^{t_1-1}
(-1)^{\lceil \frac{i}{2}\rceil}
e_1^i
\otimes
[x_1]^{t_1-1}_{i}
(\D(x_1)^i (f))
\end{eqnarray*}
for all $f\in \ker\left(\D(x_1)^{t_1}\big|_{\C^2\otimes \R[x_2,x_3]_n}\right)$. The statement (i) holds for $n\geq t_1$.

(ii), (iii): By similar arguments (ii) and (iii) follow.
\end{proof}

\begin{cor}\label{cor:dimMn}
Let $n\geq 0$ denote an integer. Then the following hold:
\begin{enumerate}
\item $
\dim \M_n
=
\left\{
\begin{array}{ll}
2(n+1)
\qquad 
&\hbox{if $n< t_1$},
\\
2(n-t_1+1)
+
{\rm nullity}
\left(
\D(x_1)^{t_1}|_{\C^2\otimes \R[x_2,x_3]_n}
\right)
\qquad 
&\hbox{if $n\geq t_1$}.
\end{array}
\right.
$ 

\item $
\dim \M_n
=
\left\{
\begin{array}{ll}
2(n+1)
\qquad 
&\hbox{if $n< t_2$},
\\
2(n-t_2+1)
+
{\rm nullity}
\left(
\D(x_2)^{t_2}|_{\C^2\otimes \R[x_1,x_3]_n}
\right)
\qquad 
&\hbox{if $n\geq t_2$}.
\end{array}
\right.
$ 

\item $
\dim \M_n
=
\left\{
\begin{array}{ll}
2(n+1)
\qquad 
&\hbox{if $n< t_3$},
\\
2(n-t_3+1)
+
{\rm nullity}
\left(
\D(x_3)^{t_3}|_{\C^2\otimes \R[x_1,x_2]_n}
\right)
\qquad 
&\hbox{if $n\geq t_3$}.
\end{array}
\right.
$ 
\end{enumerate}
\end{cor}
\begin{proof}
Combine Propositions \ref{prop:dimM(x1)} and \ref{prop:dimN(x1)} together with Theorem \ref{thm:decM}.
\end{proof}

\section{Proof of Theorem \ref{thm:dimM=2(n+1)}}
\label{s:dimMn=2n+2}

Let $n\geq 1$ denote an integer. Let $I_n$ denote the $n\times n$ identity matrix. Set $A_n$ to be the $n\times n$ diagonal matrix whose $(i,i)$-entry is $(-1)^{n-i}$ for all $i=1,2,\ldots,n$. 
%with
%$$
%(A_n)_{ii}=(-1)^{n-i} 
%\qquad 
%(1\leq i\leq n).
%$$
Define $N_n$ to be the $n\times (n+1)$ matrix with 
$$
(N_n)_{ij}
=\left\{
\begin{array}{ll}
(-1)^{n-i+\frac{1}{2}} m_2^{(n+1-i)}
\qquad 
&\hbox{if $i=j$},
\\
(-1)^{n-i} m_1^{(i)}
\qquad 
&\hbox{if $i=j-1$},
\\
0
\qquad 
&\hbox{else}
\end{array}
\right.
%\qquad 
%(1\leq i\leq n, 1\leq j\leq n+1).
$$
for all integers $i,j$ with $1\leq i\leq n$ and $1\leq j\leq n+1$.

\begin{lem}\label{lem:rankN}
Let $n\geq 1$ denote an integer. Then 
$$
\rank(N_n)
=\left\{
\begin{array}{ll}
n-1 \qquad &\hbox{if $\max\{t_1,t_2\}\leq n<t_1+t_2$},
\\
n \qquad &\hbox{else}.
\end{array}
\right.
$$
\end{lem}
\begin{proof}
If $n<\max\{t_1,t_2\}$ then at most one zero entry lies on the diagonal and superdiagonal of $N_n$. Hence $\rank(N_n)=n$ in this case.

Suppose that $n\geq \max\{t_1,t_2\}$. Since the $(n-t_2+1,n-t_2+1)$-entry of $N_n$ is zero, the matrix $N_n$ is of the form 
$$
\kbordermatrix{
    \mbox{} &n- t_2+1  & & t_2 \\
     n-t_2 &U &\rvline & {\bf 0} 
     \\
     \cline{2-4}
    t_2 &{\bf 0} &\rvline &L
  }.
$$
If $n<t_1+t_2$ then $\rank(U)=n-t_2$ and $\rank(L)=t_2-1$. If $n\geq t_1+t_2$ then $\rank(U)=n-t_2$ and $\rank(L)=t_2$. The lemma follows.
\end{proof}

For any integer $n\geq 0$
let $\alpha_n$ stand for the basis 
$$
\left\{
\begin{pmatrix}
1
\\
0
\end{pmatrix}
\otimes 
x_1^i x_2^{n-i}
\right\}_{i=0}^n
\cup 
\left\{
\begin{pmatrix}
0
\\
1
\end{pmatrix}
\otimes 
x_1^i x_2^{n-i}
\right\}_{i=0}^n
$$
for $\C^2\otimes \R[x_1,x_2]_n$.

\begin{lem}\label{lem:D3}
Let  $n\geq 1$ denote an integer. 
Then the matrix representing 
$$
\D(x_3)
\big|_{\C^2\otimes \R[x_1,x_2]_n}
:
\C^2\otimes \R[x_1,x_2]_n
\to 
\C^2\otimes \R[x_1,x_2]_{n-1}
$$ 
with respect to the basis 
$
\alpha_n
$
for $\C^2\otimes \R[x_1,x_2]_n$ and the basis $\alpha_{n-1}$ for $\C^2\otimes \R[x_1,x_2]_{n-1}$ is 
\begin{gather*}
\begin{pmatrix}
I_n &\rvline &{\bf 0}
\\
\hline
{\bf 0} &\rvline & A_n
\end{pmatrix}
\begin{pmatrix}
{\bf 0} &\rvline &N_n
\\
\hline
 N_n &\rvline &{\bf 0}
\end{pmatrix}
\begin{pmatrix}
I_{n+1} &\rvline &{\bf 0}
\\
\hline
{\bf 0} &\rvline & A_{n+1}
\end{pmatrix}.
\end{gather*}
\end{lem}
\begin{proof}
It is straightforward to verify the lemma.
\end{proof}

\noindent{\it Proof of Theorem \ref{thm:dimM=2(n+1)}}. 
For $n<\max\{t_1,t_2,t_3\}$ the result is immediate from Corollary \ref{cor:dimMn}. 
Suppose that $n+1\geq t_1+t_2+t_3$. 
Using Lemma \ref{lem:D3} yields that the matrix representing 
$$
\D(x_3)^{t_3}
\big|_{\C^2\otimes \R[x_1,x_2]_n}
:
\C^2\otimes \R[x_1,x_2]_n
\to 
\C^2\otimes \R[x_1,x_2]_{n-t_3}
$$ 
with respect to the basis $\alpha_n$ for $\C^2\otimes \R[x_1,x_2]_n$ and the basis $\alpha_{n-t_3}$ for $\C^2\otimes \R[x_1,x_2]_{n-t_3}$  is of the form
\begin{gather*}
\begin{pmatrix}
I_{n-t_3+1} &\rvline &{\bf 0}
\\
\hline
{\bf 0} &\rvline & A_{n-t_3+1}
\end{pmatrix}
\begin{pmatrix}
{\bf 0} &\rvline & P
\\
\hline
P &\rvline &{\bf 0}
\end{pmatrix}
\begin{pmatrix}
I_{n+1} &\rvline &{\bf 0}
\\
\hline
{\bf 0} &\rvline & A_{n+1}
\end{pmatrix}.
\end{gather*}
Here $P=
N_{n-t_3+1} N_{n-t_3+2}\cdots N_n$. 
Since the square matrices  $I_k$ and $A_k$ are of full rank for any integer $k\geq 1$, it follows that 
\begin{gather}\label{rankD^t3}
{\rm rank}\left(\D(x_3)^{t_3}|_{\C^2\otimes \R[x_1,x_2]_n}\right)=2\cdot {\rm rank}\left(P\right).
\end{gather}
By Lemma \ref{lem:rankN} each of $N_{n-t_3+1}, N_{n-t_3+2},\ldots, N_n$ is of full rank. Hence ${\rm rank}(P)=n-t_3+1$. Combined with (\ref{rankD^t3}) this yields that
$
{\rm rank}\left(\D(x_3)^{t_3}|_{\C^2\otimes \R[x_1,x_2]_n}\right)=2(n-t_3+1)$.
It follows from the rank-nullity theorem that 
$$
{\rm nullity}\left(\D(x_3)^{t_3}|_{\C^2\otimes \R[x_1,x_2]_n}\right)=2t_3.
$$
Hence $\dim \M_n=2(n+1)$ by Corollary \ref{cor:dimMn}(iii). The result follows.
\hfill $\square$

\section{The $\BI$-modules $\M_n$ of type {\bf (I)}}\label{s:case1}

\begin{prop}\label{prop:n<mint13}
Suppose that $n$ is an odd integer with $1\leq n<\min\{t_1,t_3\}$. 
Let 
$$
p_i
=
\sum_{h=0}^i
(-1)^{\frac{h}{2}}
\sum_{j=0}^{n-h}
(-1)^{\frac{j}{2}}
{\left\lfloor \frac{n-i-j}{2}\right\rfloor+\left\lfloor\frac{i-h}{2}\right\rfloor \choose \left\lfloor\frac{i-h}{2}\right\rfloor}
c_{h,i,j}
\otimes
[x_1]^n_j 
[x_2]^i_h
[x_3]^n_{n-h-j}
$$
for all $i=0,1,\ldots,n$ where 
\begin{align*}
c_{h,i,j}&=
\left\{
\begin{array}{ll}
\sigma_2^j \sigma_1 
\qquad &\hbox{if $h$ is odd and $i$ is odd},
\\
\sigma_2^j
\qquad &\hbox{if $h$ is even and $i$ is even},
\\
(-1)^\frac{3}{2} \sigma_3 
\qquad &\hbox{if $h$ is odd, $i$ is even and $j$ is odd},
\\
-\begin{pmatrix}
1 &0
\\
0 &1
\end{pmatrix}
\qquad &\hbox{if $h$ is even, $i$ is odd and $j$ is even},
\\
\begin{pmatrix}
0 &0
\\
0 &0
\end{pmatrix}
\qquad &\hbox{else}
\end{array}
\right.
\end{align*}
for any integers $h,i,j$. 
Then the following hold:
\begin{enumerate}
\item $\{p_i\}_{i=0}^n$ are linearly independent over ${\rm Mat}_2(\C)$.

\item The following equations hold:
\begin{align}
(X-\theta_i)p_i &=p_{i+1} \qquad (0\leq i\leq n-1),
\qquad 
(X-\theta_n) p_n=0,
\label{e:n<mint13-1}
\\
(Y-\theta_i^*) p_i &=\varphi_i p_{i-1} \qquad (1\leq i\leq n),
\qquad 
(Y-\theta_0^*) p_0=0,
\label{e:n<mint13-2}
\end{align}
where 
\begin{align*}
\theta_i &=(-1)^i
\textstyle 
(
k_2+k_3+i+\frac{1}{2}
)
\qquad 
(0\leq i\leq n),
\\
\theta_i^* &=(-1)^{i+1}
\textstyle 
(
k_1+k_3+n-i+\frac{1}{2}
)
\qquad 
(0\leq i\leq n),
\\
\varphi_i
&=
\left\{
\begin{array}{ll}
i(n-i+1)
\qquad
\hbox{if $i$ is even},
\\
(i+2k_2)
(n-i+2k_1+1)
\qquad
\hbox{if $i$ is odd}
\end{array}
\right.
\qquad 
(1\leq i\leq n).
\end{align*}

\item $\D(p_i)=0$ for all $i=0,1,\ldots,n$. 
\end{enumerate}
\end{prop}
\begin{proof}  
(i): Let $i$ be an integer with $0\leq i\leq n$.
By construction the coefficient of $x_2^h$ in $p_i$ is zero for all integers $h$ with $i<h\leq n$. Observe that the coefficient of $x_1^{n-i}x_2^i$ in $p_i$ is 
\begin{gather}\label{coeff:n<t1+t3}
(-1)^\frac{n}{2}
\prod_{h=n-i+1}^n m_1^{(h)}
\prod_{h=1}^{n} m_3^{(h)}
\times 
\left\{
\begin{array}{ll}
\sigma_2 \qquad 
&\hbox{if $i$ is even},
\\
\sigma_1 \qquad 
&\hbox{if $i$ is odd}.
\end{array}
\right.
\end{gather}
Since $n< t_1$ the scalar $\prod\limits_{h=n-i+1}^n m_1^{(h)}$ is nonzero. 
Since $n<t_3$ the scalar $\prod\limits_{h=1}^n m_3^{(h)}$ is nonzero. 
Hence the matrix (\ref{coeff:n<t1+t3}) is nonsingular. By the above comments the part (i) follows.

(ii):  
Let $i,j$ denote two nonnegative integers with $i+j\leq n-1$. Using Theorem \ref{thm:BImodule_Mn}(i) yields that the coefficient of  $[x_1]^n_j [x_2]^{i+1}_{i+1} [x_3]^n_{n-i-j-1}$ in $(X-\theta_i)p_i$ is equal to 
$$
(-1)^{\frac{i-j+1}{2}}
 \sigma_1 c_{i,i,j}=(-1)^{\frac{i+j+1}{2}} c_{i+1,i+1,j}.
$$
Now let $h,i,j$ denote three integers with $0\leq h\leq i\leq n$ and $0\leq j\leq n-h$. Using Theorem \ref{thm:BImodule_Mn}(i) yields that the coefficient of  $[x_1]^n_j [x_2]^i_h [x_3]^n_{n-h-j}$ in $(X-\theta_i)p_i$ is equal to $(-1)^{\frac{h-j}{2}}$ times the sum of 
\begin{align*}
&
m_2^{(h)}
{\left\lfloor \frac{n-i-j}{2}\right\rfloor+\left\lfloor\frac{i-h+1}{2}\right\rfloor \choose \left\lfloor\frac{i-h+1}{2}\right\rfloor}
 \sigma_1 c_{h-1,i,j},
\\
&
m_3^{(n-h-j)}
{\left\lfloor \frac{n-i-j}{2}\right\rfloor+\left\lfloor\frac{i-h-1}{2}\right\rfloor \choose \left\lfloor\frac{i-h-1}{2}\right\rfloor}
 \sigma_1 c_{h+1,i,j},
\\
&
%(-1)^{\frac{h-j}{2}}
{\textstyle(
(-1)^{h+j} k_3
-
(-1)^h k_2
-
(-1)^j\theta_i
-
\frac{1}{2}
)}
{\left\lfloor \frac{n-i-j}{2}\right\rfloor+\left\lfloor\frac{i-h}{2}\right\rfloor \choose \left\lfloor\frac{i-h}{2}\right\rfloor}
c_{h,i,j}.
\end{align*}
It is straightforward to verify that the sum of the above three terms is equal to 
$$
%(-1)^{\frac{h+j}{2}} 
(-1)^j
m_2^{(i+1)}
{\left\lfloor \frac{n-i-j-1}{2}\right\rfloor+\left\lfloor\frac{i-h+1}{2}\right\rfloor \choose \left\lfloor\frac{i-h+1}{2}\right\rfloor}
c_{h,i+1,j}.
$$
By the above comments the equations given in (\ref{e:n<mint13-1}) follow.

%Let $h,i,j$ denote three integers with $0\leq h\leq i\leq n$ and $0\leq j\leq n-h$. 
Using Theorem \ref{thm:BImodule_Mn}(i) yields that the coefficient of  $[x_1]^n_j [x_2]^i_h [x_3]^n_{n-h-j}$ in $(Y-\theta_i^*)p_i$ is equal to $(-1)^{\frac{j-h}{2}+1}$ times the sum of 
\begin{align*}
&
m_1^{(j)}
{\left\lfloor \frac{n-i-j+1}{2}\right\rfloor+\left\lfloor\frac{i-h}{2}\right\rfloor \choose \left\lfloor\frac{i-h}{2}\right\rfloor}
 \sigma_2 c_{h,i,j-1},
\\
&
m_3^{(n-h-j)}
{\left\lfloor \frac{n-i-j-1}{2}\right\rfloor+\left\lfloor\frac{i-h}{2}\right\rfloor \choose \left\lfloor\frac{i-h}{2}\right\rfloor}
 \sigma_2 c_{h,i,j+1},
\\
&
{\textstyle(
(-1)^j k_1
-
(-1)^{h+j} k_3
+
(-1)^h
\theta_i^*
+
\frac{1}{2}
)}
{\left\lfloor \frac{n-i-j}{2}\right\rfloor+\left\lfloor\frac{i-h}{2}\right\rfloor \choose \left\lfloor\frac{i-h}{2}\right\rfloor}
c_{h,i,j}.
\end{align*}
It is straightforward to verify that the sum of the above three terms is equal to $(-1)^{h+1}$ times 
$$
\left\{
\begin{array}{ll}
\begin{pmatrix}
0 &0
\\
0 &0
\end{pmatrix}
\qquad &\hbox{if $h=i$},
\\
(n-i+1)
\displaystyle{\left\lfloor \frac{n-i-j+1}{2}\right\rfloor+\left\lfloor\frac{i-h-1}{2}\right\rfloor \choose \left\lfloor\frac{i-h-1}{2}\right\rfloor}
c_{h,i-1,j}
\qquad &\hbox{if $h<i$ and $i$ is even},
\\
(n-i+2k_1+1)
\displaystyle{\left\lfloor \frac{n-i-j+1}{2}\right\rfloor+\left\lfloor\frac{i-h-1}{2}\right\rfloor \choose \left\lfloor\frac{i-h-1}{2}\right\rfloor}
c_{h,i-1,j}
\qquad &\hbox{if $h<i$ and $i$ is odd}.
\end{array}
\right.
$$
By the above comments the equations given in (\ref{e:n<mint13-2}) follow.

(iii): It is routine to verify that $\D(p_0)=0$. Combined with (ii) the statement (iii) follows.
\end{proof}

\begin{lem}\label{lem:basis:n<mint13(odd)}
Suppose that $n$ is an odd integer with $1\leq n<\min\{t_1,t_3\}$. 
Let $\{p_i\}_{i=0}^n$ be as in Proposition \ref{prop:n<mint13}. Then $\M_n$ has the basis
\begin{align}
p_i\cdot 
\begin{pmatrix}
1
\\
0
\end{pmatrix}
\otimes
1 
\qquad 
(0\leq i\leq n), 
\label{basis:n<mint13(odd)-1}
\\ 
p_i\cdot 
\begin{pmatrix}
0
\\
1
\end{pmatrix}
\otimes
1 
\qquad 
(0\leq i\leq n).
\label{basis:n<mint13(odd)-2}
\end{align}
\end{lem}
\begin{proof}
It follows from Theorem \ref{thm:dimM=2(n+1)} that $\dim \M_n=2(n+1)$.  
The linear independence of (\ref{basis:n<mint13(odd)-1}) and (\ref{basis:n<mint13(odd)-2}) follows from Proposition \ref{prop:n<mint13}(i). 
By Proposition \ref{prop:n<mint13}(iii) the polynomials (\ref{basis:n<mint13(odd)-1}) and (\ref{basis:n<mint13(odd)-2}) are in $\M_n$. 
The lemma follows. 
\end{proof}

\begin{thm}\label{thm:n<mint123(odd)}
Suppose that $n$ is an odd integer with $1\leq n<\min\{t_1,t_2,t_3\}$. Then the $\BI$-module $\M_n$ is isomorphic to a direct sum of two copies of 
\begin{gather}\label{BImodule:n<mint123(odd)}
E_n
\textstyle(
k_2+k_3+\frac{n+1}{2},-k_1-k_3-\frac{n+1}{2},k_1+k_2+\frac{n+1}{2}
).
\end{gather}
Moreover, if each of $k_1,k_2,k_3$ is nonnegative then the $\BI$-module (\ref{BImodule:n<mint123(odd)}) is irreducible and it is isomorphic to 
\begin{gather}\label{BImodule:k123>=0(odd)}
E_n
\textstyle(
k_2+k_3+\frac{n+1}{2},k_1+k_3+\frac{n+1}{2},k_1+k_2+\frac{n+1}{2}
).
\end{gather}
\end{thm}
\begin{proof}
Let $V$ and $V'$ denote the subspaces of $\M_n$ spanned by (\ref{basis:n<mint13(odd)-1}) and (\ref{basis:n<mint13(odd)-2}), respectively. It follows from Lemma \ref{lem:basis:n<mint13(odd)} that 
$
\M_n=V\oplus V'$. 
It follows from Theorem \ref{thm:BImodule_Mn}(ii) and Proposition \ref{prop:n<mint13}(ii) that $V$ and $V'$ are two isomorphic $\BI$-submodules of $\M_n$. Using Proposition \ref{prop:Ed} yields that both are isomorphic to (\ref{BImodule:n<mint123(odd)}). The first assertion follows.

Suppose that each of $k_1,k_2,k_3$ is nonnegative. Using Theorem \ref{thm:irr_E} yields that the $\BI$-module (\ref{BImodule:n<mint123(odd)}) is irreducible. By Theorem \ref{thm:onto2_E} the $\BI$-module  (\ref{BImodule:n<mint123(odd)}) is isomorphic to (\ref{BImodule:k123>=0(odd)}). The second assertion follows.
\end{proof}

\begin{prop}\label{prop:n<mint13(even)}
Suppose that $n$ is an even integer with $0\leq n<\min\{t_1,t_3\}$. 
Let 
$$
p_i
=
\sum_{h=0}^i
(-1)^{-\frac{h}{2}}
\sum_{j=0}^{n-h}
(-1)^{\frac{j}{2}}
{\left\lfloor \frac{n-i-j}{2}\right\rfloor+\left\lfloor\frac{i-h}{2}\right\rfloor \choose \left\lfloor\frac{i-h}{2}\right\rfloor}
c_{h,i,j}
[x_1]^n_j 
[x_2]^i_h
[x_3]^n_{n-h-j}
$$
for all $i=0,1,\ldots,n$ where 
\begin{align*}
c_{h,i,j}&=
\left\{
\begin{array}{ll}
\sigma_2^j \sigma_1 
\qquad &\hbox{if $h$ is odd and $i$ is odd},
\\
\sigma_2^j
\qquad &\hbox{if $h$ is even and $i$ is even},
\\
\sigma_1
\qquad &\hbox{if $h$ is odd, $i$ is even and $j$ is even},
\\
-\sigma_2
\qquad &\hbox{if $h$ is even, $i$ is odd and $j$ is odd},
\\
\begin{pmatrix}
0 &0
\\
0 &0
\end{pmatrix}
\qquad &\hbox{else}
\end{array}
\right.
\end{align*}
for any integers $h,i,j$.  Then the following hold:
\begin{enumerate}
\item $\{p_i\}_{i=0}^n$ are linearly independent over ${\rm Mat}_2(\C)$.

\item The following equations hold:
\begin{align}
(X-\theta_i)p_i &=p_{i+1} \qquad (0\leq i\leq n-1),
\qquad 
(X-\theta_n) p_n=0,
\label{e:n<mint13-1(even)}
\\
(Y-\theta_i^*) p_i &=\varphi_i p_{i-1} \qquad (1\leq i\leq n),
\qquad 
(Y-\theta_0^*) p_0=0,
\label{e:n<mint13-2(even)}
\end{align}
where 
\begin{align*}
\theta_i &=(-1)^i
\textstyle(
k_2+k_3+i+\frac{1}{2}
)
\qquad 
(0\leq i\leq n),
\\
\theta_i^* &=(-1)^i
\textstyle(
k_1+k_3+n-i+\frac{1}{2}
)
\qquad 
(0\leq i\leq n),
\\
\varphi_i
&=
\left\{
\begin{array}{ll}
i
(i-n-2k_1-1)
\qquad
\hbox{if $i$ is even},
\\
(i+2k_2)
(i-n-1)
\qquad
\hbox{if $i$ is odd}
\end{array}
\right.
\qquad 
(1\leq i\leq n).
\end{align*}

\item $\D(p_i)=0$ for all $i=0,1,\ldots,n$. 
\end{enumerate}
\end{prop}
\begin{proof}
(i): Let $i$ be an integer with $0\leq i\leq n$.
By construction the coefficient of $x_2^h$ in $p_i$ is zero for all integers $h$ with $i<h\leq n$. Observe that the coefficient of $x_1^{n-i}x_2^i$ in $p_i$ is 
\begin{gather}\label{coeff:n<t1+t3(even)}
(-1)^{\frac{n}{2}-i}
\prod_{h=n-i+1}^n m_1^{(h)}
\prod_{h=1}^{n} m_3^{(h)}
\times 
\left\{
\begin{array}{ll}
\begin{pmatrix}
1 &0
\\
0 &1
\end{pmatrix}
\qquad 
&\hbox{if $i$ is even},
\\
(-1)^\frac{3}{2}\sigma_3 \qquad 
&\hbox{if $i$ is odd}.
\end{array}
\right.
\end{gather}
Since $n< t_1$ the scalar $\prod\limits_{h=n-i+1}^n m_1^{(h)}$ is nonzero. 
Since $n<t_3$ the scalar $\prod\limits_{h=1}^n m_3^{(h)}$ is nonzero. 
Hence the matrix (\ref{coeff:n<t1+t3(even)}) is nonsingular. By the above comments the part (i) follows.

(ii): Let $i,j$ denote two nonnegative integers with $i+j\leq n-1$. Using Theorem \ref{thm:BImodule_Mn}(i) yields that the coefficient of  $[x_1]^n_j [x_2]^{i+1}_{i+1} [x_3]^n_{n-i-j-1}$ in $(X-\theta_i)p_i$ is equal to 
$$
(-1)^{\frac{3-i-j}{2}}
 \sigma_1 c_{i,i,j}=(-1)^{\frac{j-i-1}{2}} c_{i+1,i+1,j}.
$$
Now let $h,i,j$ denote three integers with $0\leq h\leq i\leq n$ and $0\leq j\leq n-h$. Using Theorem \ref{thm:BImodule_Mn}(i) yields that the coefficient of  $[x_1]^n_j [x_2]^i_h [x_3]^n_{n-h-j}$ in $(X-\theta_i)p_i$ is equal to $(-1)^{-\frac{h+j}{2}}$ times the sum of 
\begin{align*}
&
m_2^{(h)}
{\left\lfloor \frac{n-i-j}{2}\right\rfloor+\left\lfloor\frac{i-h+1}{2}\right\rfloor \choose \left\lfloor\frac{i-h+1}{2}\right\rfloor}
 \sigma_1 c_{h-1,i,j},
\\
&
m_3^{(n-h-j)}
{\left\lfloor \frac{n-i-j}{2}\right\rfloor+\left\lfloor\frac{i-h-1}{2}\right\rfloor \choose \left\lfloor\frac{i-h-1}{2}\right\rfloor}
 \sigma_1 c_{h+1,i,j},
\\
&
%(-1)^{\frac{h-j}{2}}
{\textstyle(
(-1)^{h+j} k_3
+
(-1)^h k_2
-
(-1)^j\theta_i
+
\frac{1}{2}
)}
{\left\lfloor \frac{n-i-j}{2}\right\rfloor+\left\lfloor\frac{i-h}{2}\right\rfloor \choose \left\lfloor\frac{i-h}{2}\right\rfloor}
c_{h,i,j}.
\end{align*}
It is straightforward to verify that the sum of the above three terms is equal to 
$$
%(-1)^{\frac{h+j}{2}} 
(-1)^j
m_2^{(i+1)}
{\left\lfloor \frac{n-i-j-1}{2}\right\rfloor+\left\lfloor\frac{i-h+1}{2}\right\rfloor \choose \left\lfloor\frac{i-h+1}{2}\right\rfloor}
c_{h,i+1,j}. 
$$
By the above comments the equations given in (\ref{e:n<mint13-1(even)}) follow.

%Let $h,i,j$ denote three integers with $0\leq h\leq i\leq n$ and $0\leq j\leq n-h$. 
Using Theorem \ref{thm:BImodule_Mn}(i) yields that the coefficient of  $[x_1]^n_j [x_2]^i_h [x_3]^n_{n-h-j}$ in $(Y-\theta_i^*)p_i$ is equal to $(-1)^{\frac{h+j}{2}}$ times the sum of 
\begin{align*}
&
m_1^{(j)}
{\left\lfloor \frac{n-i-j+1}{2}\right\rfloor+\left\lfloor\frac{i-h}{2}\right\rfloor \choose \left\lfloor\frac{i-h}{2}\right\rfloor}
 \sigma_2 c_{h,i,j-1},
\\
&
m_3^{(n-h-j)}
{\left\lfloor \frac{n-i-j-1}{2}\right\rfloor+\left\lfloor\frac{i-h}{2}\right\rfloor \choose \left\lfloor\frac{i-h}{2}\right\rfloor}
 \sigma_2 c_{h,i,j+1},
\\
&
{\textstyle(
(-1)^j k_1
+
(-1)^{h+j} k_3
-
(-1)^h
\theta_i^*
+
\frac{1}{2}
)}
{\left\lfloor \frac{n-i-j}{2}\right\rfloor+\left\lfloor\frac{i-h}{2}\right\rfloor \choose \left\lfloor\frac{i-h}{2}\right\rfloor}
c_{h,i,j}.
\end{align*}
It is straightforward to verify that the sum of the above three terms is equal to $(-1)^h$ times 
$$
\left\{
\begin{array}{ll}
\begin{pmatrix}
0 &0
\\
0 &0
\end{pmatrix}
\qquad &\hbox{if $h=i$},
\\
(i-n-2k_1-1)
\displaystyle{\left\lfloor \frac{n-i-j+1}{2}\right\rfloor+\left\lfloor\frac{i-h-1}{2}\right\rfloor \choose \left\lfloor\frac{i-h-1}{2}\right\rfloor}
c_{h,i-1,j}
\qquad &\hbox{if $h<i$ and $i$ is even},
\\
(i-n-1)
\displaystyle{\left\lfloor \frac{n-i-j+1}{2}\right\rfloor+\left\lfloor\frac{i-h-1}{2}\right\rfloor \choose \left\lfloor\frac{i-h-1}{2}\right\rfloor}
c_{h,i-1,j}
\qquad &\hbox{if $h<i$ and $i$ is odd}.
\end{array}
\right.
$$
By the above comments the equations given in (\ref{e:n<mint13-2(even)}) follow.

(iii): It is routine to verify that $\D(p_0)=0$. Combined with (ii) the statement (iii) follows.
\end{proof}

\begin{lem}\label{lem:basis:n<mint13(even)}
Suppose that $n$ is an even integer with $0\leq n<\min\{t_1,t_3\}$. 
Let $\{p_i\}_{i=0}^n$ be as in Proposition \ref{prop:n<mint13(even)}. Then $\M_n$ has the basis
\begin{align}
p_i\cdot 
\begin{pmatrix}
1
\\
0
\end{pmatrix}
\otimes
1 
\qquad 
(0\leq i\leq n), 
\label{basis:n<mint13(even)-1}
\\ 
p_i\cdot 
\begin{pmatrix}
0
\\
1
\end{pmatrix}
\otimes
1 
\qquad 
(0\leq i\leq n).
\label{basis:n<mint13(even)-2}
\end{align}
\end{lem}
\begin{proof}
It follows from Theorem \ref{thm:dimM=2(n+1)} that $\dim \M_n=2(n+1)$.  
The linear independence of (\ref{basis:n<mint13(even)-1}) and (\ref{basis:n<mint13(even)-2}) follows from Proposition \ref{prop:n<mint13(even)}(i). 
By Proposition \ref{prop:n<mint13(even)}(iii) the polynomials (\ref{basis:n<mint13(even)-1}) and (\ref{basis:n<mint13(even)-2}) are in $\M_n$. 
The lemma follows. 
\end{proof}

\begin{thm}\label{thm:n<mint123(even)}
Suppose that $n$ is an even integer with $0\leq n<\min\{t_1,t_2,t_3\}$. Then the $\BI$-module $\M_n$ is isomorphic to a direct sum of two copies of 
\begin{gather}\label{BImodule:n<mint123(even)}
O_n
\textstyle(
k_2+k_3+\frac{n+1}{2},
-k_1-k_3-\frac{n+1}{2},
-k_1-k_2-\frac{n+1}{2}
)^{(1,-1)}.
\end{gather}
Moreover, if each of $k_1,k_2,k_3$ is nonnegative then the $\BI$-module (\ref{BImodule:n<mint123(even)}) is irreducible and it is isomorphic to 
\begin{gather}\label{BImodule:k123>=0(even)}
O_n
\textstyle(
k_2+k_3+\frac{n+1}{2},k_1+k_3+\frac{n+1}{2},k_1+k_2+\frac{n+1}{2}
).
\end{gather}
\end{thm}
\begin{proof}
Let $V$ and $V'$ denote the subspaces of $\M_n$ spanned by (\ref{basis:n<mint13(even)-1}) and (\ref{basis:n<mint13(even)-2}), respectively. It follows from Lemma \ref{lem:basis:n<mint13(even)} that 
$
\M_n=V\oplus V'$. 
It follows from Theorem \ref{thm:BImodule_Mn}(ii) and Proposition \ref{prop:n<mint13(even)}(ii) that $V$ and $V'$ are two isomorphic $\BI$-submodules of $\M_n$. Using Proposition \ref{prop:Od} yields that both are isomorphic to (\ref{BImodule:n<mint123(even)}).  The first assertion follows.

Suppose  that each of $k_1,k_2,k_3$ is nonnegative. Using Theorem \ref{thm:irr_O} yields that the $\BI$-module (\ref{BImodule:n<mint123(even)}) is irreducible. By Theorem \ref{thm:onto2_O} the $\BI$-module  (\ref{BImodule:n<mint123(even)}) is isomorphic to (\ref{BImodule:k123>=0(even)}). The second assertion follows.
\end{proof}

\section{The $\BI$-modules $\M_n$ of type (II)}\label{s:case2}

\begin{lem}\label{lem:basis:t2<=n<mint13(odd)}
Suppose that $n$ is an odd integer with $t_2\leq n<\min\{t_1,t_3\}$. 
Let $\{p_i\}_{i=0}^n$ be as in Proposition \ref{prop:n<mint13}. Then the following hold:
\begin{enumerate}
\item $\M_n(x_2)$ has the basis
\begin{align}
p_i\cdot 
\begin{pmatrix}
1
\\
0
\end{pmatrix}
\otimes
1 
\qquad 
(t_2\leq i\leq n), 
\label{M2basis:n<mint13(odd)-1}
\\ 
p_i\cdot 
\begin{pmatrix}
0
\\
1
\end{pmatrix}
\otimes
1 
\qquad 
(t_2\leq i\leq n).
\label{M2basis:n<mint13(odd)-2}
\end{align}

\item $\M_n/\M_n(x_2)$ has the basis
\begin{align}
p_i\cdot 
\begin{pmatrix}
1
\\
0
\end{pmatrix}
\otimes
1 
+\M_n(x_2)
\qquad 
(0\leq i\leq t_2-1), 
\label{M/M2basis:n<mint13(odd)-1}
\\ 
p_i\cdot 
\begin{pmatrix}
0
\\
1
\end{pmatrix}
\otimes
1 
+\M_n(x_2)
\qquad 
(0\leq i\leq t_2-1).
\label{M/M2basis:n<mint13(odd)-2}
\end{align}
\end{enumerate}
\end{lem}
\begin{proof}
(i): 
The linear independence of (\ref{M2basis:n<mint13(odd)-1}) and (\ref{M2basis:n<mint13(odd)-2}) follows from Lemma \ref{lem:basis:n<mint13(odd)}. 
Since (\ref{M2basis:n<mint13(odd)-1}) and (\ref{M2basis:n<mint13(odd)-2}) are in 
$\C^2\otimes\left(\bigoplus\limits_{i=t_2}^n x_2^i\cdot \R[x_1,x_3]_{n-i}\right)$, it follows from Proposition \ref{prop:n<mint13}(iii) that (\ref{M2basis:n<mint13(odd)-1}) and (\ref{M2basis:n<mint13(odd)-2}) are in $\M_n(x_2)$. 
Combined with Proposition \ref{prop:dimM(x1)}(ii) the part (i) follows.

(ii): It is immediate from Lemma \ref{lem:basis:n<mint13(odd)} and Lemma \ref{lem:basis:t2<=n<mint13(odd)}(i).
\end{proof}

\begin{thm}\label{thm:n<mint13(odd)}
Suppose that $n$ is an odd integer with $t_2\leq n<\min\{t_1,t_3\}$. Then the following hold:
\begin{enumerate}
\item  The $\BI$-module $\M_n(x_2)$ is isomorphic to a direct sum of two copies of 
\begin{gather}\label{BImodule:t2<=n<mint13(odd)-1}
O_{n-t_2}
\textstyle(
k_3+\frac{n+1}{2},
-k_1-k_2-k_3-\frac{n+1}{2},
-k_1-\frac{n+1}{2}
)^{(-1,-1)}.
\end{gather}

\item The $\BI$-module $\M_n/\M_n(x_2)$ is isomorphic to a direct sum of two copies of 
\begin{gather}\label{BImodule:t2<=n<mint13(odd)-2}
O_{t_2-1}
(k_3,-k_1-k_2-k_3-n-1,k_1).
\end{gather}
\end{enumerate}
Moreover, if $k_1,k_3$ are nonnegative then the $\BI$-modules (\ref{BImodule:t2<=n<mint13(odd)-1}) and (\ref{BImodule:t2<=n<mint13(odd)-2}) are irreducible and (\ref{BImodule:t2<=n<mint13(odd)-1}) is isomorphic to 
\begin{gather}\label{BImodule:t2<=n<mint13(odd)-1'}
O_{n-t_2}
\textstyle(
-k_3-\frac{n+1}{2},
k_1+k_2+k_3+\frac{n+1}{2},
-k_1-\frac{n+1}{2}
).
\end{gather}
\end{thm}
\begin{proof}
(i): Let $V$ and $V'$ denote the subspaces of $\M_n(x_2)$ spanned by (\ref{M2basis:n<mint13(odd)-1}) and  (\ref{M2basis:n<mint13(odd)-2}), respectively. It follows from Lemma \ref{lem:basis:t2<=n<mint13(odd)}(i) that 
$
\M_n(x_2)=V\oplus V'$. 
It follows from Theorem \ref{thm:BImodule_Mn}(ii) and Proposition \ref{prop:n<mint13}(ii) that $V$ and $V'$ are two isomorphic $\BI$-submodules of $\M_n(x_2)$. Using Proposition \ref{prop:Od} yields that both are isomorphic to (\ref{BImodule:t2<=n<mint13(odd)-1}).

(ii): Let $V$ and $V'$ denote the subspaces of $\M_n/\M_n(x_2)$ spanned by (\ref{M/M2basis:n<mint13(odd)-1}) and (\ref{M/M2basis:n<mint13(odd)-2}), respectively. It follows from Lemma \ref{lem:basis:t2<=n<mint13(odd)}(ii) that 
$
\M_n/\M_n(x_2)=V\oplus V'$. 
It follows from Theorem \ref{thm:BImodule_Mn}(ii) and Proposition \ref{prop:n<mint13}(ii) that $V$ and $V'$ are two isomorphic $\BI$-submodules of $\M_n/\M_n(x_2)$. Using Proposition \ref{prop:Od} yields that both are isomorphic to (\ref{BImodule:t2<=n<mint13(odd)-2}).

Suppose that $k_1,k_3$ are nonnegative. Using Theorem \ref{thm:irr_O} yields that the $\BI$-modules (\ref{BImodule:t2<=n<mint13(odd)-1}) and (\ref{BImodule:t2<=n<mint13(odd)-2}) irreducible. By Theorem \ref{thm:onto2_O} the $\BI$-module (\ref{BImodule:t2<=n<mint13(odd)-1}) is isomorphic to (\ref{BImodule:t2<=n<mint13(odd)-1'}).
\end{proof}

By similar arguments as in the proof of Theorem \ref{thm:n<mint13(odd)} we have the following results:

\begin{thm}\label{thm:n<mint12(odd)}
Suppose that $n$ is an  odd integer with $t_3\leq n<\min\{t_1,t_2\}$. Then the following hold:
\begin{enumerate}
\item The $\BI$-module $\M_n(x_3)$ is isomorphic to a direct sum of two copies of 
\begin{gather}\label{n<mint12(odd)-1}
O_{n-t_3}
\textstyle(
k_1+\frac{n+1}{2},
-k_1-k_2-k_3-\frac{n+1}{2},
-k_2-\frac{n+1}{2}
)^{((-1,-1),(1\,3\,2))}.
\end{gather}

\item The $\BI$-module $\M_n/\M_n(x_3)$  is isomorphic to a direct sum of two copies of 
\begin{gather}\label{n<mint12(odd)-2}
O_{t_3-1}
(k_1,-k_1-k_2-k_3-n-1,k_2)^{(1\,3\,2)}.
\end{gather}
\end{enumerate}
Moreover, if $k_1,k_2$ are nonnegative then the $\BI$-modules (\ref{n<mint12(odd)-1}) and (\ref{n<mint12(odd)-2}) are irreducible and they are isomorphic to 
\begin{align*}
&O_{n-t_3}
\textstyle(
-k_2-\frac{n+1}{2},
-k_1-\frac{n+1}{2},
k_1+k_2+k_3+\frac{n+1}{2}
),
\\
&O_{t_3-1}
(k_2,k_1,-k_1-k_2-k_3-n-1),
\end{align*}
respectively.
\end{thm}

\begin{thm}\label{thm:n<mint23(odd)}
Suppose that $n$ is an odd integer with $t_1\leq n<\min\{t_2,t_3\}$. Then the following hold:
\begin{enumerate}
\item The $\BI$-module $\M_n(x_1)$ is isomorphic to a direct sum of two copies of 
\begin{gather}\label{n<mint23(odd)-1}
O_{n-t_1}
\textstyle(
k_2+\frac{n+1}{2},
-k_1-k_2-k_3-\frac{n+1}{2},
-k_3-\frac{n+1}{2}
)^{((-1,-1),(1\,2\,3))}.
\end{gather}

\item The $\BI$-module $\M_n/\M_n(x_1)$ is isomorphic to a direct sum of two copies of 
\begin{gather}\label{n<mint23(odd)-2}
O_{t_1-1}
(k_2,-k_1-k_2-k_3-n-1,k_3)^{(1\,2\,3)}.
\end{gather}
\end{enumerate}
Moreover, if $k_2,k_3$ are nonnegative then the $\BI$-modules (\ref{n<mint23(odd)-1}) and (\ref{n<mint23(odd)-2}) are irreducible and they are isomorphic to 
\begin{align*}
&O_{n-t_1}
\textstyle(
k_1+k_2+k_3+\frac{n+1}{2},
-k_3-\frac{n+1}{2},
-k_2-\frac{n+1}{2}
),
\\
&O_{t_1-1}
(-k_1-k_2-k_3-n-1,k_3,k_2),
\end{align*}
respectively.
\end{thm}

\begin{lem}\label{lem:basis:t2<=n<mint13(even)}
Suppose that $n$ is an even integer with $t_2\leq n<\min\{t_1,t_3\}$. 
Let $\{p_i\}_{i=0}^n$ be as in Proposition \ref{prop:n<mint13(even)}. Then the following hold:
\begin{enumerate}
\item $\M_n(x_2)$ has the basis
\begin{align}
p_i\cdot 
\begin{pmatrix}
1
\\
0
\end{pmatrix}
\otimes
1 
\qquad 
(t_2\leq i\leq n), 
\label{M2basis:n<mint13(even)-1}
\\ 
p_i\cdot 
\begin{pmatrix}
0
\\
1
\end{pmatrix}
\otimes
1 
\qquad 
(t_2\leq i\leq n).
\label{M2basis:n<mint13(even)-2}
\end{align}

\item $\M_n/\M_n(x_2)$ has the basis
\begin{align}
p_i\cdot 
\begin{pmatrix}
1
\\
0
\end{pmatrix}
\otimes
1 
+\M_n(x_2)
\qquad 
(0\leq i\leq t_2-1), 
\label{M/M2basis:n<mint13(even)-1}
\\ 
p_i\cdot 
\begin{pmatrix}
0
\\
1
\end{pmatrix}
\otimes
1 
+\M_n(x_2)
\qquad 
(0\leq i\leq t_2-1).
\label{M/M2basis:n<mint13(even)-2}
\end{align}
\end{enumerate}
\end{lem}
\begin{proof}
(i): 
The linear independence of (\ref{M2basis:n<mint13(even)-1}) and (\ref{M2basis:n<mint13(even)-2}) follows from Lemma \ref{lem:basis:n<mint13(even)}. 
Since (\ref{M2basis:n<mint13(even)-1}) and (\ref{M2basis:n<mint13(even)-2}) are in $\C^2\otimes\left(\bigoplus\limits_{i=t_2}^n x_2^i\cdot \R[x_1,x_3]_{n-i}\right)$,  it follows from Proposition \ref{prop:n<mint13(even)}(iii) that (\ref{M2basis:n<mint13(even)-1}) and (\ref{M2basis:n<mint13(even)-2}) are in $\M_n(x_2)$. 
Combined with Proposition \ref{prop:dimM(x1)}(ii) the part (i) follows.

(ii): It is immediate from Lemma \ref{lem:basis:n<mint13(even)} and Lemma \ref{lem:basis:t2<=n<mint13(even)}(i).
\end{proof}

\begin{thm}\label{thm:n<mint13(even)}
Suppose that $n$ is an even integer with $t_2\leq n<\min\{t_1,t_3\}$. Then the following hold:
\begin{enumerate}
\item  The $\BI$-module $\M_n(x_2)$ is isomorphic to a direct sum of two copies of 
\begin{gather}\label{BImodule:t2<=n<mint13(even)-1}
E_{n-t_2}
\textstyle(
k_3+\frac{n+1}{2},
-k_1-k_2-k_3-\frac{n+1}{2},
-k_1-\frac{n+1}{2}
)^{(-1,1)}.
\end{gather}

\item The $\BI$-module $\M_n/\M_n(x_2)$  is isomorphic to a direct sum of two copies of 
\begin{gather}\label{BImodule:t2<=n<mint13(even)-2}
O_{t_2-1}
(k_3,-k_1-k_2-k_3-n-1,-k_1)^{(1,-1)}.
\end{gather}
\end{enumerate}
Moreover, if $k_1,k_3$ are nonnegative then the $\BI$-modules (\ref{BImodule:t2<=n<mint13(even)-1}) and (\ref{BImodule:t2<=n<mint13(even)-2}) are irreducible and they are isomorphic to 
\begin{align}
&
E_{n-t_2}
\textstyle(
k_3+\frac{n+1}{2},
k_1+k_2+k_3+\frac{n+1}{2},
k_1+\frac{n+1}{2}
)^{(-1,1)},
\label{BImodule:t2<=n<mint13(even)-1'}
\\
&O_{t_2-1}(k_3,k_1+k_2+k_3+n+1,k_1),
\label{BImodule:t2<=n<mint13(even)-2'}
\end{align}
respectively.
\end{thm}
\begin{proof}
(i): Let $V$ and $V'$ denote the subspaces of $\M_n(x_2)$ spanned by (\ref{M2basis:n<mint13(even)-1}) and  (\ref{M2basis:n<mint13(even)-2}), respectively. It follows from Lemma \ref{lem:basis:t2<=n<mint13(even)}(i) that 
$
\M_n(x_2)=V\oplus V'$. 
It follows from Theorem \ref{thm:BImodule_Mn}(ii) and Proposition \ref{prop:n<mint13(even)}(ii) that $V$ and $V'$ are two isomorphic $\BI$-submodules of $\M_n(x_2)$. Using Proposition \ref{prop:Ed} yields that both are isomorphic to (\ref{BImodule:t2<=n<mint13(even)-1}).

(ii): Let $V$ and $V'$ denote the subspaces of $\M_n/\M_n(x_2)$ spanned by (\ref{M/M2basis:n<mint13(even)-1}) and (\ref{M/M2basis:n<mint13(even)-2}), respectively. It follows from Lemma \ref{lem:basis:t2<=n<mint13(even)}(ii) that 
$
\M_n/\M_n(x_2)=V\oplus V'$. 
It follows from Theorem \ref{thm:BImodule_Mn}(ii) and Proposition \ref{prop:n<mint13(even)}(ii) that $V$ and $V'$ are two isomorphic $\BI$-submodules of $\M_n/\M_n(x_2)$. Using Proposition \ref{prop:Od} yields that both are isomorphic to (\ref{BImodule:t2<=n<mint13(even)-2}).

Suppose that $k_1,k_3$ are nonnegative. Using Theorem \ref{thm:irr_E} yields that the $\BI$-module (\ref{BImodule:t2<=n<mint13(even)-1}) is irreducible. Using Theorem \ref{thm:irr_O} yields that the $\BI$-module (\ref{BImodule:t2<=n<mint13(even)-2}) is irreducible. By Theorem \ref{thm:onto2_E} the $\BI$-module (\ref{BImodule:t2<=n<mint13(even)-1}) is isomorphic to (\ref{BImodule:t2<=n<mint13(even)-1'}). 
By Theorem \ref{thm:onto2_O} the $\BI$-module (\ref{BImodule:t2<=n<mint13(even)-2}) is isomorphic to (\ref{BImodule:t2<=n<mint13(even)-2'}).
\end{proof}

By similar arguments as in the proof of Theorem \ref{thm:n<mint13(even)} we have the following results:

\begin{thm}\label{thm:n<mint12(even)}
Suppose that $n$ is an even integer with $t_3\leq n<\min\{t_1,t_2\}$. Then the following hold:
\begin{enumerate}
\item The $\BI$-module $\M_n(x_3)$ is isomorphic to a direct sum of two copies of 
\begin{gather}\label{n<mint12(even)-1}
E_{n-t_3}
\textstyle(
k_1+\frac{n+1}{2},
-k_1-k_2-k_3-\frac{n+1}{2},
-k_2-\frac{n+1}{2}
)^{((-1,1),(1\,3\,2))}.
\end{gather}

\item The $\BI$-module $\M_n/\M_n(x_3)$  is isomorphic to a direct sum of two copies of 
\begin{gather}\label{n<mint12(even)-2}
O_{t_3-1}
(k_1,-k_1-k_2-k_3-n-1,-k_2)^{((1,-1),(1\,3\,2))}.
\end{gather}
\end{enumerate}
Moreover, if $k_1,k_2$ are nonnegative then the $\BI$-modules (\ref{n<mint12(even)-1}) and (\ref{n<mint12(even)-2}) are irreducible and they are isomorphic to 
\begin{align*}
&
E_{n-t_3}
\textstyle(
k_2+\frac{n+1}{2},
k_1+\frac{n+1}{2},
k_1+k_2+k_3+\frac{n+1}{2}
)^{(-1,-1)},
\\
&O_{t_3-1}(k_2,k_1,k_1+k_2+k_3+n+1),
\end{align*}
respectively.
\end{thm}

\begin{thm}\label{thm:n<mint23(even)}
Suppose that $n$ is an even integer with $t_1\leq n<\min\{t_2,t_3\}$. Then the following hold:
\begin{enumerate}
\item The $\BI$-module $\M_n(x_1)$ is isomorphic to a direct sum of two copies of 
\begin{gather}\label{n<mint23(even)-1}
E_{n-t_1}
\textstyle(
k_2+\frac{n+1}{2},
-k_1-k_2-k_3-\frac{n+1}{2},
-k_3-\frac{n+1}{2}
)^{((-1,1),(1\,2\,3))}.
\end{gather}

\item The $\BI$-module $\M_n/\M_n(x_1)$ is isomorphic to a direct sum of two copies of 
\begin{gather}\label{n<mint23(even)-2}
O_{t_1-1}
(k_2,-k_1-k_2-k_3-n-1,-k_3)^{((1,-1),(1\,2\,3))}.
\end{gather}
\end{enumerate}
Moreover, if $k_2,k_3$ are nonnegative then the $\BI$-modules (\ref{n<mint23(even)-1}) and (\ref{n<mint23(even)-2}) are irreducible and they are isomorphic to 
\begin{align*}
&
E_{n-t_1}
\textstyle(
k_1+k_2+k_3+\frac{n+1}{2},
k_3+\frac{n+1}{2},
k_2+\frac{n+1}{2}
)^{(1,-1)},
\\
&
O_{t_1-1}
(k_1+k_2+k_3+n+1,k_3,k_2),
\end{align*}
respectively.
\end{thm}

\section{The $\BI$-modules $\M_n$ of type (III)}\label{s:case3}

\begin{prop}\label{prop:t1<=n<t1+t2(odd)}
Suppose that $n$ is an odd integer with $t_1\leq n<t_1+t_2$. 
Let 
$$
p_i
=
\sum_{h=0}^i
(-1)^{-\frac{h}{2}}
\sum_{j=t_1}^{n-h}
(-1)^{-\frac{j}{2}}
{\left\lfloor \frac{n-i-j}{2}\right\rfloor+\left\lfloor\frac{i-h}{2}\right\rfloor \choose \left\lfloor\frac{i-h}{2}\right\rfloor}
c_{h,i,j}
\otimes 
[x_1]^n_j 
[x_2]^{n-t_1}_{n-h-j}
[x_3]^i_h
$$
for all $i=0,1,\ldots,n-t_1$ where
\begin{align*}
c_{h,i,j}&=
\left\{
\begin{array}{ll}
\sigma_3^j \sigma_1 
\qquad &\hbox{if $h$ is odd and $i$ is odd},
\\
\sigma_3^j
\qquad &\hbox{if $h$ is even and $i$ is even},
\\
(-1)^\frac{1}{2} \sigma_2
\qquad &\hbox{if $h$ is odd, $i$ is even and $j$ is odd},
\\
-\begin{pmatrix}
1 &0
\\
0 &1
\end{pmatrix}
\qquad &\hbox{if $h$ is even, $i$ is odd and $j$ is even},
\\
\begin{pmatrix}
0 &0
\\
0 &0
\end{pmatrix}
\qquad &\hbox{else}
\end{array}
\right.
\end{align*}
for any integers $h,i,j$. 
Then the following hold:
\begin{enumerate}
\item $\{p_i\}_{i=0}^{n-t_1}$ are linearly independent over ${\rm Mat}_2(\C)$.

\item The following equations hold:
\begin{align}
(X-\theta_i)p_i &=p_{i+1} \qquad (0\leq i\leq n-t_1-1),
\qquad 
(X-\theta_n) p_{n-t_1}=0,
\label{e:t1<=n<t1+t2-1}
\\
(Z-\theta_i^*) p_i &=\varphi_i p_{i-1} \qquad (1\leq i\leq n-t_1),
\qquad 
(Z-\theta_0^*) p_0=0,
\label{e:t1<=n<t1+t2-2}
\end{align} 
where
\begin{align*}
\theta_i &=(-1)^i
\textstyle(
k_2+k_3+i+\frac{1}{2}
)
\qquad 
(0\leq i\leq n-t_1),
\\
\theta_i^* &=(-1)^{i+1}
\textstyle(
n+k_1+k_2-i+\frac{1}{2}
)
\qquad 
(0\leq i\leq n-t_1),
\\
\varphi_i
&=
\left\{
\begin{array}{ll}
i(n-i+1)
\qquad
\hbox{if $i$ is even},
\\
(i+2k_3)(n-i+2k_1+1)
\qquad
\hbox{if $i$ is odd}
\end{array}
\right.
\qquad 
(1\leq i\leq n-t_1).
\end{align*}

\item $\D(p_i)=0$ for all $i=0,1,\ldots,n-t_1$. 
\end{enumerate}
\end{prop}
\begin{proof}  
(i): Let $i$ be an integer with $0\leq i\leq n-t_1$.
By construction the coefficient of $x_3^h$ in $p_i$ is zero for all integers $h$ with $i<h\leq n-t_1$. Observe that the coefficient of $x_1^{n-i}x_3^i$ in $p_i$ is 
\begin{gather}\label{coeff:t1<=n<t1+t2}
(-1)^{-\frac{n}{2}}
\prod_{h=n-i+1}^n m_1^{(h)}
\prod_{h=1}^{n-t_1} m_2^{(h)}
\times 
\left\{
\begin{array}{ll}
\sigma_1 \qquad 
&\hbox{if $i$ is odd},
\\
\sigma_3 \qquad 
&\hbox{if $i$ is even}.
\end{array}
\right.
\end{gather}
Since $i\leq n-t_1$ the scalar $\prod\limits_{h=n-i+1}^n m_1^{(h)}$ is nonzero. 
Since $n<t_1+t_2$ the scalar $\prod\limits_{h=1}^{n-t_1} m_2^{(h)}$ is nonzero. 
Hence the matrix (\ref{coeff:t1<=n<t1+t2}) is nonsingular. By the above comments the part (i) follows.

(ii): 
Let $i$ be a nonnegative integer and let $j$ be an integer with $t_1\leq j\leq n-i-1$. 
Using Theorem \ref{thm:BImodule_Mn}(i) yields that the coefficient of 
$[x_1]^n_j [x_2]^{n-t_1}_{n-i-j-1} [x_3]^{i+1}_{i+1}$ in $(X-\theta_i)p_i$ is equal to 
$$
(-1)^{\frac{j-i-1}{2}}\sigma_1 c_{i,i,j}=(-1)^{\frac{-i-j-1}{2}}c_{i+1,i+1,j}.
$$
Now let $h,i,j$ denote three integers with $0\leq h\leq i\leq n-t_1$ and $t_1\leq j\leq n-h$. Using Theorem \ref{thm:BImodule_Mn}(i) yields that the coefficient of 
$[x_1]^n_j [x_2]^{n-t_1}_{n-h-j} [x_3]^i_h$ in $(X-\theta_i)p_i$ is equal to $(-1)^{\frac{j-h}{2}}$ times the sum of 
\begin{align*}
&
m_3^{(h)}
{\left\lfloor \frac{n-i-j}{2}\right\rfloor+\left\lfloor\frac{i-h+1}{2}\right\rfloor \choose \left\lfloor\frac{i-h+1}{2}\right\rfloor}
 \sigma_1 c_{h-1,i,j},
\\
&
m_2^{(n-h-j)}
{\left\lfloor \frac{n-i-j}{2}\right\rfloor+\left\lfloor\frac{i-h-1}{2}\right\rfloor \choose \left\lfloor\frac{i-h-1}{2}\right\rfloor}
 \sigma_1 c_{h+1,i,j},
\\
&
{\textstyle(
(-1)^{h+1} k_3
+
(-1)^{h+j} k_2
-
(-1)^j\theta_i
-
\frac{1}{2}
)}
{\left\lfloor \frac{n-i-j}{2}\right\rfloor+\left\lfloor\frac{i-h}{2}\right\rfloor \choose \left\lfloor\frac{i-h}{2}\right\rfloor}
c_{h,i,j}.
\end{align*}
It is straightforward to verify that the sum of the above three terms is equal to 
$$
(-1)^j
m_3^{(i+1)}
{\left\lfloor \frac{n-i-j-1}{2}\right\rfloor+\left\lfloor\frac{i-h+1}{2}\right\rfloor \choose \left\lfloor\frac{i-h+1}{2}\right\rfloor}
c_{h,i+1,j}. 
$$
By the above comments the equations given in (\ref{e:t1<=n<t1+t2-1}) follow.

Using Theorem \ref{thm:BImodule_Mn}(i) yields that the coefficient of 
$[x_1]^n_j [x_2]^{n-t_1}_{n-h-j} [x_3]^i_h$ in $(Z-\theta_i^*)p_i$ is equal to $(-1)^{\frac{h-j}{2}+1}$ times the sum of 
\begin{align*}
&
m_1^{(j)}
{\left\lfloor \frac{n-i-j+1}{2}\right\rfloor+\left\lfloor\frac{i-h}{2}\right\rfloor \choose \left\lfloor\frac{i-h}{2}\right\rfloor}
 \sigma_3 c_{h,i,j-1},
\\
&
m_2^{(n-h-j)}
{\left\lfloor \frac{n-i-j-1}{2}\right\rfloor+\left\lfloor\frac{i-h}{2}\right\rfloor \choose \left\lfloor\frac{i-h}{2}\right\rfloor}
 \sigma_3 c_{h,i,j+1},
\\
&
{\textstyle(
(-1)^j k_1
-
(-1)^{h+j} k_2
+
(-1)^h
\theta_i^*
+
\frac{1}{2}
)}
{\left\lfloor \frac{n-i-j}{2}\right\rfloor+\left\lfloor\frac{i-h}{2}\right\rfloor \choose \left\lfloor\frac{i-h}{2}\right\rfloor}
c_{h,i,j}.
\end{align*}
It is straightforward to verify that the sum of the above three terms is equal to $(-1)^{h+1}$ times 
$$
\left\{
\begin{array}{ll}
\begin{pmatrix}
0 &0
\\
0 &0
\end{pmatrix}
\qquad &\hbox{if $h=i$},
\\
(n-i+1)
\displaystyle{\left\lfloor \frac{n-i-j+1}{2}\right\rfloor+\left\lfloor\frac{i-h-1}{2}\right\rfloor \choose \left\lfloor\frac{i-h-1}{2}\right\rfloor}
c_{h,i-1,j}
\qquad &\hbox{if $h<i$ and $i$ is even},
\\
(n-i+2k_1+1)
\displaystyle{\left\lfloor \frac{n-i-j+1}{2}\right\rfloor+\left\lfloor\frac{i-h-1}{2}\right\rfloor \choose \left\lfloor\frac{i-h-1}{2}\right\rfloor}
c_{h,i-1,j}
\qquad &\hbox{if $h<i$ and $i$ is odd}.
\end{array}
\right.
$$
By the above comments the equations given in (\ref{e:t1<=n<t1+t2-2}) follow.

(iii): It is routine to verify that $\D(p_0)=0$. Combined with (ii) the statement (iii) follows.
\end{proof}

\begin{lem}\label{lem:basis:t1<=n<t1+t2(odd)}
Suppose that $n$ is an odd integer with $t_1+t_3\leq n< t_2$. 
Let $\{p_i\}_{i=0}^{n-t_1}$ be as in Proposition \ref{prop:t1<=n<t1+t2(odd)}. Then the following hold:
\begin{enumerate}
\item $\M_n(x_1)\cap \M_n(x_3)$ has the basis 
\begin{align}
p_i\cdot 
\begin{pmatrix}
1
\\
0
\end{pmatrix}
\otimes
1 
\qquad 
(t_3\leq i\leq n-t_1), 
\label{M13basis:n<t1+t2(odd)-1}
\\ 
p_i\cdot 
\begin{pmatrix}
0
\\
1
\end{pmatrix}
\otimes
1 
\qquad 
(t_3\leq i\leq n-t_1).
\label{M13basis:n<t1+t2(odd)-2}
\end{align}

\item $\M_n(x_1)/\M_n(x_1)\cap \M_n(x_3)$ has the basis
\begin{align}
p_i\cdot 
\begin{pmatrix}
1
\\
0
\end{pmatrix}
\otimes
1 
+
\M_n(x_1)\cap \M_n(x_3)
\qquad 
(0\leq i\leq t_3-1), 
\label{M1/M13basis:n<t1+t2(odd)-1}
\\ 
p_i\cdot 
\begin{pmatrix}
0
\\
1
\end{pmatrix}
\otimes
1 
+
\M_n(x_1)\cap \M_n(x_3)
\qquad 
(0\leq i\leq t_3-1).
\label{M1/M13basis:n<t1+t2(odd)-2}
\end{align}

\item $\M_n/\M_n(x_3)$ has the basis
\begin{align}
p_i\cdot 
\begin{pmatrix}
1
\\
0
\end{pmatrix}
\otimes
1 
+
\M_n(x_3)
\qquad 
(0\leq i\leq t_3-1), 
\label{M/M3basis:n<t1+t2(odd)-1}
\\ 
p_i\cdot 
\begin{pmatrix}
0
\\
1
\end{pmatrix}
\otimes
1 
+
\M_n(x_3)
\qquad 
(0\leq i\leq t_3-1).
\label{M/M3basis:n<t1+t2(odd)-2}
\end{align}
\end{enumerate}
\end{lem}
\begin{proof}
(i): The linear independence of (\ref{M13basis:n<t1+t2(odd)-1}) and (\ref{M13basis:n<t1+t2(odd)-2}) follows from Proposition \ref{prop:t1<=n<t1+t2(odd)}(i). Since (\ref{M13basis:n<t1+t2(odd)-1}) and (\ref{M13basis:n<t1+t2(odd)-2}) are in 
$
\C^2\otimes\left(\bigoplus\limits_{i=t_1}^n x_1^i\cdot \R[x_2,x_3]_{n-i}
\cap 
\bigoplus\limits_{i=t_3}^n x_3^i\cdot \R[x_1,x_2]_{n-i}
\right)$,  
it follows from Proposition \ref{prop:t1<=n<t1+t2(odd)}(iii) that  (\ref{M13basis:n<t1+t2(odd)-1}) and (\ref{M13basis:n<t1+t2(odd)-2}) are in $\M_n(x_1)\cap \M_n(x_3)$. Combined with Proposition \ref{prop:dimM(x12)}(iii) the statement (i) follows.

(ii): The linear independence of (\ref{M1/M13basis:n<t1+t2(odd)-1}) and (\ref{M1/M13basis:n<t1+t2(odd)-2}) follows from Proposition \ref{prop:t1<=n<t1+t2(odd)}(i). 
By Proposition \ref{prop:t1<=n<t1+t2(odd)}(iii) the cosets (\ref{M1/M13basis:n<t1+t2(odd)-1}) and (\ref{M1/M13basis:n<t1+t2(odd)-2}) are in 
$
\M_n(x_1)/\M_n(x_1)\cap \M_n(x_3).
$  
By Propositions \ref{prop:dimM(x1)}(i) and \ref{prop:dimM(x12)}(iii) the dimension of $\M_n(x_1)/\M_n(x_1)\cap \M_n(x_3)$ is $2t_3$. The statement (ii) follows.

(iii): The linear independence of (\ref{M/M3basis:n<t1+t2(odd)-1}) and (\ref{M/M3basis:n<t1+t2(odd)-2}) follows from Proposition \ref{prop:t1<=n<t1+t2(odd)}(i). By Proposition \ref{prop:t1<=n<t1+t2(odd)}(iii) the cosets (\ref{M/M3basis:n<t1+t2(odd)-1}) and (\ref{M/M3basis:n<t1+t2(odd)-2}) are in  $\M_n/\M_n(x_3)$. By Theorem \ref{thm:dimM=2(n+1)} and Proposition \ref{prop:dimM(x1)}(iii) 
the dimension of $\M_n/\M_n(x_3)$ is $2t_3$.
The statement (iii) follows.
\end{proof}

\begin{prop}\label{prop:t3<=n<t2+t3(odd)}
Suppose that $n$ is an odd integer with $t_3\leq n<t_2+t_3$. 
Let 
$$
p_i
=
\sum_{h=0}^i
(-1)^{\frac{h}{2}}
\sum_{j=t_3}^{n-h}
(-1)^{\frac{j}{2}}
{\left\lfloor \frac{n-i-j}{2}\right\rfloor+\left\lfloor\frac{i-h}{2}\right\rfloor \choose \left\lfloor\frac{i-h}{2}\right\rfloor}
c_{h,i,j}
[x_3]^n_j[x_1]^i_h[x_2]^{n-t_3}_{n-h-j}
$$
for all $i=0,1,\ldots,n-t_3$ where
\begin{align*}
c_{h,i,j}&=
\left\{
\begin{array}{ll}
\sigma_1^j \sigma_3
\qquad &\hbox{if $h$ is odd and $i$ is odd},
\\
\sigma_1^j
\qquad &\hbox{if $h$ is even and $i$ is even},
\\
(-1)^\frac{3}{2}
\sigma_2 
\qquad &\hbox{if $h$ is odd, $i$ is even and $j$ is odd},
\\
-\begin{pmatrix}
1 &0
\\
0 &1
\end{pmatrix}
\qquad &\hbox{if $h$ is even, $i$ is odd and $j$ is even},
\\
\begin{pmatrix}
0 &0
\\
0 &0
\end{pmatrix}
\qquad &\hbox{else}
\end{array}
\right.
\end{align*}
for any integers $h,i,j$. 
Then the following hold:
\begin{enumerate}
\item $\{p_i\}_{i=0}^{n-t_3}$ are linearly independent over ${\rm Mat}_2(\C)$.

\item The following equations hold:
\begin{align}
(Z-\theta_i)p_i &=p_{i+1} \qquad (0\leq i\leq n-t_3-1),
\qquad 
(Z-\theta_n) p_{n-t_3}=0,
\label{e:t3<=n<t2+t3-1}
\\
(X-\theta_i^*) p_i &=\varphi_i p_{i-1} \qquad (1\leq i\leq n-t_3),
\qquad 
(X-\theta_0^*) p_0=0,
\label{e:t3<=n<t2+t3-2}
\end{align}
where
\begin{align*}
\theta_i &=(-1)^i
\textstyle(
k_1+k_2+i+\frac{1}{2}
)
\qquad 
 (0\leq i\leq n-t_3),
\\
\theta_i^* &=(-1)^{i+1}
\textstyle(
k_2+k_3+n-i+\frac{1}{2}
)
\qquad 
 (0\leq i\leq n-t_3),
\\
\varphi_i
&=
\left\{
\begin{array}{ll}
i(n-i+1)
\qquad
\hbox{if $i$ is even},
\\
(i+2k_1)(n-i+2k_3+1)
\qquad
\hbox{if $i$ is odd}
\end{array}
\right.
\qquad 
(1\leq i\leq n-t_3).
\end{align*}

\item $\D(p_i)=0$ for all $i=0,1,\ldots,n-t_3$. 
\end{enumerate}
\end{prop}
\begin{proof}  
(i): Let $i$ be an integer with $0\leq i\leq n-t_3$.
By construction the coefficient of $x_1^h$ in $p_i$ is zero for all integers $h$ with $i<h\leq n-t_3$. Observe that the coefficient of $x_3^{n-i} x_1^i$ in $p_i$ is 
\begin{gather}\label{coeff:t3<=n<t2+t3}
(-1)^\frac{n}{2}
\prod_{h=n-i+1}^n m_3^{(h)}
\prod_{h=1}^{n-t_3} m_2^{(h)}
\times 
\left\{
\begin{array}{ll}
\sigma_3 \qquad 
&\hbox{if $i$ is odd},
\\
\sigma_1 \qquad 
&\hbox{if $i$ is even}.
\end{array}
\right.
\end{gather}
Since $i\leq n-t_3$ the scalar $\prod\limits_{h=n-i+1}^n m_3^{(h)}$ is nonzero. 
Since $n<t_2+t_3$ the scalar $\prod\limits_{h=1}^{n-t_3} m_2^{(h)}$ is nonzero. 
Hence the matrix (\ref{coeff:t3<=n<t2+t3}) is nonsingular. By the above comments the part (i) follows.

(ii): 
Let $i$ be a nonnegative integer and let $j$ be an integer with $t_3\leq j\leq n-i-1$. 
Using Theorem \ref{thm:BImodule_Mn}(i) yields that the coefficient of 
$[x_3]^n_j  [x_1]^{i+1}_{i+1} [x_2]^{n-t_3}_{n-i-j-1}$ in $(Z-\theta_i)p_i$ is equal to 
$$
(-1)^{\frac{i-j+1}{2}}\sigma_3 c_{i,i,j}=
(-1)^{\frac{i+j+1}{2}}c_{i+1,i+1,j}.
$$
Now let $h,i,j$ denote three integers with $0\leq h\leq i\leq n-t_3$ and $t_3\leq j\leq n-h$. Using Theorem \ref{thm:BImodule_Mn}(i) yields that the coefficient of 
$[x_3]^n_j [x_1]^i_h [x_2]^{n-t_3}_{n-h-j} $ in $(Z-\theta_i)p_i$ is equal to $(-1)^{\frac{h-j}{2}}$ times the sum of 
\begin{align*}
&
m_1^{(h)}
{\left\lfloor \frac{n-i-j}{2}\right\rfloor+\left\lfloor\frac{i-h+1}{2}\right\rfloor \choose \left\lfloor\frac{i-h+1}{2}\right\rfloor}
 \sigma_3 c_{h-1,i,j},
\\
&
m_2^{(n-h-j)}
{\left\lfloor \frac{n-i-j}{2}\right\rfloor+\left\lfloor\frac{i-h-1}{2}\right\rfloor \choose \left\lfloor\frac{i-h-1}{2}\right\rfloor}
 \sigma_3 c_{h+1,i,j},
\\
&
{\textstyle(
(-1)^{h+1} k_1
+
(-1)^{h+j} k_2
-
(-1)^j \theta_i
-
\frac{1}{2}
)}
{\left\lfloor \frac{n-i-j}{2}\right\rfloor+\left\lfloor\frac{i-h}{2}\right\rfloor \choose \left\lfloor\frac{i-h}{2}\right\rfloor}
c_{h,i,j}.
\end{align*}
It is straightforward to verify that the sum of the above three terms is equal to 
$$
(-1)^j
m_1^{(i+1)}
{\left\lfloor \frac{n-i-j-1}{2}\right\rfloor+\left\lfloor\frac{i-h+1}{2}\right\rfloor \choose \left\lfloor\frac{i-h+1}{2}\right\rfloor}
c_{h,i+1,j}. 
$$
By the above comments the equations given in (\ref{e:t3<=n<t2+t3-1}) follow.

Using Theorem \ref{thm:BImodule_Mn}(i) yields that the coefficient of 
$[x_3]^n_j [x_1]^i_h [x_2]^{n-t_3}_{n-h-j}$ in $(X-\theta_i^*)p_i$ is equal to $(-1)^{\frac{j-h}{2}+1}$ times the sum of 
\begin{align*}
&
m_3^{(j)}
{\left\lfloor \frac{n-i-j+1}{2}\right\rfloor+\left\lfloor\frac{i-h}{2}\right\rfloor \choose \left\lfloor\frac{i-h}{2}\right\rfloor}
 \sigma_1 c_{h,i,j-1},
\\
&
m_2^{(n-h-j)}
{\left\lfloor \frac{n-i-j-1}{2}\right\rfloor+\left\lfloor\frac{i-h}{2}\right\rfloor \choose \left\lfloor\frac{i-h}{2}\right\rfloor}
 \sigma_1 c_{h,i,j+1},
\\
&
{\textstyle(
(-1)^j k_3
-
(-1)^{h+j} k_2
+
(-1)^h
\theta_i^*
+
\frac{1}{2}
)}
{\left\lfloor \frac{n-i-j}{2}\right\rfloor+\left\lfloor\frac{i-h}{2}\right\rfloor \choose \left\lfloor\frac{i-h}{2}\right\rfloor}
c_{h,i,j}.
\end{align*}
It is straightforward to verify that the sum of the above three terms is equal to $(-1)^{h+1}$ times 
$$
\left\{
\begin{array}{ll}
\begin{pmatrix}
0 &0
\\
0 &0
\end{pmatrix}
\qquad &\hbox{if $h=i$},
\\
(n-i+1)
\displaystyle{\left\lfloor \frac{n-i-j+1}{2}\right\rfloor+\left\lfloor\frac{i-h-1}{2}\right\rfloor \choose \left\lfloor\frac{i-h-1}{2}\right\rfloor}
c_{h,i-1,j}
\qquad &\hbox{if $h<i$ and $i$ is even},
\\
(n-i+2k_3+1)
\displaystyle{\left\lfloor \frac{n-i-j+1}{2}\right\rfloor+\left\lfloor\frac{i-h-1}{2}\right\rfloor \choose \left\lfloor\frac{i-h-1}{2}\right\rfloor}
c_{h,i-1,j}
\qquad &\hbox{if $h<i$ and $i$ is odd}.
\end{array}
\right.
$$
By the above comments the equations given in (\ref{e:t3<=n<t2+t3-2}) follow.

(iii): It is routine to verify that $\D(p_0)=0$. Combined with (ii) the statement (iii) follows.
\end{proof}

\begin{lem}\label{lem:basis:n<t2+t3(odd)}
Suppose that $n$ is an odd integer with $t_1+t_3\leq n< t_2$. 
Let $\{p_i\}_{i=0}^{n-t_3}$ be as in Proposition \ref{prop:t3<=n<t2+t3(odd)}. Then the following hold:
\begin{enumerate}
\item $\M_n(x_1)\cap \M_n(x_3)$ has the basis 
\begin{align}
p_i\cdot 
\begin{pmatrix}
1
\\
0
\end{pmatrix}
\otimes
1 
\qquad 
(t_1\leq i\leq n-t_3), 
\label{M13basis:n<t2+t3(odd)-1}
\\ 
p_i\cdot 
\begin{pmatrix}
0
\\
1
\end{pmatrix}
\otimes
1 
\qquad 
(t_1\leq i\leq n-t_3).
\label{M13basis:n<t2+t3(odd)-2}
\end{align}

\item $\M_n(x_3)/\M_n(x_1)\cap \M_n(x_3)$ has the basis
\begin{align}
p_i\cdot 
\begin{pmatrix}
1
\\
0
\end{pmatrix}
\otimes
1 
+
\M_n(x_1)\cap \M_n(x_3)
\qquad 
(0\leq i\leq t_1-1), 
\label{M3/M13basis:n<t2+t3(odd)-1}
\\ 
p_i\cdot 
\begin{pmatrix}
0
\\
1
\end{pmatrix}
\otimes
1 
+
\M_n(x_1)\cap \M_n(x_3)
\qquad 
(0\leq i\leq t_1-1).
\label{M3/M13basis:n<t2+t3(odd)-2}
\end{align}

\item $\M_n/\M_n(x_1)$ has the basis
\begin{align}
p_i\cdot 
\begin{pmatrix}
1
\\
0
\end{pmatrix}
\otimes
1 
+
\M_n(x_1)
\qquad 
(0\leq i\leq t_1-1), 
\label{M/M1basis:n<t2+t3(odd)-1}
\\ 
p_i\cdot 
\begin{pmatrix}
0
\\
1
\end{pmatrix}
\otimes
1 
+
\M_n(x_1)
\qquad 
(0\leq i\leq t_1-1).
\label{M/M1basis:n<t2+t3(odd)-2}
\end{align}
\end{enumerate}
\end{lem}
\begin{proof}
(i): The linear independence of (\ref{M13basis:n<t2+t3(odd)-1}) and (\ref{M13basis:n<t2+t3(odd)-2}) follows from Proposition \ref{prop:t3<=n<t2+t3(odd)}(i). Since (\ref{M13basis:n<t2+t3(odd)-1}) and (\ref{M13basis:n<t2+t3(odd)-2}) are in 
$
\C^2\otimes\left(\bigoplus\limits_{i=t_1}^n x_1^i\cdot \R[x_2,x_3]_{n-i}
\cap 
\bigoplus\limits_{i=t_3}^n x_3^i\cdot \R[x_1,x_2]_{n-i}
\right)$,  
it follows from Proposition \ref{prop:t3<=n<t2+t3(odd)}(iii) that  (\ref{M13basis:n<t2+t3(odd)-1}) and (\ref{M13basis:n<t2+t3(odd)-2}) are in $\M_n(x_1)\cap \M_n(x_3)$. Combined with Proposition \ref{prop:dimM(x12)}(iii) the statement (i) follows.

(ii): The linear independence of (\ref{M3/M13basis:n<t2+t3(odd)-1}) and (\ref{M3/M13basis:n<t2+t3(odd)-2}) follows from Proposition \ref{prop:t3<=n<t2+t3(odd)}(i). 
By Proposition \ref{prop:t3<=n<t2+t3(odd)}(iii) the cosets (\ref{M3/M13basis:n<t2+t3(odd)-1}) and (\ref{M3/M13basis:n<t2+t3(odd)-2}) are in 
$
\M_n(x_3)/\M_n(x_1)\cap \M_n(x_3).
$  
By Propositions \ref{prop:dimM(x1)}(iii) and \ref{prop:dimM(x12)}(iii) the dimension of $\M_n(x_3)/\M_n(x_1)\cap \M_n(x_3)$ is $2t_1$. The statement (ii) follows.

(iii): The linear independence of (\ref{M/M1basis:n<t2+t3(odd)-1}) and (\ref{M/M1basis:n<t2+t3(odd)-2}) follows from Proposition \ref{prop:t3<=n<t2+t3(odd)}(i). By Proposition \ref{prop:t1<=n<t1+t2(odd)}(iii) the cosets (\ref{M/M1basis:n<t2+t3(odd)-1}) and (\ref{M/M1basis:n<t2+t3(odd)-2}) are in  $\M_n/\M_n(x_1)$. By Theorem \ref{thm:dimM=2(n+1)} and Proposition \ref{prop:dimM(x1)}(i) 
the dimension of $\M_n/\M_n(x_1)$ is $2t_1$.
The statement (iii) follows.
\end{proof}

\begin{thm}\label{thm:t1+t3<=n<t2(odd)}
Suppose that $n$ is an odd integer with $t_1+t_3\leq n< t_2$. Then the following hold:
\begin{enumerate}
\item The $\BI$-module $\M_n(x_1)\cap \M_n(x_3)$ 
is isomorphic to a direct sum of two copies of 
\begin{gather}\label{BImodule:t1+t3<=n<t2(odd)-1}
E_{n-t_1-t_3}
\textstyle(
k_1+k_2+\frac{n+1}{2},
-k_2-k_3-\frac{n+1}{2},
\frac{n+1}{2}
)^{((-1,-1),(2\,3))}.
\end{gather}

\item The $\BI$-modules $\M_n(x_1)/\M_n(x_1)\cap \M_n(x_3)$ and $\M_n/\M_n(x_3)$ are isomorphic to a direct sum of two copies of 
\begin{gather}\label{BImodule:t1+t3<=n<t2(odd)-2}
O_{t_3-1}(k_2,-k_1-k_2-k_3-n-1,k_1)^{(2\,3)}.
\end{gather}

\item The $\BI$-modules $\M_n(x_3)/\M_n(x_1)\cap \M_n(x_3)$ and $\M_n/\M_n(x_1)$ are isomorphic to a direct sum of two copies of 
\begin{gather}\label{BImodule:t1+t3<=n<t2(odd)-3}
O_{t_1-1}(k_2,-k_1-k_2-k_3-n-1,k_3)^{(1\,2\,3)}.
\end{gather}
\end{enumerate}
Moreover, if $k_2$ is nonnegative then the $\BI$-modules (\ref{BImodule:t1+t3<=n<t2(odd)-1})--(\ref{BImodule:t1+t3<=n<t2(odd)-3}) are irreducible and they are isomorphic to 
\begin{align}
&E_{n-t_1-t_3}
\textstyle(
k_1+k_2+\frac{n+1}{2},
\frac{n+1}{2},
k_2+k_3+\frac{n+1}{2}
)^{(-1,1)},
\label{BImodule:t1+t3<=n<t2(odd)-1'}
\\
&O_{t_3-1}(k_2,k_1,-k_1-k_2-k_3-n-1),
\label{BImodule:t1+t3<=n<t2(odd)-2'}
\\
&
O_{t_1-1}(-k_1-k_2-k_3-n-1,k_3,k_2),
\label{BImodule:t1+t3<=n<t2(odd)-3'}
\end{align}
respectively.
\end{thm}
\begin{proof}
(i): Let $V$ and $V'$ denote the subspaces of $\M_n$ spanned by (\ref{M13basis:n<t1+t2(odd)-1}) and (\ref{M13basis:n<t1+t2(odd)-2}), respectively. It follows from Lemma \ref{lem:basis:t1<=n<t1+t2(odd)}(i) that 
$
\M_n(x_1)\cap \M_n(x_3)=V\oplus V'$. 
It follows from Theorem \ref{thm:BImodule_Mn}(ii) and Proposition \ref{prop:t1<=n<t1+t2(odd)}(ii) that $V$ and $V'$ are two isomorphic $\BI$-submodules of $\M_n$. Using Proposition \ref{prop:Ed} yields that both are isomorphic to (\ref{BImodule:t1+t3<=n<t2(odd)-1}). 

(ii): 
Let $V$ and $V'$ denote the subspaces of $\M_n(x_1)/\M_n(x_1)\cap \M_n(x_3)$ spanned by (\ref{M1/M13basis:n<t1+t2(odd)-1}) and (\ref{M1/M13basis:n<t1+t2(odd)-2}), respectively. It follows from Lemma \ref{lem:basis:t1<=n<t1+t2(odd)}(ii) that 
$
\M_n(x_1)/\M_n(x_1)\cap \M_n(x_3)=V\oplus V'$. 
It follows from Theorem \ref{thm:BImodule_Mn}(ii) and Proposition \ref{prop:t1<=n<t1+t2(odd)}(ii) that $V$ and $V'$ are two isomorphic $\BI$-submodules of $\M_n(x_1)/\M_n(x_1)\cap \M_n(x_3)$. Using Proposition \ref{prop:Od} yields that both are isomorphic to (\ref{BImodule:t1+t3<=n<t2(odd)-2}). 

Let $V$ and $V'$ denote the subspaces of $\M_n/\M_n(x_3)$ spanned by (\ref{M/M3basis:n<t1+t2(odd)-1}) and (\ref{M/M3basis:n<t1+t2(odd)-2}), respectively. It follows from Lemma \ref{lem:basis:t1<=n<t1+t2(odd)}(iii) that 
$
\M_n/\M_n(x_3)=V\oplus V'$. 
It follows from Theorem \ref{thm:BImodule_Mn}(ii) and Proposition \ref{prop:t1<=n<t1+t2(odd)}(ii) that $V$ and $V'$ are two isomorphic $\BI$-submodules of $\M_n/\M_n(x_3)$. Using Proposition \ref{prop:Od} yields that both are isomorphic to (\ref{BImodule:t1+t3<=n<t2(odd)-2}). 

(iii): 
Let $V$ and $V'$ denote the subspaces of $\M_n(x_3)/\M_n(x_1)\cap \M_n(x_3)$ spanned by (\ref{M3/M13basis:n<t2+t3(odd)-1}) and (\ref{M3/M13basis:n<t2+t3(odd)-2}), respectively. It follows from Lemma \ref{lem:basis:n<t2+t3(odd)}(ii) that 
$
\M_n(x_3)/\M_n(x_1)\cap \M_n(x_3)=V\oplus V'$. 
It follows from Theorem \ref{thm:BImodule_Mn}(ii) and Proposition \ref{prop:t3<=n<t2+t3(odd)}(ii) that $V$ and $V'$ are two isomorphic $\BI$-submodules of $\M_n(x_3)/\M_n(x_1)\cap \M_n(x_3)$. Using Proposition \ref{prop:Od} yields that both are isomorphic to (\ref{BImodule:t1+t3<=n<t2(odd)-3}). 

Let $V$ and $V'$ denote the subspaces of $\M_n/\M_n(x_1)$ spanned by (\ref{M/M1basis:n<t2+t3(odd)-1}) and (\ref{M/M1basis:n<t2+t3(odd)-2}), respectively. It follows from Lemma \ref{lem:basis:n<t2+t3(odd)}(iii) that 
$
\M_n/\M_n(x_1)=V\oplus V'$. 
It follows from Theorem \ref{thm:BImodule_Mn}(ii) and Proposition \ref{prop:t3<=n<t2+t3(odd)}(ii) that $V$ and $V'$ are two isomorphic $\BI$-submodules of $\M_n/\M_n(x_1)$. Using Proposition \ref{prop:Od} yields that both are isomorphic to (\ref{BImodule:t1+t3<=n<t2(odd)-3}).

Suppose that $k_2$ is nonnegative. Using Theorem \ref{thm:irr_E} yields that the $\BI$-module (\ref{BImodule:t1+t3<=n<t2(odd)-1}) is irreducible. Using Theorem \ref{thm:irr_O} yields that the $\BI$-modules (\ref{BImodule:t1+t3<=n<t2(odd)-2}) and (\ref{BImodule:t1+t3<=n<t2(odd)-3}) are irreducible. 
Using Theorem \ref{thm:onto2_E} yields that the $\BI$-module (\ref{BImodule:t1+t3<=n<t2(odd)-1}) is isomorphic to  (\ref{BImodule:t1+t3<=n<t2(odd)-1'}). 
Using Theorem \ref{thm:onto2_O} yields that the $\BI$-modules (\ref{BImodule:t1+t3<=n<t2(odd)-2}) and (\ref{BImodule:t1+t3<=n<t2(odd)-3}) are isomorphic to  (\ref{BImodule:t1+t3<=n<t2(odd)-2'}) and (\ref{BImodule:t1+t3<=n<t2(odd)-3'}), respectively. 
\end{proof}

By similar arguments as in the proof of Theorem \ref{thm:t1+t3<=n<t2(odd)} we have the following results:

\begin{thm}\label{thm:t1+t2<=n<t3(odd)}
Suppose that $n$ is an odd integer with $t_1+t_2\leq n<t_3$. Then the following hold:
\begin{enumerate}
\item The $\BI$-module $\M_n(x_1)\cap \M_n(x_2)$ 
is isomorphic to a direct sum of two copies of 
\begin{gather}\label{t1+t2<=n<t3(odd)-1}
E_{n-t_1-t_2}
\textstyle(
k_2+k_3+\frac{n+1}{2},
-k_1-k_3-\frac{n+1}{2},
\frac{n+1}{2}
)^{((-1,-1),(1\,2))}.
\end{gather}

\item The $\BI$-modules $\M_n(x_2)/\M_n(x_1)\cap \M_n(x_2)$ and $\M_n/\M_n(x_1)$ are isomorphic to a direct sum of two copies of 
\begin{gather}\label{t1+t2<=n<t3(odd)-2}
O_{t_1-1}(k_3,-k_1-k_2-k_3-n-1,k_2)^{(1\,2)}.
\end{gather}

\item The $\BI$-modules $\M_n(x_1)/\M_n(x_1)\cap \M_n(x_2)$ and $\M_n/\M_n(x_2)$
are isomorphic to a direct sum of two copies of 
\begin{gather}\label{t1+t2<=n<t3(odd)-3}
O_{t_2-1}(k_3,-k_1-k_2-k_3-n-1,k_1).
\end{gather}
\end{enumerate}
Moreover, if $k_3$ is nonnegative then the $\BI$-modules (\ref{t1+t2<=n<t3(odd)-1})--(\ref{t1+t2<=n<t3(odd)-3}) are irreducible and (\ref{t1+t2<=n<t3(odd)-1}), (\ref{t1+t2<=n<t3(odd)-2}) are isomorphic to 
\begin{align*}
&E_{n-t_1-t_2}
\textstyle(
k_1+k_3+\frac{n+1}{2},
k_2+k_3+\frac{n+1}{2},
\frac{n+1}{2}
)^{(-1,-1)},
\\
&O_{t_1-1}(-k_1-k_2-k_3-n-1,k_3,k_2),
\end{align*}
respectively.
\end{thm}

\begin{thm}\label{thm:t2+t3<=n<t1(odd)}
Suppose that $n$ is an odd integer with $t_2+t_3\leq n<t_1$. Then the following hold:
\begin{enumerate}
\item The $\BI$-module $\M_n(x_2)\cap \M_n(x_3)$ is isomorphic to a direct sum of two copies of 
\begin{gather}\label{t2+t3<=n<t1(odd)-1}
E_{n-t_2-t_3}
\textstyle(
k_1+k_3+\frac{n+1}{2},
-k_1-k_2-\frac{n+1}{2},
\frac{n+1}{2}
)^{((-1,-1),(1\,3))}.
\end{gather}

\item The $\BI$-modules $\M_n(x_3)/\M_n(x_2)\cap \M_n(x_3)$ and $\M_n/\M_n(x_2)$ are isomorphic to a direct sum of two copies of 
\begin{gather}\label{t2+t3<=n<t1(odd)-2}
O_{t_2-1}(k_1,-k_1-k_2-k_3-n-1,k_3)^{(1\,3)}.
\end{gather}

\item The $\BI$-modules $\M_n(x_2)/\M_n(x_2)\cap \M_n(x_3)$ and $\M_n/\M_n(x_3)$ are isomorphic to a direct sum of two copies of 
\begin{gather}\label{t2+t3<=n<t1(odd)-3}
O_{t_3-1}(k_1,-k_1-k_2-k_3-n-1,k_2)^{(1\,3\,2)}.
\end{gather}
\end{enumerate}
Moreover, if $k_1$ is nonnegative then the $\BI$-modules (\ref{t2+t3<=n<t1(odd)-1})--(\ref{t2+t3<=n<t1(odd)-3}) are irreducible and they are isomorphic to 
\begin{align*}
&E_{n-t_2-t_3}
\textstyle(
\frac{n+1}{2},
k_1+k_2+\frac{n+1}{2},
k_1+k_3+\frac{n+1}{2}
)^{(1,-1)},
\\
&O_{t_2-1}(k_3,-k_1-k_2-k_3-n-1,k_1),
\\
&O_{t_3-1}(k_2,k_1,-k_1-k_2-k_3-n-1),
\end{align*}
respectively.
\end{thm}

\begin{prop}\label{prop:t1<=n<t1+t2(even)}
Suppose that $n$ is an even integer with $t_1\leq n<t_1+t_2$. 
Let 
$$
p_i
=
\sum_{h=0}^i
(-1)^{\frac{h}{2}}
\sum_{j=t_1}^{n-h}
(-1)^{-\frac{j}{2}}
{\left\lfloor \frac{n-i-j}{2}\right\rfloor+\left\lfloor\frac{i-h}{2}\right\rfloor \choose \left\lfloor\frac{i-h}{2}\right\rfloor}
c_{h,i,j}
\otimes 
[x_1]^n_j 
[x_2]^{n-t_1}_{n-h-j}
[x_3]^i_h
$$
for all $i=0,1,\ldots,n-t_1$ where
\begin{align*}
c_{h,i,j}&=
\left\{
\begin{array}{ll}
\sigma_3^j \sigma_1 
\qquad &\hbox{if $h$ is odd and $i$ is odd},
\\
\sigma_3^j
\qquad &\hbox{if $h$ is even and $i$ is even},
\\
\sigma_1
\qquad &\hbox{if $h$ is odd, $i$ is even and $j$ is even},
\\
-\sigma_3
\qquad &\hbox{if $h$ is even, $i$ is odd and $j$ is odd},
\\
\begin{pmatrix}
0 &0
\\
0 &0
\end{pmatrix}
\qquad &\hbox{else}
\end{array}
\right.
\end{align*}
for any integers $h,i,j$. 
Then the following hold:
\begin{enumerate}
\item $\{p_i\}_{i=0}^{n-t_1}$ are linearly independent over ${\rm Mat}_2(\C)$.

\item The following equations hold:
\begin{align}
(X-\theta_i)p_i &=p_{i+1} \qquad (0\leq i\leq n-t_1-1),
\qquad 
(X-\theta_{n-t_1}) p_{n-t_1}=0,
\label{e:t1<=n<t1+t2(even)-1}
\\
(Z-\theta_i^*) p_i &=\varphi_i p_{i-1} \qquad (1\leq i\leq n-t_1),
\qquad 
(Z-\theta_0^*) p_0=0,
\label{e:t1<=n<t1+t2(even)-2}
\end{align}
where
\begin{align*}
\theta_i &=(-1)^i
\textstyle(
k_2+k_3+i+\frac{1}{2}
)
\qquad 
(0\leq i\leq n-t_1),
\\
\theta_i^* &=(-1)^i
\textstyle(
n+k_1+k_2-i+\frac{1}{2}
)
\qquad 
(0\leq i\leq n-t_1),
\\
\varphi_i
&=
\left\{
\begin{array}{ll}
i(i-n-2k_1-1)
\qquad
\hbox{if $i$ is even},
\\
(i+2k_3)(i-n-1)
\qquad
\hbox{if $i$ is odd}
\end{array}
\right.
\qquad 
(1\leq i\leq n-t_1).
\end{align*}

\item $\D(p_i)=0$ for all $i=0,1,\ldots,n-t_1$. 
\end{enumerate}
\end{prop}
\begin{proof}  
(i): Let $i$ be an integer with $0\leq i\leq n-t_1$.
By construction the coefficient of $x_3^h$ in $p_i$ is zero for all integers $h$ with $i<h\leq n-t_1$. Observe that the coefficient of $x_1^{n-i}x_3^i$ in $p_i$ is 
\begin{gather}\label{coeff:t1<=n<t1+t2(even)}
(-1)^{i-\frac{n}{2}}
\prod_{h=n-i+1}^n m_1^{(h)}
\prod_{h=1}^{n-t_1} m_2^{(h)}
\times 
\left\{
\begin{array}{ll}
(-1)^\frac{1}{2}
\sigma_2
\qquad 
&\hbox{if $i$ is odd},
\\
\begin{pmatrix}
1 &0
\\
0 &1
\end{pmatrix} 
\qquad 
&\hbox{if $i$ is even}.
\end{array}
\right.
\end{gather}
Since $i\leq n-t_1$ the scalar $\prod\limits_{h=n-i+1}^n m_1^{(h)}$ is nonzero. 
Since $n<t_1+t_2$ the scalar $\prod\limits_{h=1}^{n-t_1} m_2^{(h)}$ is nonzero. 
Hence the matrix (\ref{coeff:t1<=n<t1+t2(even)}) is nonsingular. By the above comments the part (i) follows.

(ii): 
Let $i$ be a nonnegative integer and let $j$ be an integer with $t_1\leq j\leq n-i-1$. 
Using Theorem \ref{thm:BImodule_Mn}(i) yields that the coefficient of 
$[x_1]^n_j [x_2]^{n-t_1}_{n-i-j-1} [x_3]^{i+1}_{i+1}$ in $(X-\theta_i)p_i$ is equal to 
$$
(-1)^{\frac{i+j+1}{2}}\sigma_1 c_{i,i,j}=(-1)^{\frac{i-j+1}{2}}c_{i+1,i+1,j}.
$$
Now let $h,i,j$ denote three integers with $0\leq h\leq i\leq n-t_1$ and $t_1\leq j\leq n-h$. Using Theorem \ref{thm:BImodule_Mn}(i) yields that the coefficient of 
$[x_1]^n_j [x_2]^{n-t_1}_{n-h-j} [x_3]^i_h$ in $(X-\theta_i)p_i$ is equal to $(-1)^{\frac{h+j}{2}}$ times the sum of 
\begin{align*}
&
m_3^{(h)}
{\left\lfloor \frac{n-i-j}{2}\right\rfloor+\left\lfloor\frac{i-h+1}{2}\right\rfloor \choose \left\lfloor\frac{i-h+1}{2}\right\rfloor}
 \sigma_1 c_{h-1,i,j},
\\
&
m_2^{(n-h-j)}
{\left\lfloor \frac{n-i-j}{2}\right\rfloor+\left\lfloor\frac{i-h-1}{2}\right\rfloor \choose \left\lfloor\frac{i-h-1}{2}\right\rfloor}
 \sigma_1 c_{h+1,i,j},
\\
&
{\textstyle(
(-1)^h k_3
+
(-1)^{h+j} k_2
-
(-1)^j\theta_i
+
\frac{1}{2}
)}
{\left\lfloor \frac{n-i-j}{2}\right\rfloor+\left\lfloor\frac{i-h}{2}\right\rfloor \choose \left\lfloor\frac{i-h}{2}\right\rfloor}
c_{h,i,j}.
\end{align*}
It is straightforward to verify that the sum of the above three terms is equal to 
$$
(-1)^j
m_3^{(i+1)}
{\left\lfloor \frac{n-i-j-1}{2}\right\rfloor+\left\lfloor\frac{i-h+1}{2}\right\rfloor \choose \left\lfloor\frac{i-h+1}{2}\right\rfloor}
c_{h,i+1,j}. 
$$
By the above comments the equations given in (\ref{e:t1<=n<t1+t2(even)-1}) follow.

Using Theorem \ref{thm:BImodule_Mn}(i) yields that the coefficient of 
$[x_1]^n_j [x_2]^{n-t_1}_{n-h-j} [x_3]^i_h$ in $(Z-\theta_i^*)p_i$ is equal to $(-1)^{-\frac{h+j}{2}}$ times the sum of 
\begin{align*}
&
m_1^{(j)}
{\left\lfloor \frac{n-i-j+1}{2}\right\rfloor+\left\lfloor\frac{i-h}{2}\right\rfloor \choose \left\lfloor\frac{i-h}{2}\right\rfloor}
 \sigma_3 c_{h,i,j-1},
\\
&
m_2^{(n-h-j)}
{\left\lfloor \frac{n-i-j-1}{2}\right\rfloor+\left\lfloor\frac{i-h}{2}\right\rfloor \choose \left\lfloor\frac{i-h}{2}\right\rfloor}
 \sigma_3 c_{h,i,j+1},
\\
&
{\textstyle(
(-1)^j k_1
+
(-1)^{h+j} k_2
-
(-1)^h
\theta_i^*
+
\frac{1}{2}
)}
{\left\lfloor \frac{n-i-j}{2}\right\rfloor+\left\lfloor\frac{i-h}{2}\right\rfloor \choose \left\lfloor\frac{i-h}{2}\right\rfloor}
c_{h,i,j}.
\end{align*}
It is straightforward to verify that the sum of the above three terms is equal to $(-1)^h$ times 
$$
\left\{
\begin{array}{ll}
\begin{pmatrix}
0 &0
\\
0 &0
\end{pmatrix}
\qquad &\hbox{if $h=i$},
\\
(i-n-2k_1-1)
\displaystyle{\left\lfloor \frac{n-i-j+1}{2}\right\rfloor+\left\lfloor\frac{i-h-1}{2}\right\rfloor \choose \left\lfloor\frac{i-h-1}{2}\right\rfloor}
c_{h,i-1,j}
\qquad &\hbox{if $h<i$ and $i$ is even},
\\
(i-n-1)
\displaystyle{\left\lfloor \frac{n-i-j+1}{2}\right\rfloor+\left\lfloor\frac{i-h-1}{2}\right\rfloor \choose \left\lfloor\frac{i-h-1}{2}\right\rfloor}
c_{h,i-1,j}
\qquad &\hbox{if $h<i$ and $i$ is odd}.
\end{array}
\right.
$$
By the above comments the equations given in (\ref{e:t1<=n<t1+t2(even)-2}) follow.

(iii): It is routine to verify that $\D(p_0)=0$. Combined with (ii) the statement (iii) follows.
\end{proof}

\begin{lem}\label{lem:basis:t1<=n<t1+t2(even)}
Suppose that $n$ is an even integer with $t_1+t_3\leq n<t_2$. 
Let $\{p_i\}_{i=0}^{n-t_1}$ be as in Proposition \ref{prop:t1<=n<t1+t2(even)}. Then the following hold:
\begin{enumerate}
\item $\M_n(x_1)\cap \M_n(x_3)$ has the basis 
\begin{align}
p_i\cdot 
\begin{pmatrix}
1
\\
0
\end{pmatrix}
\otimes
1 
\qquad 
(t_3\leq i\leq n-t_1), 
\label{M13basis:n<t1+t2(even)-1}
\\ 
p_i\cdot 
\begin{pmatrix}
0
\\
1
\end{pmatrix}
\otimes
1 
\qquad 
(t_3\leq i\leq n-t_1).
\label{M13basis:n<t1+t2(even)-2}
\end{align}

\item $\M_n(x_1)/\M_n(x_1)\cap \M_n(x_3)$ has the basis
\begin{align}
p_i\cdot 
\begin{pmatrix}
1
\\
0
\end{pmatrix}
\otimes
1 
+
\M_n(x_1)\cap \M_n(x_3)
\qquad 
(0\leq i\leq t_3-1), 
\label{M1/M13basis:n<t1+t2(even)-1}
\\ 
p_i\cdot 
\begin{pmatrix}
0
\\
1
\end{pmatrix}
\otimes
1 
+
\M_n(x_1)\cap \M_n(x_3)
\qquad 
(0\leq i\leq t_3-1).
\label{M1/M13basis:n<t1+t2(even)-2}
\end{align}

\item $\M_n/\M_n(x_3)$ has the basis
\begin{align}
p_i\cdot 
\begin{pmatrix}
1
\\
0
\end{pmatrix}
\otimes
1 
+
\M_n(x_3)
\qquad 
(0\leq i\leq t_3-1), 
\label{M/M3basis:n<t1+t2(even)-1}
\\ 
p_i\cdot 
\begin{pmatrix}
0
\\
1
\end{pmatrix}
\otimes
1 
+
\M_n(x_3)
\qquad 
(0\leq i\leq t_3-1).
\label{M/M3basis:n<t1+t2(even)-2}
\end{align}
\end{enumerate}
\end{lem}
\begin{proof}
(i): The linear independence of (\ref{M13basis:n<t1+t2(even)-1}) and (\ref{M13basis:n<t1+t2(even)-2}) follows from Proposition \ref{prop:t1<=n<t1+t2(even)}(i). Since (\ref{M13basis:n<t1+t2(even)-1}) and (\ref{M13basis:n<t1+t2(even)-2}) are in 
$
\C^2\otimes\left(\bigoplus\limits_{i=t_1}^n x_1^i\cdot \R[x_2,x_3]_{n-i}
\cap 
\bigoplus\limits_{i=t_3}^n x_3^i\cdot \R[x_1,x_2]_{n-i}
\right)$,  
it follows from Proposition \ref{prop:t1<=n<t1+t2(even)}(iii) that (\ref{M13basis:n<t1+t2(even)-1}) and (\ref{M13basis:n<t1+t2(even)-2}) are in $\M_n(x_1)\cap \M_n(x_3)$. Combined with Proposition \ref{prop:dimM(x12)}(iii) the statement (i) follows.

(ii): The linear independence of (\ref{M1/M13basis:n<t1+t2(even)-1}) and (\ref{M1/M13basis:n<t1+t2(even)-2}) follows from Proposition \ref{prop:t1<=n<t1+t2(even)}(i). 
By Proposition \ref{prop:t1<=n<t1+t2(even)}(iii) the cosets (\ref{M1/M13basis:n<t1+t2(even)-1}) and (\ref{M1/M13basis:n<t1+t2(even)-2}) are in 
$
\M_n(x_1)/\M_n(x_1)\cap \M_n(x_3).
$  
By Propositions \ref{prop:dimM(x1)}(i) and \ref{prop:dimM(x12)}(iii) the dimension of $\M_n(x_1)/\M_n(x_1)\cap \M_n(x_3)$ is $2t_3$. The statement (ii) follows.

(iii): The linear independence of (\ref{M/M3basis:n<t1+t2(even)-1}) and (\ref{M/M3basis:n<t1+t2(even)-2}) follows from Proposition \ref{prop:t1<=n<t1+t2(even)}(i). By Proposition \ref{prop:t1<=n<t1+t2(even)}(iii) the cosets (\ref{M/M3basis:n<t1+t2(even)-1}) and (\ref{M/M3basis:n<t1+t2(even)-2}) are in  $\M_n/\M_n(x_3)$. By Theorem \ref{thm:dimM=2(n+1)} and Proposition \ref{prop:dimM(x1)}(iii) 
the dimension of $\M_n/\M_n(x_3)$ is $2t_3$.
The statement (iii) follows.
\end{proof}

\begin{prop}\label{prop:t3<=n<t2+t3(even)}
Suppose that $n$ is an even integer with $t_3\leq n<t_2+t_3$. 
Let 
$$
p_i
=
\sum_{h=0}^i
(-1)^{-\frac{h}{2}}
\sum_{j=t_3}^{n-h}
(-1)^{\frac{j}{2}}
{\left\lfloor \frac{n-i-j}{2}\right\rfloor+\left\lfloor\frac{i-h}{2}\right\rfloor \choose \left\lfloor\frac{i-h}{2}\right\rfloor}
c_{h,i,j}
\otimes
[x_3]^n_j[x_1]^i_h[x_2]^{n-t_3}_{n-h-j}
$$
for all $i=0,1,\ldots,n-t_3$ where
\begin{align*}
c_{h,i,j}&=
\left\{
\begin{array}{ll}
\sigma_1^j \sigma_3
\qquad &\hbox{if $h$ is odd and $i$ is odd},
\\
\sigma_1^j
\qquad &\hbox{if $h$ is even and $i$ is even},
\\
\sigma_3
\qquad &\hbox{if $h$ is odd, $i$ is even and $j$ is even},
\\
-\sigma_1
\qquad &\hbox{if $h$ is even, $i$ is odd and $j$ is odd},
\\
\begin{pmatrix}
0 &0
\\
0 &0
\end{pmatrix}
\qquad &\hbox{else}
\end{array}
\right.
\end{align*}
for any integers $h,i,j$.
Then the following hold:
\begin{enumerate}
\item $\{p_i\}_{i=0}^{n-t_3}$ are linearly independent over ${\rm Mat}_2(\C)$.

\item The following equations hold:
\begin{align}
(Z-\theta_i)p_i &=p_{i+1} \qquad (0\leq i\leq n-t_3-1),
\qquad 
(Z-\theta_{n-t_3}) p_{n-t_3}=0,
\label{e:t3<=n<t2+t3(even)-1}
\\
(X-\theta_i^*) p_i &=\varphi_i p_{i-1} \qquad (1\leq i\leq n-t_3),
\qquad 
(X-\theta_0^*) p_0=0,
\label{e:t3<=n<t2+t3(even)-2}
\end{align}
where
\begin{align*}
\theta_i &=(-1)^i
\textstyle(
k_1+k_2+i+\frac{1}{2}
)
\qquad 
(0\leq i\leq n-t_3),
\\
\theta_i^* &=
(-1)^i
\textstyle(
n+k_2+k_3-i+\frac{1}{2}
)
\qquad 
(0\leq i\leq n-t_3),
\\
\varphi_i
&=
\left\{
\begin{array}{ll}
i
(i-n-2k_3-1)
\qquad
\hbox{if $i$ is even},
\\
(i+2 k_1)(i-n-1)
\qquad
\hbox{if $i$ is odd}
\end{array}
\right.
\qquad 
(1\leq i\leq n-t_3).
\end{align*}

\item $\D(p_i)=0$ for all $i=0,1,\ldots,n-t_3$. 
\end{enumerate}
\end{prop}
\begin{proof}
(i): Let $i$ be an integer with $0\leq i\leq n-t_3$.
By construction the coefficient of $x_1^h$ in $p_i$ is zero for all integers $h$ with $i<h\leq n-t_3$. Observe that the coefficient of $x_3^{n-i} x_1^i$ in $p_i$ is 
\begin{gather}\label{coeff:t3<=n<t2+t3(even)}
(-1)^{\frac{n}{2}-i}
\prod_{h=n-i+1}^n m_3^{(h)}
\prod_{h=1}^{n-t_3} m_2^{(h)}
\times 
\left\{
\begin{array}{ll}
(-1)^\frac{3}{2}\sigma_2
\qquad 
&\hbox{if $i$ is odd},
\\
\begin{pmatrix}
1 &0
\\
0 &1
\end{pmatrix} 
\qquad 
&\hbox{if $i$ is even}.
\end{array}
\right.
\end{gather}
Since $i\leq n-t_3$ the scalar $\prod\limits_{h=n-i+1}^n m_3^{(h)}$ is nonzero. 
Since $n<t_2+t_3$ the scalar $\prod\limits_{h=1}^{n-t_3} m_2^{(h)}$ is nonzero. 
Hence the matrix (\ref{coeff:t3<=n<t2+t3(even)}) is nonsingular. By the above comments the part (i) follows.

(ii): 
Let $i$ be a nonnegative integer and let $j$ be an integer with $t_3\leq j\leq n-i-1$. 
Using Theorem \ref{thm:BImodule_Mn}(i) yields that the coefficient of 
$[x_3]^n_j  [x_1]^{i+1}_{i+1} [x_2]^{n-t_3}_{n-i-j-1}$ in $(Z-\theta_i)p_i$ is equal to 
$$
(-1)^{-\frac{i+j+1}{2}}\sigma_3 c_{i,i,j}=
(-1)^{-\frac{i-j+1}{2}}c_{i+1,i+1,j}.
$$
Now let $h,i,j$ denote three integers with $0\leq h\leq i\leq n-t_3$ and $t_3\leq j\leq n-h$. Using Theorem \ref{thm:BImodule_Mn}(i) yields that the coefficient of 
$[x_3]^n_j [x_1]^i_h [x_2]^{n-t_3}_{n-h-j} $ in $(Z-\theta_i)p_i$ is equal to $(-1)^{-\frac{h+j}{2}}$ times the sum of 
\begin{align*}
&
m_1^{(h)}
{\left\lfloor \frac{n-i-j}{2}\right\rfloor+\left\lfloor\frac{i-h+1}{2}\right\rfloor \choose \left\lfloor\frac{i-h+1}{2}\right\rfloor}
 \sigma_3 c_{h-1,i,j},
\\
&
m_2^{(n-h-j)}
{\left\lfloor \frac{n-i-j}{2}\right\rfloor+\left\lfloor\frac{i-h-1}{2}\right\rfloor \choose \left\lfloor\frac{i-h-1}{2}\right\rfloor}
 \sigma_3 c_{h+1,i,j},
\\
&
{\textstyle(
(-1)^h k_1
+
(-1)^{h+j} k_2
-
(-1)^j \theta_i
+
\frac{1}{2}
)}
{\left\lfloor \frac{n-i-j}{2}\right\rfloor+\left\lfloor\frac{i-h}{2}\right\rfloor \choose \left\lfloor\frac{i-h}{2}\right\rfloor}
c_{h,i,j}.
\end{align*}
It is straightforward to verify that the sum of the above three terms is equal to 
$$
(-1)^j
m_1^{(i+1)}
{\left\lfloor \frac{n-i-j-1}{2}\right\rfloor+\left\lfloor\frac{i-h+1}{2}\right\rfloor \choose \left\lfloor\frac{i-h+1}{2}\right\rfloor}
c_{h,i+1,j}.
$$
By the above comments the equations given in (\ref{e:t3<=n<t2+t3(even)-1}) follow.

Using Theorem \ref{thm:BImodule_Mn}(i) yields that the coefficient of 
$[x_3]^n_j [x_1]^i_h [x_2]^{n-t_3}_{n-h-j}$ in $(X-\theta_i^*)p_i$ is equal to $(-1)^{\frac{h+j}{2}}$ times the sum of 
\begin{align*}
&
m_3^{(j)}
{\left\lfloor \frac{n-i-j+1}{2}\right\rfloor+\left\lfloor\frac{i-h}{2}\right\rfloor \choose \left\lfloor\frac{i-h}{2}\right\rfloor}
 \sigma_1 c_{h,i,j-1},
\\
&
m_2^{(n-h-j)}
{\left\lfloor \frac{n-i-j-1}{2}\right\rfloor+\left\lfloor\frac{i-h}{2}\right\rfloor \choose \left\lfloor\frac{i-h}{2}\right\rfloor}
 \sigma_1 c_{h,i,j+1},
\\
&
{\textstyle(
(-1)^j k_3
+
(-1)^{h+j} k_2
-
(-1)^h
\theta_i^*
+
\frac{1}{2}
)}
{\left\lfloor \frac{n-i-j}{2}\right\rfloor+\left\lfloor\frac{i-h}{2}\right\rfloor \choose \left\lfloor\frac{i-h}{2}\right\rfloor}
c_{h,i,j}.
\end{align*}
It is straightforward to verify that the sum of the above three terms is equal to $(-1)^h$ times 
$$
\left\{
\begin{array}{ll}
\begin{pmatrix}
0 &0
\\
0 &0
\end{pmatrix}
\qquad &\hbox{if $h=i$},
\\
(i-n-2k_3-1)
\displaystyle{\left\lfloor \frac{n-i-j+1}{2}\right\rfloor+\left\lfloor\frac{i-h-1}{2}\right\rfloor \choose \left\lfloor\frac{i-h-1}{2}\right\rfloor}
c_{h,i-1,j}
\qquad &\hbox{if $h<i$ and $i$ is even},
\\
(i-n-1)
\displaystyle{\left\lfloor \frac{n-i-j+1}{2}\right\rfloor+\left\lfloor\frac{i-h-1}{2}\right\rfloor \choose \left\lfloor\frac{i-h-1}{2}\right\rfloor}
c_{h,i-1,j}
\qquad &\hbox{if $h<i$ and $i$ is odd}.
\end{array}
\right.
$$
By the above comments the equations given in (\ref{e:t3<=n<t2+t3(even)-2}) follow.

(iii): It is routine to verify that $\D(p_0)=0$. Combined with (ii) the statement (iii) follows.
\end{proof}

\begin{lem}\label{lem:basis:n<t2+t3(even)}
Suppose that $n$ is an even integer with $t_1+t_3\leq n<t_2$. 
Let $\{p_i\}_{i=0}^{n-t_3}$ be as in Proposition \ref{prop:t3<=n<t2+t3(even)}. Then the following hold:
\begin{enumerate}
\item $\M_n(x_1)\cap \M_n(x_3)$ has the basis 
\begin{align}
p_i\cdot 
\begin{pmatrix}
1
\\
0
\end{pmatrix}
\otimes
1 
\qquad 
(t_1\leq i\leq n-t_3), 
\label{M13basis:n<t2+t3(even)-1}
\\ 
p_i\cdot 
\begin{pmatrix}
0
\\
1
\end{pmatrix}
\otimes
1 
\qquad 
(t_1\leq i\leq n-t_3).
\label{M13basis:n<t2+t3(even)-2}
\end{align}

\item $\M_n(x_3)/\M_n(x_1)\cap \M_n(x_3)$ has the basis
\begin{align}
p_i\cdot 
\begin{pmatrix}
1
\\
0
\end{pmatrix}
\otimes
1 
+
\M_n(x_1)\cap \M_n(x_3)
\qquad 
(0\leq i\leq t_1-1), 
\label{M3/M13basis:n<t2+t3(even)-1}
\\ 
p_i\cdot 
\begin{pmatrix}
0
\\
1
\end{pmatrix}
\otimes
1 
+
\M_n(x_1)\cap \M_n(x_3)
\qquad 
(0\leq i\leq t_1-1).
\label{M3/M13basis:n<t2+t3(even)-2}
\end{align}

\item $\M_n/\M_n(x_1)$ has the basis
\begin{align}
p_i\cdot 
\begin{pmatrix}
1
\\
0
\end{pmatrix}
\otimes
1 
+
\M_n(x_1)
\qquad 
(0\leq i\leq t_1-1), 
\label{M/M1basis:n<t2+t3(even)-1}
\\ 
p_i\cdot 
\begin{pmatrix}
0
\\
1
\end{pmatrix}
\otimes
1 
+
\M_n(x_1)
\qquad 
(0\leq i\leq t_1-1).
\label{M/M1basis:n<t2+t3(even)-2}
\end{align}
\end{enumerate}
\end{lem}
\begin{proof}
(i): The linear independence of (\ref{M13basis:n<t2+t3(even)-1}) and (\ref{M13basis:n<t2+t3(even)-2}) follows from Proposition \ref{prop:t3<=n<t2+t3(even)}(i). Since (\ref{M13basis:n<t2+t3(even)-1}) and (\ref{M13basis:n<t2+t3(even)-2}) are in 
$
\C^2\otimes\left(\bigoplus\limits_{i=t_1}^n x_1^i\cdot \R[x_2,x_3]_{n-i}
\cap 
\bigoplus\limits_{i=t_3}^n x_3^i\cdot \R[x_1,x_2]_{n-i}
\right)$,  
it follows from  Proposition \ref{prop:t3<=n<t2+t3(even)}(iii) that (\ref{M13basis:n<t2+t3(even)-1}) and (\ref{M13basis:n<t2+t3(even)-2}) are in $\M_n(x_1)\cap \M_n(x_3)$. Combined with Proposition \ref{prop:dimM(x12)}(iii) the statement (i) follows.

(ii): The linear independence of (\ref{M3/M13basis:n<t2+t3(even)-1}) and (\ref{M3/M13basis:n<t2+t3(even)-2}) follows from Proposition \ref{prop:t3<=n<t2+t3(even)}(i). 
By Proposition \ref{prop:t3<=n<t2+t3(even)}(iii) the cosets (\ref{M3/M13basis:n<t2+t3(even)-1}) and (\ref{M3/M13basis:n<t2+t3(even)-2}) are in 
$
\M_n(x_3)/\M_n(x_1)\cap \M_n(x_3).
$  
By Propositions \ref{prop:dimM(x1)}(iii) and \ref{prop:dimM(x12)}(iii) the dimension of $\M_n(x_3)/\M_n(x_1)\cap \M_n(x_3)$ is $2t_1$. The statement (ii) follows.

(iii): The linear independence of (\ref{M/M1basis:n<t2+t3(even)-1}) and (\ref{M/M1basis:n<t2+t3(even)-2}) follows from Proposition \ref{prop:t3<=n<t2+t3(even)}(i). By Proposition \ref{prop:t3<=n<t2+t3(even)}(iii) the cosets (\ref{M/M1basis:n<t2+t3(even)-1}) and (\ref{M/M1basis:n<t2+t3(even)-2}) are in  $\M_n/\M_n(x_1)$. By Theorem \ref{thm:dimM=2(n+1)} and Proposition \ref{prop:dimM(x1)}(i) 
the dimension of $\M_n/\M_n(x_1)$ is $2t_1$.
The statement (iii) follows.
\end{proof}

\begin{thm}\label{thm:t1+t3<=n<t2(even)}
Suppose that $n$ is an even integer with $t_1+t_3\leq n<t_2$. Then the following hold:
\begin{enumerate}
\item The $\BI$-module $\M_n(x_1)\cap \M_n(x_3)$ is isomorphic to a direct sum of two copies of 
\begin{gather}\label{BImodule:t1+t3<=n<t2(even)-1}
O_{n-t_1-t_3}
(
\textstyle 
k_1+k_2+\frac{n+1}{2},
-k_2-k_3-\frac{n+1}{2},
-\frac{n+1}{2}
)^{((-1,1),(2\,3))}.
\end{gather}

\item The $\BI$-modules $\M_n(x_1)/\M_n(x_1)\cap \M_n(x_3)$ and $\M_n/\M_n(x_3)$ are isomorphic to a direct sum of two copies of 
\begin{gather}\label{BImodule:t1+t3<=n<t2(even)-2}
O_{t_3-1}(k_2,-k_1-k_2-k_3-n-1,-k_1)^{((1,-1),(2\,3))}.
\end{gather}

\item The $\BI$-modules $\M_n(x_3)/\M_n(x_1)\cap \M_n(x_3)$ and $\M_n/\M_n(x_1)$ are  isomorphic to a direct sum of two copies of 
\begin{gather}\label{BImodule:t1+t3<=n<t2(even)-3}
O_{t_1-1}(k_2,-k_1-k_2-k_3-n-1,-k_3)^{((1,-1),(1\,2\,3))}.
\end{gather}
\end{enumerate}
Moreover, if $k_2$ is nonnegative then the $\BI$-modules (\ref{BImodule:t1+t3<=n<t2(even)-1})--(\ref{BImodule:t1+t3<=n<t2(even)-3}) are irreducible and they are isomorphic to 
\begin{align}
&O_{n-t_1-t_3}
(
\textstyle 
-k_1-k_2-\frac{n+1}{2},
\frac{n+1}{2},
-k_2-k_3-\frac{n+1}{2}
),
\label{BImodule:t1+t3<=n<t2(even)-1'}
\\
&O_{t_3-1}(k_2,k_1,k_1+k_2+k_3+n+1),
\label{BImodule:t1+t3<=n<t2(even)-2'}
\\
&O_{t_1-1}(k_1+k_2+k_3+n+1,k_3,k_2),
\label{BImodule:t1+t3<=n<t2(even)-3'}
\end{align}
respectively.
\end{thm}
\begin{proof}
(i): Let $V$ and $V'$ denote the subspaces of $\M_n$ spanned by (\ref{M13basis:n<t1+t2(even)-1}) and (\ref{M13basis:n<t1+t2(even)-2}), respectively. It follows from Lemma \ref{lem:basis:t1<=n<t1+t2(even)}(i) that 
$
\M_n(x_1)\cap \M_n(x_3)=V\oplus V'$. 
It follows from Theorem \ref{thm:BImodule_Mn}(ii) and Proposition \ref{prop:t1<=n<t1+t2(even)}(ii) that $V$ and $V'$ are two isomorphic $\BI$-submodules of $\M_n$. Using Proposition \ref{prop:Od} yields that both are isomorphic to (\ref{BImodule:t1+t3<=n<t2(even)-1}). 

(ii): 
Let $V$ and $V'$ denote the subspaces of $\M_n(x_1)/\M_n(x_1)\cap \M_n(x_3)$ spanned by (\ref{M1/M13basis:n<t1+t2(even)-1}) and (\ref{M1/M13basis:n<t1+t2(even)-2}), respectively. It follows from Lemma \ref{lem:basis:t1<=n<t1+t2(even)}(ii) that 
$
\M_n(x_1)/\M_n(x_1)\cap \M_n(x_3)=V\oplus V'$. 
It follows from Theorem \ref{thm:BImodule_Mn}(ii) and Proposition \ref{prop:t1<=n<t1+t2(even)}(ii) that $V$ and $V'$ are two isomorphic $\BI$-submodules of $\M_n(x_1)/\M_n(x_1)\cap \M_n(x_3)$. Using Proposition \ref{prop:Od} yields that both are isomorphic to (\ref{BImodule:t1+t3<=n<t2(even)-2}). 

Let $V$ and $V'$ denote the subspaces of $\M_n/\M_n(x_3)$ spanned by (\ref{M/M3basis:n<t1+t2(even)-1}) and (\ref{M/M3basis:n<t1+t2(even)-2}), respectively. It follows from Lemma \ref{lem:basis:t1<=n<t1+t2(even)}(iii) that 
$
\M_n/\M_n(x_3)=V\oplus V'$. 
It follows from Theorem \ref{thm:BImodule_Mn}(ii) and Proposition \ref{prop:t1<=n<t1+t2(even)}(ii) that $V$ and $V'$ are two isomorphic $\BI$-submodules of $\M_n/\M_n(x_3)$. Using Proposition \ref{prop:Od} yields that both are isomorphic to (\ref{BImodule:t1+t3<=n<t2(even)-2}). 

(iii): 
Let $V$ and $V'$ denote the subspaces of $\M_n(x_3)/\M_n(x_1)\cap \M_n(x_3)$ spanned by (\ref{M3/M13basis:n<t2+t3(even)-1}) and (\ref{M3/M13basis:n<t2+t3(even)-2}), respectively. It follows from Lemma \ref{lem:basis:n<t2+t3(even)}(ii) that 
$
\M_n(x_3)/\M_n(x_1)\cap \M_n(x_3)=V\oplus V'$. 
It follows from Theorem \ref{thm:BImodule_Mn}(ii) and Proposition \ref{prop:t3<=n<t2+t3(even)}(ii) that $V$ and $V'$ are two isomorphic $\BI$-submodules of $\M_n(x_3)/\M_n(x_1)\cap \M_n(x_3)$. Using Proposition \ref{prop:Od} yields that both are isomorphic to (\ref{BImodule:t1+t3<=n<t2(even)-3}). 

Let $V$ and $V'$ denote the subspaces of $\M_n/\M_n(x_1)$ spanned by (\ref{M/M1basis:n<t2+t3(even)-1}) and (\ref{M/M1basis:n<t2+t3(even)-2}), respectively. It follows from Lemma \ref{lem:basis:n<t2+t3(even)}(iii) that 
$
\M_n/\M_n(x_1)=V\oplus V'$. 
It follows from Theorem \ref{thm:BImodule_Mn}(ii) and Proposition \ref{prop:t3<=n<t2+t3(even)}(ii) that $V$ and $V'$ are two isomorphic $\BI$-submodules of $\M_n/\M_n(x_1)$. Using Proposition \ref{prop:Od} yields that both are isomorphic to (\ref{BImodule:t1+t3<=n<t2(even)-3}).

Suppose that $k_2$ is nonnegative. Using Theorem \ref{thm:irr_O} yields that the $\BI$-modules (\ref{BImodule:t1+t3<=n<t2(even)-1})--(\ref{BImodule:t1+t3<=n<t2(even)-3}) are irreducible. 
Using Theorem \ref{thm:onto2_O} yields that the $\BI$-modules (\ref{BImodule:t1+t3<=n<t2(even)-1})--(\ref{BImodule:t1+t3<=n<t2(even)-3}) are isomorphic to  (\ref{BImodule:t1+t3<=n<t2(even)-1'})--(\ref{BImodule:t1+t3<=n<t2(even)-3'}), respectively. 
\end{proof}

By similar arguments as in the proof of Theorem \ref{thm:t1+t3<=n<t2(even)} we have the following results:

\begin{thm}\label{thm:t1+t2<=n<t3(even)}
Suppose that $n$ is an even integer with $t_1+t_2\leq n<t_3$. Then the following hold:
\begin{enumerate}
\item The $\BI$-module $\M_n(x_1)\cap \M_n(x_2)$ is isomorphic to a direct sum of two copies of 
\begin{gather}\label{t1+t2<=n<t3(even)-1}
O_{n-t_1-t_2}
(
\textstyle
k_2+k_3+\frac{n+1}{2},
-k_1-k_3-\frac{n+1}{2},
-\frac{n+1}{2}
)^{((-1,1),(1\,2))}.
\end{gather}

\item The $\BI$-modules $\M_n(x_2)/\M_n(x_1)\cap \M_n(x_2)$ and $\M_n/\M_n(x_1)$ are isomorphic to a direct sum of two copies of 
\begin{gather}\label{t1+t2<=n<t3(even)-2}
O_{t_1-1}(k_3,-k_1-k_2-k_3-n-1,-k_2)^{((1,-1),(1\,2))}.
\end{gather}

\item The $\BI$-modules $\M_n(x_1)/\M_n(x_1)\cap \M_n(x_2)$ and $\M_n/\M_n(x_2)$ are isomorphic to a direct sum of two copies of 
\begin{gather}\label{t1+t2<=n<t3(even)-3}
O_{t_2-1}(k_3,-k_1-k_2-k_3-n-1,-k_1)^{(1,-1)}.
\end{gather}
\end{enumerate}
Moreover, if $k_3$ is nonnegative then the $\BI$-modules (\ref{t1+t2<=n<t3(even)-1})--(\ref{t1+t2<=n<t3(even)-3}) are irreducible and they are isomorphic to 
\begin{align*}
&O_{n-t_1-t_2}
\textstyle
(-k_1-k_3-\frac{n+1}{2},
-k_2-k_3-\frac{n+1}{2},
\frac{n+1}{2}),
\\
&O_{t_1-1}(k_1+k_2+k_3+n+1,k_3,k_2),
\\
&O_{t_2-1}(k_3,k_1+k_2+k_3+n+1,k_1),
\end{align*}
respectively.
\end{thm}

\begin{thm}\label{thm:t2+t3<=n<t1(even)}
Suppose that $n$ is an even integer with $t_2+t_3\leq n<t_1$. Then the following hold:
\begin{enumerate}
\item The $\BI$-module $\M_n(x_2)\cap \M_n(x_3)$ is isomorphic to a direct sum of two copies of 
\begin{gather}\label{t2+t3<=n<t1(even)-1}
O_{n-t_2-t_3}
\textstyle
(
k_1+k_3+\frac{n+1}{2},
-k_1-k_2-\frac{n+1}{2},
-\frac{n+1}{2}
)^{((-1,1),(1\,3))}.
\end{gather}

\item The $\BI$-modules $\M_n(x_3)/\M_n(x_2)\cap \M_n(x_3)$ and $\M_n/\M_n(x_2)$ are isomorphic to a direct sum of two copies of 
\begin{gather}\label{t2+t3<=n<t1(even)-2}
O_{t_2-1}(k_1,-k_1-k_2-k_3-n-1,-k_3)^{((1,-1),(1\,3))}.
\end{gather}

\item The $\BI$-modules $\M_n(x_2)/\M_n(x_2)\cap \M_n(x_3)$ and $\M_n/\M_n(x_3)$ are isomorphic to a direct sum of two copies of 
\begin{gather}\label{t2+t3<=n<t1(even)-3}
O_{t_3-1}(k_1,-k_1-k_2-k_3-n-1,-k_2)^{((1,-1),(1\,3\,2))}.
\end{gather}
\end{enumerate}
Moreover, if $k_1$ is nonnegative then the $\BI$-modules (\ref{t2+t3<=n<t1(even)-1})--(\ref{t2+t3<=n<t1(even)-3}) are irreducible and they are isomorphic to 
\begin{align*}
&O_{n-t_2-t_3}
\textstyle
(
\frac{n+1}{2},
-k_1-k_2-\frac{n+1}{2},
-k_1-k_3-\frac{n+1}{2}
),
\\
&O_{t_2-1}(k_3,k_1+k_2+k_3+n+1,k_1),
\\
&O_{t_3-1}(k_2,k_1,k_1+k_2+k_3+n+1),
\end{align*}
respectively.
\end{thm}

\section{The $\BI$-modules $\M_n$ of type (IV)}\label{s:case4}

\begin{prop}\label{prop:n>=t1+t3(odd)}
Suppose that $n$ is an odd integer with $n\geq t_1+t_3$. 
Let 
$$
p_i
=
\sum_{h=0}^i
(-1)^{\frac{h}{2}}
\sum_{j=t_1}^{n-t_3-h}
(-1)^{\frac{j}{2}}
{\left\lfloor \frac{n-t_3-i-j}{2}\right\rfloor+\left\lfloor\frac{i-h}{2}\right\rfloor \choose \left\lfloor\frac{i-h}{2}\right\rfloor}
c_{h,i,j}
\otimes
[x_1]^{n-t_3}_j[x_2]^i_h[x_3]^{n-t_1}_{n-h-j}
$$
for all $i=0,1,\ldots,n-t_1-t_3$ where
\begin{align*}
c_{h,i,j}&=
\left\{
\begin{array}{ll}
\sigma_2^j \sigma_1 
\qquad &\hbox{if $h$ is odd and $i$ is odd},
\\
\sigma_2^j
\qquad &\hbox{if $h$ is even and $i$ is even},
\\
\sigma_2
\qquad &\hbox{if $h$ is even, $i$ is odd and $j$ is odd},
\\
-\sigma_1
\qquad &\hbox{if $h$ is odd, $i$ is even and $j$ is even},
\\
0
\qquad &\hbox{else}
\end{array}
\right.
\end{align*}
for any integers $h,i,j$. 
Then the following hold:
\begin{enumerate}
\item $\{p_i\}_{i=0}^{n-t_1-t_3}$ are linearly independent over ${\rm Mat}_2(\C)$.

\item The following equations hold:
\begin{align}
(X-\theta_i)p_i &=p_{i+1} \qquad (0\leq i\leq n-t_1-t_3-1),
\qquad 
(X-\theta_{n-t_1-t_3}) p_{n-t_1-t_3}=0,
\label{e:n>=t1+t3(odd)-1}
\\
(Y-\theta_i^*) p_i &=\varphi_i p_{i-1} \qquad (1\leq i\leq n-t_1-t_3),
\qquad 
(Y-\theta_0^*) p_0=0,
\label{e:n>=t1+t3(odd)-2}
\end{align}
where
\begin{align*}
\theta_i &=(-1)^i
\textstyle(
k_3-k_2-i-\frac{1}{2}
)
\qquad 
(0\leq i\leq n-t_1-t_3),
\\
\theta_i^* &=(-1)^{i+1}
\textstyle(
k_1+k_3+n-i+\frac{1}{2}
)
\qquad 
(0\leq i\leq n-t_1-t_3),
\\
\varphi_i
&=
\left\{
\begin{array}{ll}
i(i-2k_1-2k_3-n-1)
\qquad
\hbox{if $i$ is even},
\\
(i+2k_2)(i-2k_3-n-1)
\qquad
\hbox{if $i$ is odd}
\end{array}
\right.
\qquad 
(1\leq i\leq n-t_1-t_3).
\end{align*}

\item $\D(p_i)=0$ for all $i=0,1,\ldots,n-t_1-t_3$. 

\end{enumerate}
\end{prop}
\begin{proof}
(i): Let $i$ be an integer with $0\leq i\leq n-t_1-t_3$.
By construction the coefficient of $x_2^h$ in $p_i$ is zero for all integers $h$ with $i<h\leq n-t_1-t_3$. Observe that the coefficient of $x_1^{n-t_3-i}x_2^i x_3^{t_3}$ in $p_i$ is 
\begin{gather}\label{coeff:n>=t1+t3(odd)}
(-1)^{\frac{n-t_3}{2}}
\prod_{h=n-t_3-i+1}^{n-t_3} m_1^{(h)}
\prod_{h=t_3+1}^{n-t_1} m_3^{(h)}
\times 
\left\{
\begin{array}{ll}
(-1)^\frac{3}{2}\sigma_3
\qquad 
&\hbox{if $i$ is odd},
\\
\begin{pmatrix}
1 &0
\\
0 &1
\end{pmatrix} \qquad 
&\hbox{if $i$ is even}.
\end{array}
\right.
\end{gather}
Clearly $\prod\limits_{h=t_3+1}^{n-t_1} m_3^{(h)}$ is nonzero.
Since $i\leq n-t_1-t_3$ the scalar $\prod\limits_{h=n-t_3-i+1}^{n-t_3} m_1^{(h)}$ is nonzero. 
Hence the matrix (\ref{coeff:n>=t1+t3(odd)}) is nonsingular. By the above comments the part (i) follows.

(ii): 
Let $i$ be a nonnegative integer and let $j$ be an integer with $t_1\leq j\leq n-t_3-i-1$. 
Using Theorem \ref{thm:BImodule_Mn}(i) yields that the coefficient of 
$[x_1]^{n-t_3}_j [x_2]^{i+1}_{i+1} [x_3]^{n-t_1}_{n-i-j-1} $ in $(X-\theta_i)p_i$ is equal to 
$$
(-1)^{\frac{i-j+1}{2}}\sigma_1 c_{i,i,j}=(-1)^{\frac{i+j+1}{2}}c_{i+1,i+1,j}.
$$
Now let $h,i,j$ denote three integers with $0\leq h\leq i\leq n-t_1-t_3$ and $t_1\leq j\leq n-t_3-h$. Using Theorem \ref{thm:BImodule_Mn}(i) yields that the coefficient of 
$[x_1]^{n-t_3}_j [x_2]^i_h [x_3]^{n-t_1}_{n-h-j}$ in $(X-\theta_i)p_i$ is equal to $(-1)^{\frac{h-j}{2}}$ times the sum of 
\begin{align*}
&
m_2^{(h)}
{\left\lfloor \frac{n-t_3-i-j}{2}\right\rfloor+\left\lfloor\frac{i-h+1}{2}\right\rfloor \choose \left\lfloor\frac{i-h+1}{2}\right\rfloor}
 \sigma_1 c_{h-1,i,j},
\\
&
m_3^{(n-h-j)}
{\left\lfloor \frac{n-t_3-i-j}{2}\right\rfloor+\left\lfloor\frac{i-h-1}{2}\right\rfloor \choose \left\lfloor\frac{i-h-1}{2}\right\rfloor}
 \sigma_1 c_{h+1,i,j},
\\
&
{\textstyle(
(-1)^{h+1} k_2
+
(-1)^{h+j} k_3
-
(-1)^j\theta_i
-
\frac{1}{2}
)}
{\left\lfloor \frac{n-t_3-i-j}{2}\right\rfloor+\left\lfloor\frac{i-h}{2}\right\rfloor \choose \left\lfloor\frac{i-h}{2}\right\rfloor}
c_{h,i,j}.
\end{align*}
It is straightforward to verify that the sum of the above three terms is equal to 
$$
(-1)^j
m_2^{(i+1)}
{\left\lfloor \frac{n-t_3-i-j-1}{2}\right\rfloor+\left\lfloor\frac{i-h+1}{2}\right\rfloor \choose \left\lfloor\frac{i-h+1}{2}\right\rfloor}
c_{h,i+1,j}. 
$$
By the above comments the equations given in (\ref{e:n>=t1+t3(odd)-1}) follow.

Using Theorem \ref{thm:BImodule_Mn}(i) yields that the coefficient of 
$[x_1]^{n-t_3}_j [x_2]^i_h [x_3]^{n-t_1}_{n-h-j}$ in $(Y-\theta_i^*)p_i$ is equal to $(-1)^{\frac{j-h}{2}+1}$ times the sum of 
\begin{align*}
&
m_1^{(j)}
{\left\lfloor \frac{n-t_3-i-j+1}{2}\right\rfloor+\left\lfloor\frac{i-h}{2}\right\rfloor \choose \left\lfloor\frac{i-h}{2}\right\rfloor}
 \sigma_2 c_{h,i,j-1},
\\
&
m_3^{(n-h-j)}
{\left\lfloor \frac{n-t_3-i-j-1}{2}\right\rfloor+\left\lfloor\frac{i-h}{2}\right\rfloor \choose \left\lfloor\frac{i-h}{2}\right\rfloor}
 \sigma_2 c_{h,i,j+1},
\\
&
{\textstyle(
(-1)^{j} k_1
-
(-1)^{h+j} k_3
+
(-1)^h
\theta_i^*
+
\frac{1}{2}
)}
{\left\lfloor \frac{n-t_3-i-j}{2}\right\rfloor+\left\lfloor\frac{i-h}{2}\right\rfloor \choose \left\lfloor\frac{i-h}{2}\right\rfloor}
c_{h,i,j}.
\end{align*}
It is straightforward to verify that the sum of the above three terms is equal to $(-1)^{h+1}$ times 
$$
\left\{
\begin{array}{ll}
\begin{pmatrix}
0 &0
\\
0 &0
\end{pmatrix}
\qquad &\hbox{if $h=i$},
\\
(i-2k_1-2k_3-n-1)
\displaystyle{\left\lfloor \frac{n-t_3-i-j+1}{2}\right\rfloor+\left\lfloor\frac{i-h-1}{2}\right\rfloor \choose \left\lfloor\frac{i-h-1}{2}\right\rfloor}
c_{h,i-1,j}
\qquad &\hbox{if $h<i$ and $i$ is even},
\\
(i-2k_3-n-1)
\displaystyle{\left\lfloor \frac{n-t_3-i-j+1}{2}\right\rfloor+\left\lfloor\frac{i-h-1}{2}\right\rfloor \choose \left\lfloor\frac{i-h-1}{2}\right\rfloor}
c_{h,i-1,j}
\qquad &\hbox{if $h<i$ and $i$ is odd}.
\end{array}
\right.
$$
By the above comments the equations given in (\ref{e:n>=t1+t3(odd)-2}) follow.

(iii): It is routine to verify that $\D(p_0)=0$. Combined with (ii) the statement (iii) follows.
\end{proof}

\begin{lem}\label{lem:basis:n>=t1+t3(odd)}
Suppose that $n$ is an odd integer with $n\geq t_1+t_2+t_3$. 
Let $\{p_i\}_{i=0}^{n-t_1-t_3}$ be as in Proposition \ref{prop:n>=t1+t3(odd)}. Then the following hold:
\begin{enumerate}
\item $\M_n(x_1)\cap \M_n(x_2)\cap \M_n(x_3)$ has the basis 
\begin{align}
p_i\cdot 
\begin{pmatrix}
1
\\
0
\end{pmatrix}
\otimes
1 
\qquad 
(t_2\leq i\leq n-t_1-t_3), 
\label{M123basis:n>=t1+t3(odd)-1}
\\ 
p_i\cdot 
\begin{pmatrix}
0
\\
1
\end{pmatrix}
\otimes
1 
\qquad 
(t_2\leq i\leq n-t_1-t_3).
\label{M123basis:n>=t1+t3(odd)-2}
\end{align}

\item $\M_n/\M_n(x_2)$ has the basis
\begin{align}
p_i\cdot 
\begin{pmatrix}
1
\\
0
\end{pmatrix}
\otimes
1 
+
\M_n(x_2)
\qquad 
(0\leq i\leq t_2-1), 
\label{M/M2basis:n>=t1+t3(odd)-1}
\\ 
p_i\cdot 
\begin{pmatrix}
0
\\
1
\end{pmatrix}
\otimes
1 
+
\M_n(x_2)
\qquad 
(0\leq i\leq t_2-1).
\label{M/M2basis:n>=t1+t3(odd)-2}
\end{align}
\end{enumerate}
\end{lem}
\begin{proof}
(i): The linear independence of (\ref{M123basis:n>=t1+t3(odd)-1}) and (\ref{M123basis:n>=t1+t3(odd)-2}) follows from Proposition \ref{prop:n>=t1+t3(odd)}(i). Since (\ref{M123basis:n>=t1+t3(odd)-1}) and (\ref{M123basis:n>=t1+t3(odd)-2}) are in 
$
\C^2\otimes\left(\bigoplus\limits_{i=t_1}^n x_1^i\cdot \R[x_2,x_3]_{n-i}
\cap
\bigoplus\limits_{i=t_2}^n x_2^i\cdot \R[x_1,x_3]_{n-i}
\cap 
\bigoplus\limits_{i=t_3}^n x_3^i\cdot \R[x_1,x_2]_{n-i}
\right)$,  
it follows from Proposition \ref{prop:n>=t1+t3(odd)}(iii) that (\ref{M123basis:n>=t1+t3(odd)-1}) and (\ref{M123basis:n>=t1+t3(odd)-2}) are in $\M_n(x_1)\cap \M_n(x_2)\cap \M_n(x_3)$. Combined with Proposition \ref{prop:dimM(x123)} the statement (i) follows.

(ii): The linear independence of (\ref{M/M2basis:n>=t1+t3(odd)-1}) and (\ref{M/M2basis:n>=t1+t3(odd)-2}) follows from Proposition \ref{prop:n>=t1+t3(odd)}(i). 
By Proposition \ref{prop:n>=t1+t3(odd)}(iii) the cosets (\ref{M/M2basis:n>=t1+t3(odd)-1}) and (\ref{M/M2basis:n>=t1+t3(odd)-2}) are in 
$
\M_n/\M_n(x_2).
$  
By  Theorem \ref{thm:dimM=2(n+1)} and Proposition \ref{prop:dimM(x1)}(ii) the dimension of $\M_n/\M_n(x_2)$ is $2t_2$. The statement (ii) follows.
\end{proof}

\begin{prop}\label{prop:n>=t1+t2(odd)}
Suppose that $n$ is an odd integer with $n\geq t_1+t_2$. 
Let 
$$
q_i
=
\sum_{h=0}^i
(-1)^{\frac{h}{2}}
\sum_{j=t_2}^{n-t_1-h}
(-1)^{\frac{j}{2}}
{\left\lfloor \frac{n-t_1-i-j}{2}\right\rfloor+\left\lfloor\frac{i-h}{2}\right\rfloor \choose \left\lfloor\frac{i-h}{2}\right\rfloor}
c_{h,i,j}
[x_2]^{n-t_1}_j[x_3]^i_h[x_1]^{n-t_2}_{n-h-j}
$$
for all $i=0,1,\ldots,n-t_1-t_2$ where
\begin{align*}
c_{h,i,j}&=
\left\{
\begin{array}{ll}
\sigma_3^j \sigma_2
\qquad &\hbox{if $h$ is odd and $i$ is odd},
\\
\sigma_3^j
\qquad &\hbox{if $h$ is even and $i$ is even},
\\
\sigma_3
\qquad &\hbox{if $h$ is even, $i$ is odd and $j$ is odd},
\\
-\sigma_2
\qquad &\hbox{if $h$ is odd, $i$ is even and $j$ is even},
\\
\begin{pmatrix}
0 &0 \\
0 &0
\end{pmatrix}
\qquad &\hbox{else}
\end{array}
\right.
\end{align*}
for any integers $h,i,j$. 
Then the following hold:
\begin{enumerate}
\item $\{q_i\}_{i=0}^{n-t_1-t_2}$ are linearly independent over ${\rm Mat}_2(\C)$.

\item The following equations hold:
\begin{align}
(Y-\theta_i)q_i &=q_{i+1} \qquad (0\leq i\leq n-t_1-t_2-1),
\qquad 
(Y-\theta_{n-t_1-t_2}) q_{n-t_1-t_2}=0,
\label{e:n>=t1+t2(odd)-1}
\\
(Z-\theta_i^*) q_i &=\varphi_i q_{i-1} \qquad (1\leq i\leq n-t_1-t_2),
\qquad 
(Z-\theta_0^*) q_0=0,
\label{e:n>=t1+t2(odd)-2}
\end{align}
where
\begin{align*}
\theta_i &=(-1)^i
\textstyle(
k_1-k_3-i-\frac{1}{2}
)
\qquad 
(0\leq i\leq n-t_1-t_2),
\\
\theta_i^* &=(-1)^{i+1}
\textstyle(
k_2+k_1+n-i+\frac{1}{2}
)
\qquad 
(0\leq i\leq n-t_1-t_2),
\\
\varphi_i
&=
\left\{
\begin{array}{ll}
i(i-2k_2-2k_1-n-1)
\qquad
\hbox{if $i$ is even},
\\
(i+2k_3)(i-2k_1-n-1)
\qquad
\hbox{if $i$ is odd}
\end{array}
\right.
\qquad 
(1\leq i\leq n-t_1-t_2).
\end{align*}

\item $\D(q_i)=0$ for all $i=0,1,\ldots,n-t_1-t_2$.
\end{enumerate}
\end{prop}
\begin{proof}
(i): Let $i$ be an integer with $0\leq i\leq n-t_1-t_2$.
By construction the coefficient of $x_3^h$ in $q_i$ is zero for all integers $h$ with $i<h\leq n-t_1-t_2$. 
Observe that the coefficient of $x_2^{n-t_1-i}x_3^i x_1^{t_1}$ in $q_i$ is 
\begin{gather}\label{coeff:n>=t1+t2(odd)}
(-1)^{\frac{n-t_1}{2}}
\prod_{h=n-t_1-i+1}^{n-t_1} m_2^{(h)}
\prod_{h=t_1+1}^{n-t_2} m_1^{(h)}
\times 
\left\{
\begin{array}{ll}
(-1)^\frac{3}{2}\sigma_1
\qquad 
&\hbox{if $i$ is odd},
\\
\begin{pmatrix}
1 &0
\\
0 &1
\end{pmatrix} \qquad 
&\hbox{if $i$ is even}.
\end{array}
\right.
\end{gather}
Clearly $\prod\limits_{h=t_1+1}^{n-t_2} m_1^{(h)}$ is nonzero.
Since $i\leq n-t_1-t_2$ the scalar $\prod\limits_{h=n-t_1-i+1}^{n-t_1} m_2^{(h)}$ is nonzero. 
Hence the matrix (\ref{coeff:n>=t1+t2(odd)}) is nonsingular. By the above comments the part (i) follows.

(ii): 
Let $i$ be a nonnegative integer and let $j$ be an integer with $t_2\leq j\leq n-t_1-i-1$. 
Using Theorem \ref{thm:BImodule_Mn}(i) yields that the coefficient of 
$[x_2]^{n-t_1}_j [x_3]^{i+1}_{i+1} [x_1]^{n-t_2}_{n-i-j-1} $ in $(Y-\theta_i)q_i$ is equal to 
$$
(-1)^{\frac{i-j+1}{2}}\sigma_2 c_{i,i,j}=(-1)^{\frac{i+j+1}{2}}c_{i+1,i+1,j}.
$$
Now let $h,i,j$ denote three integers with $0\leq h\leq i\leq n-t_1-t_2$ and $t_2\leq j\leq n-t_1-h$. Using Theorem \ref{thm:BImodule_Mn}(i) yields that the coefficient of 
$[x_2]^{n-t_1}_j [x_3]^i_h [x_1]^{n-t_2}_{n-h-j}$ in $(Y-\theta_i)q_i$ is equal to $(-1)^{\frac{h-j}{2}}$ times the sum of 
\begin{align*}
&
m_3^{(h)}
{\left\lfloor \frac{n-t_1-i-j}{2}\right\rfloor+\left\lfloor\frac{i-h+1}{2}\right\rfloor \choose \left\lfloor\frac{i-h+1}{2}\right\rfloor}
 \sigma_2 c_{h-1,i,j},
\\
&
m_1^{(n-h-j)}
{\left\lfloor \frac{n-t_1-i-j}{2}\right\rfloor+\left\lfloor\frac{i-h-1}{2}\right\rfloor \choose \left\lfloor\frac{i-h-1}{2}\right\rfloor}
 \sigma_2 c_{h+1,i,j},
\\
&
{\textstyle(
(-1)^{h+1} k_3
+
(-1)^{h+j} k_1
-
(-1)^j\theta_i
-
\frac{1}{2}
)}
{\left\lfloor \frac{n-t_1-i-j}{2}\right\rfloor+\left\lfloor\frac{i-h}{2}\right\rfloor \choose \left\lfloor\frac{i-h}{2}\right\rfloor}
c_{h,i,j}.
\end{align*}
It is straightforward to verify that the sum of the above three terms is equal to 
$$
(-1)^j
m_3^{(i+1)}
{\left\lfloor \frac{n-t_1-i-j-1}{2}\right\rfloor+\left\lfloor\frac{i-h+1}{2}\right\rfloor \choose \left\lfloor\frac{i-h+1}{2}\right\rfloor}
c_{h,i+1,j}. 
$$
By the above comments the equations given in (\ref{e:n>=t1+t2(odd)-1}) follow.

Using Theorem \ref{thm:BImodule_Mn}(i) yields that the coefficient of 
$[x_2]^{n-t_1}_j [x_3]^i_h [x_1]^{n-t_2}_{n-h-j}$ in $(Z-\theta_i^*)q_i$ is equal to $(-1)^{\frac{j-h}{2}+1}$ times the sum of 
\begin{align*}
&
m_2^{(j)}
{\left\lfloor \frac{n-t_1-i-j+1}{2}\right\rfloor+\left\lfloor\frac{i-h}{2}\right\rfloor \choose \left\lfloor\frac{i-h}{2}\right\rfloor}
 \sigma_3 c_{h,i,j-1},
\\
&
m_1^{(n-h-j)}
{\left\lfloor \frac{n-t_1-i-j-1}{2}\right\rfloor+\left\lfloor\frac{i-h}{2}\right\rfloor \choose \left\lfloor\frac{i-h}{2}\right\rfloor}
 \sigma_3 c_{h,i,j+1},
\\
&
{\textstyle(
(-1)^{j} k_2
-
(-1)^{h+j} k_1
+
(-1)^h
\theta_i^*
+
\frac{1}{2}
)}
{\left\lfloor \frac{n-t_1-i-j}{2}\right\rfloor+\left\lfloor\frac{i-h}{2}\right\rfloor \choose \left\lfloor\frac{i-h}{2}\right\rfloor}
c_{h,i,j}.
\end{align*}
It is straightforward to verify that the sum of the above three terms is equal to $(-1)^{h+1}$ times 
$$
\left\{
\begin{array}{ll}
\begin{pmatrix}
0 &0
\\
0 &0
\end{pmatrix}
\qquad &\hbox{if $h=i$},
\\
(i-2k_2-2k_1-n-1)
\displaystyle{\left\lfloor \frac{n-t_1-i-j+1}{2}\right\rfloor+\left\lfloor\frac{i-h-1}{2}\right\rfloor \choose \left\lfloor\frac{i-h-1}{2}\right\rfloor}
c_{h,i-1,j}
\qquad &\hbox{if $h<i$ and $i$ is even},
\\
(i-2k_1-n-1)
\displaystyle{\left\lfloor \frac{n-t_1-i-j+1}{2}\right\rfloor+\left\lfloor\frac{i-h-1}{2}\right\rfloor \choose \left\lfloor\frac{i-h-1}{2}\right\rfloor}
c_{h,i-1,j}
\qquad &\hbox{if $h<i$ and $i$ is odd}.
\end{array}
\right.
$$
By the above comments the equations given in (\ref{e:n>=t1+t2(odd)-2}) follow.

(iii): It is routine to verify that $\D(q_0)=0$. Combined with (ii) the statement (iii) follows.
\end{proof}

\begin{lem}\label{lem:basis:n>=t1+t2(odd)}
Suppose that $n$ is an odd integer with $n\geq t_1+t_2+t_3$. 
Let $\{q_i\}_{i=0}^{n-t_1-t_2}$ be as in Proposition \ref{prop:n>=t1+t2(odd)}. Then the following hold:
\begin{enumerate}
\item $\M_n(x_1)\cap \M_n(x_2)\cap \M_n(x_3)$ has the basis 
\begin{align}
q_i\cdot 
\begin{pmatrix}
1
\\
0
\end{pmatrix}
\otimes
1 
\qquad 
(t_3\leq i\leq n-t_1-t_2), 
\label{M123basis:n>=t1+t2(odd)-1}
\\ 
q_i\cdot 
\begin{pmatrix}
0
\\
1
\end{pmatrix}
\otimes
1 
\qquad 
(t_3\leq i\leq n-t_1-t_2).
\label{M123basis:n>=t1+t2(odd)-2}
\end{align}

\item $\M_n/\M_n(x_3)$ has the basis
\begin{align}
q_i\cdot 
\begin{pmatrix}
1
\\
0
\end{pmatrix}
\otimes
1 
+
\M_n(x_3)
\qquad 
(0\leq i\leq t_3-1), 
\label{M/M3basis:n>=t1+t2(odd)-1}
\\ 
q_i\cdot 
\begin{pmatrix}
0
\\
1
\end{pmatrix}
\otimes
1 
+
\M_n(x_3)
\qquad 
(0\leq i\leq t_3-1).
\label{M/M3basis:n>=t1+t2(odd)-2}
\end{align}
\end{enumerate}
\end{lem}
\begin{proof}
(i): The linear independence of  (\ref{M123basis:n>=t1+t2(odd)-1}) and (\ref{M123basis:n>=t1+t2(odd)-2}) follows from Proposition \ref{prop:n>=t1+t2(odd)}(i). Since (\ref{M123basis:n>=t1+t2(odd)-1}) and (\ref{M123basis:n>=t1+t2(odd)-2}) are in 
$
\C^2\otimes\left(\bigoplus\limits_{i=t_1}^n x_1^i\cdot \R[x_2,x_3]_{n-i}
\cap
\bigoplus\limits_{i=t_2}^n x_2^i\cdot \R[x_1,x_3]_{n-i}
\cap 
\bigoplus\limits_{i=t_3}^n x_3^i\cdot \R[x_1,x_2]_{n-i}
\right)$,  
it follows from Proposition \ref{prop:n>=t1+t2(odd)}(iii) that  (\ref{M123basis:n>=t1+t2(odd)-1}) and (\ref{M123basis:n>=t1+t2(odd)-2}) are in $\M_n(x_1)\cap \M_n(x_2)\cap \M_n(x_3)$. Combined with Proposition \ref{prop:dimM(x123)} the statement (i) follows.

(ii): The linear independence of  (\ref{M/M3basis:n>=t1+t2(odd)-1}) and (\ref{M/M3basis:n>=t1+t2(odd)-2}) follows from Proposition \ref{prop:n>=t1+t2(odd)}(i). 
By Proposition \ref{prop:n>=t1+t2(odd)}(iii) the cosets (\ref{M/M3basis:n>=t1+t2(odd)-1}) and (\ref{M/M3basis:n>=t1+t2(odd)-2}) are in 
$
\M_n/\M_n(x_3).
$  
By  Theorem \ref{thm:dimM=2(n+1)} and Proposition \ref{prop:dimM(x1)}(iii) the dimension of $\M_n/\M_n(x_3)$ is $2t_3$. The statement (ii) follows.
\end{proof}

\begin{prop}\label{prop:n>=t2+t3(odd)}
Suppose that $n$ is an odd integer with $n\geq t_2+t_3$. 
Let 
$$
r_i
=
\sum_{h=0}^i
(-1)^{\frac{h}{2}}
\sum_{j=t_3}^{n-t_2-h}
(-1)^{\frac{j}{2}}
{\left\lfloor \frac{n-t_2-i-j}{2}\right\rfloor+\left\lfloor\frac{i-h}{2}\right\rfloor \choose \left\lfloor\frac{i-h}{2}\right\rfloor}
c_{h,i,j}
[x_3]^{n-t_2}_j[x_1]^i_h[x_2]^{n-t_3}_{n-h-j}
$$
for all $i=0,1,\ldots,n-t_2-t_3$ where
\begin{align*}
c_{h,i,j}&=
\left\{
\begin{array}{ll}
\sigma_1^j \sigma_3
\qquad &\hbox{if $h$ is odd and $i$ is odd},
\\
\sigma_1^j
\qquad &\hbox{if $h$ is even and $i$ is even},
\\
\sigma_1
\qquad &\hbox{if $h$ is even, $i$ is odd and $j$ is odd},
\\
-\sigma_3
\qquad &\hbox{if $h$ is odd, $i$ is even and $j$ is even},
\\
\begin{pmatrix}
0 &0 \\
0 &0
\end{pmatrix}
\qquad &\hbox{else}
\end{array}
\right.
\end{align*}
for any integers $h,i,j$. 
Then the following hold:
\begin{enumerate}
\item $\{r_i\}_{i=0}^{n-t_2-t_3}$ are linearly independent over ${\rm Mat}_2(\C)$.

\item The following equations hold:
\begin{align}
(Z-\theta_i)r_i &=r_{i+1} \qquad (0\leq i\leq n-t_2-t_3-1),
\qquad 
(Z-\theta_{n-t_2-t_3}) r_{n-t_2-t_3}=0,
\label{e:n>=t2+t3(odd)-1}
\\
(X-\theta_i^*) r_i &=\varphi_i r_{i-1} \qquad (1\leq i\leq n-t_2-t_3),
\qquad 
(X-\theta_0^*) r_0=0,
\label{e:n>=t2+t3(odd)-2}
\end{align}
where
\begin{align*}
\theta_i &=(-1)^i
\textstyle(
k_2-k_1-i-\frac{1}{2}
)
\qquad 
(0\leq i\leq n-t_2-t_3),
\\
\theta_i^* &=(-1)^{i+1}
\textstyle(
k_3+k_2+n-i+\frac{1}{2}
)
\qquad 
(0\leq i\leq n-t_2-t_3),
\\
\varphi_i
&=
\left\{
\begin{array}{ll}
i(i-2k_3-2k_2-n-1)
\qquad
\hbox{if $i$ is even},
\\
(i+2k_1)(i-2k_2-n-1)
\qquad
\hbox{if $i$ is odd}
\end{array}
\right.
\qquad 
(1\leq i\leq n-t_2-t_3).
\end{align*}

\item $\D(r_i)=0$ for all $i=0,1,\ldots,n-t_2-t_3$. 
\end{enumerate}
\end{prop}
\begin{proof}
(i): Let $i$ be an integer with $0\leq i\leq n-t_2-t_3$.
By construction the coefficient of $x_1^h$ in $r_i$ is zero for all integers $h$ with $i<h\leq n-t_2-t_3$. 
Observe that the coefficient of $x_3^{n-t_2-i}x_1^i x_2^{t_2}$ in $r_i$ is 
\begin{gather}\label{coeff:n>=t2+t3(odd)}
(-1)^{\frac{n-t_2}{2}}
\prod_{h=n-t_2-i+1}^{n-t_2} m_3^{(h)}
\prod_{h=t_2+1}^{n-t_3} m_2^{(h)}
\times 
\left\{
\begin{array}{ll}
(-1)^\frac{3}{2}\sigma_2
\qquad 
&\hbox{if $i$ is odd},
\\
\begin{pmatrix}
1 &0
\\
0 &1
\end{pmatrix} \qquad 
&\hbox{if $i$ is even}.
\end{array}
\right.
\end{gather}
Clearly $\prod\limits_{h=t_2+1}^{n-t_3} m_2^{(h)}$ is nonzero.
Since $i\leq n-t_2-t_3$ the scalar $\prod\limits_{h=n-t_2-i+1}^{n-t_2} m_3^{(h)}$ is nonzero. 
Hence the matrix (\ref{coeff:n>=t2+t3(odd)}) is nonsingular. By the above comments the part (i) follows.

(ii): 
Let $i$ be a nonnegative integer and let $j$ be an integer with $t_3\leq j\leq n-t_2-i-1$. 
Using Theorem \ref{thm:BImodule_Mn}(i) yields that the coefficient of 
$[x_3]^{n-t_2}_j [x_1]^{i+1}_{i+1} [x_2]^{n-t_3}_{n-i-j-1} $ in $(Z-\theta_i)r_i$ is equal to 
$$
(-1)^{\frac{i-j+1}{2}}\sigma_3 c_{i,i,j}=(-1)^{\frac{i+j+1}{2}}c_{i+1,i+1,j}.
$$
Now let $h,i,j$ denote three integers with $0\leq h\leq i\leq n-t_2-t_3$ and $t_3\leq j\leq n-t_2-h$. Using Theorem \ref{thm:BImodule_Mn}(i) yields that the coefficient of 
$[x_3]^{n-t_2}_j [x_1]^i_h [x_2]^{n-t_3}_{n-h-j}$ in $(Z-\theta_i)r_i$ is equal to $(-1)^{\frac{h-j}{2}}$ times the sum of 
\begin{align*}
&
m_1^{(h)}
{\left\lfloor \frac{n-t_2-i-j}{2}\right\rfloor+\left\lfloor\frac{i-h+1}{2}\right\rfloor \choose \left\lfloor\frac{i-h+1}{2}\right\rfloor}
 \sigma_3 c_{h-1,i,j},
\\
&
m_2^{(n-h-j)}
{\left\lfloor \frac{n-t_2-i-j}{2}\right\rfloor+\left\lfloor\frac{i-h-1}{2}\right\rfloor \choose \left\lfloor\frac{i-h-1}{2}\right\rfloor}
 \sigma_3 c_{h+1,i,j},
\\
&
{\textstyle(
(-1)^{h+1} k_1
+
(-1)^{h+j} k_2
-
(-1)^j\theta_i
-
\frac{1}{2}
)}
{\left\lfloor \frac{n-t_2-i-j}{2}\right\rfloor+\left\lfloor\frac{i-h}{2}\right\rfloor \choose \left\lfloor\frac{i-h}{2}\right\rfloor}
c_{h,i,j}.
\end{align*}
It is straightforward to verify that the sum of the above three terms is equal to 
$$
(-1)^j
m_1^{(i+1)}
{\left\lfloor \frac{n-t_2-i-j-1}{2}\right\rfloor+\left\lfloor\frac{i-h+1}{2}\right\rfloor \choose \left\lfloor\frac{i-h+1}{2}\right\rfloor}
c_{h,i+1,j}. 
$$
By the above comments the equations given in (\ref{e:n>=t2+t3(odd)-1}) follow.

Using Theorem \ref{thm:BImodule_Mn}(i) yields that the coefficient of 
$[x_3]^{n-t_2}_j [x_1]^i_h [x_2]^{n-t_3}_{n-h-j}$ in $(X-\theta_i^*)r_i$ is equal to $(-1)^{\frac{j-h}{2}+1}$ times the sum of 
\begin{align*}
&
m_3^{(j)}
{\left\lfloor \frac{n-t_2-i-j+1}{2}\right\rfloor+\left\lfloor\frac{i-h}{2}\right\rfloor \choose \left\lfloor\frac{i-h}{2}\right\rfloor}
 \sigma_1 c_{h,i,j-1},
\\
&
m_2^{(n-h-j)}
{\left\lfloor \frac{n-t_2-i-j-1}{2}\right\rfloor+\left\lfloor\frac{i-h}{2}\right\rfloor \choose \left\lfloor\frac{i-h}{2}\right\rfloor}
 \sigma_1 c_{h,i,j+1},
\\
&
{\textstyle(
(-1)^{j} k_3
-
(-1)^{h+j} k_2
+
(-1)^h
\theta_i^*
+
\frac{1}{2}
)}
{\left\lfloor \frac{n-t_2-i-j}{2}\right\rfloor+\left\lfloor\frac{i-h}{2}\right\rfloor \choose \left\lfloor\frac{i-h}{2}\right\rfloor}
c_{h,i,j}.
\end{align*}
It is straightforward to verify that the sum of the above three terms is equal to $(-1)^{h+1}$ times 
$$
\left\{
\begin{array}{ll}
\begin{pmatrix}
0 &0
\\
0 &0
\end{pmatrix}
\qquad &\hbox{if $h=i$},
\\
(i-2k_3-2k_2-n-1)
\displaystyle{\left\lfloor \frac{n-t_2-i-j+1}{2}\right\rfloor+\left\lfloor\frac{i-h-1}{2}\right\rfloor \choose \left\lfloor\frac{i-h-1}{2}\right\rfloor}
c_{h,i-1,j}
\qquad &\hbox{if $h<i$ and $i$ is even},
\\
(i-2k_2-n-1)
\displaystyle{\left\lfloor \frac{n-t_2-i-j+1}{2}\right\rfloor+\left\lfloor\frac{i-h-1}{2}\right\rfloor \choose \left\lfloor\frac{i-h-1}{2}\right\rfloor}
c_{h,i-1,j}
\qquad &\hbox{if $h<i$ and $i$ is odd}.
\end{array}
\right.
$$
By the above comments the equations given in (\ref{e:n>=t2+t3(odd)-2}) follow.

(iii): It is routine to verify that $\D(r_0)=0$. Combined with (ii) the statement (iii) follows.
\end{proof}

\begin{lem}\label{lem:basis:n>=t2+t3(odd)}
Suppose that $n$ is an odd integer with $n\geq t_1+t_2+t_3$. 
Let $\{r_i\}_{i=0}^{n-t_2-t_3}$ be as in Proposition \ref{prop:n>=t2+t3(odd)}. Then the following hold:
\begin{enumerate}
\item $\M_n(x_1)\cap \M_n(x_2)\cap \M_n(x_3)$ has the basis 
\begin{align}
r_i\cdot 
\begin{pmatrix}
1
\\
0
\end{pmatrix}
\otimes
1 
\qquad 
(t_1\leq i\leq n-t_2-t_3), 
\label{M123basis:n>=t2+t3(odd)-1}
\\ 
r_i\cdot 
\begin{pmatrix}
0
\\
1
\end{pmatrix}
\otimes
1 
\qquad 
(t_1\leq i\leq n-t_2-t_3).
\label{M123basis:n>=t2+t3(odd)-2}
\end{align}

\item $\M_n/\M_n(x_1)$ has the basis
\begin{align}
r_i\cdot 
\begin{pmatrix}
1
\\
0
\end{pmatrix}
\otimes
1 
+
\M_n(x_1)
\qquad 
(0\leq i\leq t_1-1), 
\label{M/M1basis:n>=t2+t3(odd)-1}
\\ 
r_i\cdot 
\begin{pmatrix}
0
\\
1
\end{pmatrix}
\otimes
1 
+
\M_n(x_1)
\qquad 
(0\leq i\leq t_1-1).
\label{M/M1basis:n>=t2+t3(odd)-2}
\end{align}
\end{enumerate}
\end{lem}
\begin{proof}
(i): The linear independence of  (\ref{M123basis:n>=t2+t3(odd)-1}) and (\ref{M123basis:n>=t2+t3(odd)-2}) follows from Proposition \ref{prop:n>=t2+t3(odd)}(i). Since (\ref{M123basis:n>=t2+t3(odd)-1}) and (\ref{M123basis:n>=t2+t3(odd)-2}) are in 
$
\C^2\otimes\left(\bigoplus\limits_{i=t_1}^n x_1^i\cdot \R[x_2,x_3]_{n-i}
\cap
\bigoplus\limits_{i=t_2}^n x_2^i\cdot \R[x_1,x_3]_{n-i}
\cap 
\bigoplus\limits_{i=t_3}^n x_3^i\cdot \R[x_1,x_2]_{n-i}
\right)$,  
it follows from Proposition \ref{prop:n>=t2+t3(odd)}(iii) that (\ref{M123basis:n>=t2+t3(odd)-1}) and (\ref{M123basis:n>=t2+t3(odd)-2}) are in $\M_n(x_1)\cap \M_n(x_2)\cap \M_n(x_3)$. Combined with Proposition \ref{prop:dimM(x123)} the statement (i) follows.

(ii): The linear independence of  (\ref{M/M1basis:n>=t2+t3(odd)-1}) and (\ref{M/M1basis:n>=t2+t3(odd)-2}) follows from Proposition \ref{prop:n>=t2+t3(odd)}(i). 
By Proposition \ref{prop:n>=t2+t3(odd)}(iii) the cosets (\ref{M/M1basis:n>=t2+t3(odd)-1}) and (\ref{M/M1basis:n>=t2+t3(odd)-2}) are in 
$
\M_n/\M_n(x_1).
$  
By  Theorem \ref{thm:dimM=2(n+1)} and Proposition \ref{prop:dimM(x1)}(i) the dimension of $\M_n/\M_n(x_1)$ is $2t_1$. The statement (ii) follows.
\end{proof}

\begin{lem}\label{lem:basis:n+1>=t1+t2+t3(odd)}
Suppose that $n$ is an odd integer with $n\geq t_1+t_2+t_3$. 
Let 
$$
\{p_i\}_{i=0}^{n-t_1-t_3},
\qquad 
\{q_i\}_{i=0}^{n-t_1-t_2},
\qquad 
\{r_i\}_{i=0}^{n-t_2-t_3}
$$ 
be as in Propositions \ref{prop:n>=t1+t3(odd)}, \ref{prop:n>=t1+t2(odd)}, \ref{prop:n>=t2+t3(odd)} respectively. Then the following hold:
\begin{enumerate}
\item $\M_n(x_1)/\M_n(x_1)\cap \M_n(x_2)\cap \M_n(x_3)$ has the basis 
\begin{align}
p_i\cdot 
\begin{pmatrix}
1
\\
0
\end{pmatrix}
\otimes
1
+
\M_n(x_1)\cap \M_n(x_2)\cap \M_n(x_3)
\qquad 
(0\leq i\leq t_2-1), 
\label{M1/M123basis:n+1>=t1+t2+t3(odd)-1}
\\ 
p_i\cdot 
\begin{pmatrix}
0
\\
1
\end{pmatrix}
\otimes
1 
+
\M_n(x_1)\cap \M_n(x_2)\cap \M_n(x_3)
\qquad 
(0\leq i\leq t_2-1),
\label{M1/M123basis:n+1>=t1+t2+t3(odd)-2}
\\
q_i\cdot 
\begin{pmatrix}
1
\\
0
\end{pmatrix}
\otimes
1
+
\M_n(x_1)\cap \M_n(x_2)\cap \M_n(x_3)
\qquad 
(0\leq i\leq t_3-1), 
\label{M1/M123basis:n+1>=t1+t2+t3(odd)-3}
\\
q_i\cdot 
\begin{pmatrix}
0
\\
1
\end{pmatrix}
\otimes
1 
+
\M_n(x_1)\cap \M_n(x_2)\cap \M_n(x_3)
\qquad 
(0\leq i\leq t_3-1).
\label{M1/M123basis:n+1>=t1+t2+t3(odd)-4}
\end{align}

\item $\M_n(x_2)/\M_n(x_1)\cap \M_n(x_2)\cap \M_n(x_3)$ has the basis 
\begin{align}
q_i\cdot 
\begin{pmatrix}
1
\\
0
\end{pmatrix}
\otimes
1
+
\M_n(x_1)\cap \M_n(x_2)\cap \M_n(x_3)
\qquad 
(0\leq i\leq t_3-1), 
\label{M2/M123basis:n+1>=t1+t2+t3(odd)-1}
\\ 
q_i\cdot 
\begin{pmatrix}
0
\\
1
\end{pmatrix}
\otimes
1 
+
\M_n(x_1)\cap \M_n(x_2)\cap \M_n(x_3)
\qquad 
(0\leq i\leq t_3-1),
\label{M2/M123basis:n+1>=t1+t2+t3(odd)-2}
\\
r_i\cdot 
\begin{pmatrix}
1
\\
0
\end{pmatrix}
\otimes
1
+
\M_n(x_1)\cap \M_n(x_2)\cap \M_n(x_3)
\qquad 
(0\leq i\leq t_1-1), 
\label{M2/M123basis:n+1>=t1+t2+t3(odd)-3}
\\
r_i\cdot 
\begin{pmatrix}
0
\\
1
\end{pmatrix}
\otimes
1 
+
\M_n(x_1)\cap \M_n(x_2)\cap \M_n(x_3)
\qquad 
(0\leq i\leq t_1-1).
\label{M2/M123basis:n+1>=t1+t2+t3(odd)-4}
\end{align}

\item $\M_n(x_3)/\M_n(x_1)\cap \M_n(x_2)\cap \M_n(x_3)$ has the basis 
\begin{align}
r_i\cdot 
\begin{pmatrix}
1
\\
0
\end{pmatrix}
\otimes
1
+
\M_n(x_1)\cap \M_n(x_2)\cap \M_n(x_3)
\qquad 
(0\leq i\leq t_1-1), 
\label{M3/M123basis:n+1>=t1+t2+t3(odd)-1}
\\ 
r_i\cdot 
\begin{pmatrix}
0
\\
1
\end{pmatrix}
\otimes
1 
+
\M_n(x_1)\cap \M_n(x_2)\cap \M_n(x_3)
\qquad 
(0\leq i\leq t_1-1),
\label{M3/M123basis:n+1>=t1+t2+t3(odd)-2}
\\
p_i\cdot 
\begin{pmatrix}
1
\\
0
\end{pmatrix}
\otimes
1
+
\M_n(x_1)\cap \M_n(x_2)\cap \M_n(x_3)
\qquad 
(0\leq i\leq t_2-1), 
\label{M3/M123basis:n+1>=t1+t2+t3(odd)-3}
\\
p_i\cdot 
\begin{pmatrix}
0
\\
1
\end{pmatrix}
\otimes
1 
+
\M_n(x_1)\cap \M_n(x_2)\cap \M_n(x_3)
\qquad 
(0\leq i\leq t_2-1).
\label{M3/M123basis:n+1>=t1+t2+t3(odd)-4}
\end{align}
\end{enumerate}
\end{lem}
\begin{proof}
(i): 
Observe that (\ref{M1/M123basis:n+1>=t1+t2+t3(odd)-1})--(\ref{M1/M123basis:n+1>=t1+t2+t3(odd)-4}) are in 
$
\M_n(x_1)/\M_n(x_1)\cap \M_n(x_2)\cap \M_n(x_3).
$ 
The linear independence of  (\ref{M1/M123basis:n+1>=t1+t2+t3(odd)-1})--(\ref{M1/M123basis:n+1>=t1+t2+t3(odd)-4}) follows from Propositions \ref{prop:n>=t1+t3(odd)}(i) and \ref{prop:n>=t1+t2(odd)}(i). 
By Propositions \ref{prop:dimM(x1)}(i) and \ref{prop:dimM(x123)} the dimension of $\M_n(x_1)/\M_n(x_1)\cap \M_n(x_2)\cap \M_n(x_3)$ is $2(t_2+t_3)$.
The statement (i) follows.

(ii): 
Observe that (\ref{M2/M123basis:n+1>=t1+t2+t3(odd)-1})--(\ref{M2/M123basis:n+1>=t1+t2+t3(odd)-4}) are in 
$
\M_n(x_2)/\M_n(x_1)\cap \M_n(x_2)\cap \M_n(x_3).
$ 
The linear independence of  (\ref{M2/M123basis:n+1>=t1+t2+t3(odd)-1})--(\ref{M2/M123basis:n+1>=t1+t2+t3(odd)-4}) follows from Propositions \ref{prop:n>=t1+t2(odd)}(i) and \ref{prop:n>=t2+t3(odd)}(i). 
By Propositions \ref{prop:dimM(x1)}(ii) and \ref{prop:dimM(x123)} the dimension of $\M_n(x_2)/\M_n(x_1)\cap \M_n(x_2)\cap \M_n(x_3)$ is $2(t_1+t_3)$.
The statement (ii) follows.

(iii): 
Observe that (\ref{M3/M123basis:n+1>=t1+t2+t3(odd)-1})--(\ref{M3/M123basis:n+1>=t1+t2+t3(odd)-4}) are in 
$
\M_n(x_3)/\M_n(x_1)\cap \M_n(x_2)\cap \M_n(x_3).
$ 
The linear independence of  (\ref{M3/M123basis:n+1>=t1+t2+t3(odd)-1})--(\ref{M3/M123basis:n+1>=t1+t2+t3(odd)-4}) follows from Propositions \ref{prop:n>=t2+t3(odd)}(i) and \ref{prop:n>=t1+t3(odd)}(i). 
By Propositions \ref{prop:dimM(x1)}(iii) and \ref{prop:dimM(x123)} the dimension of $\M_n(x_3)/\M_n(x_1)\cap \M_n(x_2)\cap \M_n(x_3)$ is $2(t_1+t_2)$.
The statement (iii) follows. 
\end{proof}

\begin{thm}\label{thm:n+1>=t1+t2+t3(odd)}
Suppose that $n$ is an odd integer with $n\geq t_1+t_2+t_3$. Then the following hold: 
\begin{enumerate}
\item The $\BI$-module $\M_n(x_1)\cap \M_n(x_2)\cap \M_n(x_3)$ is isomorphic to a direct sum of two copies of 
\begin{gather}\label{M123(odd)}
O_{n-t_1-t_2-t_3}
\textstyle
(
k_1+\frac{n+1}{2},
k_2+\frac{n+1}{2},
k_3+\frac{n+1}{2}
).
\end{gather}

\item The $\BI$-module $\M_n/\M_n(x_1)$ is isomorphic to 
a direct sum of two copies of 
\begin{gather}\label{M/M1(odd)}
O_{t_1-1}(-k_1-k_2-k_3-n-1,k_3,k_2).
\end{gather}

\item The $\BI$-module $\M_n/\M_n(x_2)$ is isomorphic to 
a direct sum of two copies of 
\begin{gather}\label{M/M2(odd)}
O_{t_2-1}(k_3,-k_1-k_2-k_3-n-1,k_1).
\end{gather}

\item The $\BI$-module $\M_n/\M_n(x_3)$ is isomorphic to 
a direct sum of two copies of 
\begin{gather}\label{M/M3(odd)}
O_{t_3-1}(k_2,k_1,-k_1-k_2-k_3-n-1).
\end{gather}

\item The $\BI$-module $\M_n(x_1)/\M_n(x_1)\cap \M_n(x_2)\cap \M_n(x_3)$ is isomorphic to a direct sum of two copies of (\ref{M/M2(odd)}) and (\ref{M/M3(odd)}).

\item The $\BI$-module $\M_n(x_2)/\M_n(x_1)\cap \M_n(x_2)\cap \M_n(x_3)$ is isomorphic to a direct sum of two copies of (\ref{M/M1(odd)}) and (\ref{M/M3(odd)}).

\item The $\BI$-module $\M_n(x_3)/\M_n(x_1)\cap \M_n(x_2)\cap \M_n(x_3)$ is isomorphic to a direct sum of two copies of (\ref{M/M1(odd)}) and (\ref{M/M2(odd)}).
\end{enumerate}
Moreover the $\BI$-modules (\ref{M123(odd)})--(\ref{M/M3(odd)}) are irreducible.
\end{thm}
\begin{proof}
(i):  Let $V$ and $V'$ denote the subspaces of $\M_n$ spanned by (\ref{M123basis:n>=t1+t3(odd)-1}) and (\ref{M123basis:n>=t1+t3(odd)-2}), respectively. It follows from Lemma \ref{lem:basis:n>=t1+t3(odd)}(i) that $\M_n(x_1)\cap \M_n(x_2)\cap \M_n(x_3)=V\oplus V'$. It follows from Theorem \ref{thm:BImodule_Mn}(ii) and Proposition \ref{prop:n>=t1+t3(odd)}(i) that $V$ and $V'$ are two isomorphic $\BI$-modules. 
Compared with Proposition \ref{prop:Od} both are isomorphic to 
\begin{gather}\label{M123'(odd)}
O_{n-t_1-t_2-t_3}
(
\textstyle
k_1+\frac{n+1}{2},
-k_2-\frac{n+1}{2},
-k_3-\frac{n+1}{2}
)^{(1,-1)}.
\end{gather}
Using Theorem \ref{thm:irr_O} yields that the $\BI$-module (\ref{M123'(odd)}) is irreducible. By Theorem \ref{thm:onto2_O} the $\BI$-module (\ref{M123'(odd)}) is isomorphic to (\ref{M123(odd)}). The statement (i) follows.

(ii): Let $V$ and $V'$ denote the subspaces of $\M_n/\M_n(x_1)$ spanned by (\ref{M/M1basis:n>=t2+t3(odd)-1}) and (\ref{M/M1basis:n>=t2+t3(odd)-2}), respectively. 
It follows from Lemma \ref{lem:basis:n>=t2+t3(odd)}(ii) that $\M_n/\M_n(x_1)=V\oplus V'$. It follows from Theorem \ref{thm:BImodule_Mn}(ii) and Proposition \ref{prop:n>=t2+t3(odd)}(ii) that $V$ and $V'$ are two isomorphic $\BI$-modules. 
Compared with Proposition \ref{prop:Od} both are isomorphic to 
\begin{gather}\label{M/M1'(odd)}
O_{t_1-1}(-k_2,-k_1-k_2-k_3-n-1,-k_3)^{((-1,1),(1\,2\,3))}.
\end{gather}
Using Theorem \ref{thm:irr_O} yields that the $\BI$-module (\ref{M/M1'(odd)}) is irreducible. By Theorem \ref{thm:onto2_O} the $\BI$-module (\ref{M/M1'(odd)}) is isomorphic to (\ref{M/M1(odd)}). The statement (ii) follows.

(iii): Let $V$ and $V'$ denote the subspaces of $\M_n/\M_n(x_2)$ spanned by (\ref{M/M2basis:n>=t1+t3(odd)-1}) and (\ref{M/M2basis:n>=t1+t3(odd)-2}), respectively. 
It follows from Lemma \ref{lem:basis:n>=t1+t3(odd)}(ii) that $\M_n/\M_n(x_2)=V\oplus V'$. It follows from Theorem \ref{thm:BImodule_Mn}(ii) and Proposition \ref{prop:n>=t1+t3(odd)}(ii) that $V$ and $V'$ are two isomorphic $\BI$-modules. 
Compared with Proposition \ref{prop:Od} both are isomorphic to 
\begin{gather}\label{M/M2'(odd)}
O_{t_2-1}(-k_3,-k_1-k_2-k_3-n-1,-k_1)^{(-1,1)}.
\end{gather}
Using Theorem \ref{thm:irr_O} yields that the $\BI$-module (\ref{M/M2'(odd)}) is irreducible. By Theorem \ref{thm:onto2_O} the $\BI$-module (\ref{M/M2'(odd)}) is isomorphic to (\ref{M/M2(odd)}). The statement (iii) follows.

(iv): Let $V$ and $V'$ denote the subspaces of $\M_n/\M_n(x_3)$ spanned by (\ref{M/M3basis:n>=t1+t2(odd)-1}) and (\ref{M/M3basis:n>=t1+t2(odd)-2}), respectively. 
It follows from Lemma \ref{lem:basis:n>=t1+t2(odd)}(ii) that $\M_n/\M_n(x_3)=V\oplus V'$. It follows from Theorem \ref{thm:BImodule_Mn}(ii) and Proposition \ref{prop:n>=t1+t2(odd)}(ii) that $V$ and $V'$ are two isomorphic $\BI$-modules. 
Compared with Proposition \ref{prop:Od} both are isomorphic to 
\begin{gather}\label{M/M3'(odd)}
O_{t_3-1}(-k_1,-k_1-k_2-k_3-n-1,-k_2)^{((-1,1),(1\,3\,2))}.
\end{gather}
Using Theorem \ref{thm:irr_O} yields that the $\BI$-module (\ref{M/M3'(odd)}) is irreducible. By Theorem \ref{thm:onto2_O} the $\BI$-module (\ref{M/M3'(odd)}) is isomorphic to (\ref{M/M3(odd)}). The statement (iv) follows.

(v): 
Let $V,V',W,W'$ denote the subspaces of $\M_n(x_1)/\M_n(x_1)\cap \M_n(x_2)\cap \M_n(x_3)$ spanned by (\ref{M1/M123basis:n+1>=t1+t2+t3(odd)-1})--(\ref{M1/M123basis:n+1>=t1+t2+t3(odd)-4}), respectively. 
It follows from Lemma \ref{lem:basis:n+1>=t1+t2+t3(odd)}(i) that 
$$
\M_n(x_1)/\M_n(x_1)\cap \M_n(x_2)\cap \M_n(x_3)=V\oplus V'\oplus W\oplus W'.
$$ 
Similar to the proof of Theorem \ref{thm:n+1>=t1+t2+t3(odd)}(iii) it follows that $V$ and $V'$ are isomorphic to the $\BI$-module (\ref{M/M2(odd)}). Similar to the proof of Theorem \ref{thm:n+1>=t1+t2+t3(odd)}(iv) it follows that that $W$ and $W'$ are isomorphic to the $\BI$-module (\ref{M/M3(odd)}). The statement (v) follows.

(vi): 
Let $V,V',W,W'$ denote the subspaces of $\M_n(x_1)/\M_n(x_1)\cap \M_n(x_2)\cap \M_n(x_3)$ spanned by (\ref{M2/M123basis:n+1>=t1+t2+t3(odd)-1})--(\ref{M2/M123basis:n+1>=t1+t2+t3(odd)-4}), respectively. 
It follows from Lemma \ref{lem:basis:n+1>=t1+t2+t3(odd)}(ii) that 
$$
\M_n(x_2)/\M_n(x_1)\cap \M_n(x_2)\cap \M_n(x_3)=V\oplus V'\oplus W\oplus W'.
$$ 
Similar to the proof of Theorem \ref{thm:n+1>=t1+t2+t3(odd)}(iv) it follows that $V$ and $V'$ are isomorphic to the $\BI$-module (\ref{M/M3(odd)}). Similar to the proof of Theorem \ref{thm:n+1>=t1+t2+t3(odd)}(ii) it follows that $W$ and $W'$ are isomorphic to the $\BI$-module (\ref{M/M1(odd)}). The statement (vi) follows.

(vii): 
Let $V,V',W,W'$ denote the subspaces of $\M_n(x_3)/\M_n(x_1)\cap \M_n(x_2)\cap \M_n(x_3)$ spanned by (\ref{M3/M123basis:n+1>=t1+t2+t3(odd)-1})--(\ref{M3/M123basis:n+1>=t1+t2+t3(odd)-4}), respectively. 
It follows from Lemma \ref{lem:basis:n+1>=t1+t2+t3(odd)}(iii) that 
$$
\M_n(x_3)/\M_n(x_1)\cap \M_n(x_2)\cap \M_n(x_3)=V\oplus V'\oplus W\oplus W'.
$$ 
Similar to the proof of Theorem \ref{thm:n+1>=t1+t2+t3(odd)}(ii) it follows that $V$ and $V'$ are isomorphic to the $\BI$-module (\ref{M/M1(odd)}). Similar to the proof of Theorem \ref{thm:n+1>=t1+t2+t3(odd)}(iii) it follows that $W$ and $W'$ are isomorphic to the $\BI$-module (\ref{M/M2(odd)}). The statement (vii) follows.
\end{proof}

\begin{prop}\label{prop:n>=t1+t3(even)}
Suppose that $n$ is an even integer with $n\geq t_1+t_3$. 
Let 
$$
p_i
=
\sum_{h=0}^i
(-1)^{-\frac{h}{2}}
\sum_{j=t_1}^{n-t_3-h}
(-1)^{\frac{j}{2}}
{\left\lfloor \frac{n-t_3-i-j}{2}\right\rfloor+\left\lfloor\frac{i-h}{2}\right\rfloor \choose \left\lfloor\frac{i-h}{2}\right\rfloor}
c_{h,i,j}
\otimes
[x_1]^{n-t_3}_j[x_2]^i_h[x_3]^{n-t_1}_{n-h-j}
$$
for all $i=0,1,\ldots,n-t_1-t_3$ where
\begin{align*}
c_{h,i,j}&=
\left\{
\begin{array}{ll}
\sigma_2^j \sigma_1 
\qquad &\hbox{if $h$ is odd and $i$ is odd},
\\
\sigma_2^j
\qquad &\hbox{if $h$ is even and $i$ is even},
\\
(-1)^\frac{1}{2} \sigma_3
\qquad &\hbox{if $h$ is odd, $i$ is even and $j$ is odd},
\\
\begin{pmatrix}
1 &0
\\
0 &1
\end{pmatrix}
\qquad &\hbox{if $h$ is even, $i$ is odd and $j$ is even},
\\
\begin{pmatrix}
0&0
\\
0 &0
\end{pmatrix}
\qquad &\hbox{else}
\end{array}
\right.
\end{align*}
for any integers $h,i,j$. 
Then the following hold:
\begin{enumerate}
\item $\{p_i\}_{i=0}^{n-t_1-t_3}$ are linearly independent over ${\rm Mat}_2(\C)$.

\item The following equations hold:
\begin{align}
(X-\theta_i)p_i &=p_{i+1} \qquad (0\leq i\leq n-t_1-t_3-1),
\qquad 
(X-\theta_{n-t_1-t_3}) p_{n-t_1-t_3}=0,
\label{e:n>=t1+t3(even)-1}
\\
(Y-\theta_i^*) p_i &=\varphi_i p_{i-1} \qquad (1\leq i\leq n-t_1-t_3),
\qquad 
(Y-\theta_0^*) p_0=0,
\label{e:n>=t1+t3(even)-2}
\end{align}
where
\begin{align*}
\theta_i &=(-1)^i
\textstyle(
k_3-k_2-i-\frac{1}{2}
)
\qquad 
(0\leq i\leq n-t_1-t_3),
%\hbox{for $i=0,1,\ldots,n-t_1-t_3$},
\\
\theta_i^* &=(-1)^i
\textstyle(
k_1+k_3+n-i+\frac{1}{2}
)
\qquad 
(0\leq i\leq n-t_1-t_3),
\\
\varphi_i
&=
\left\{
\begin{array}{ll}
i(2 k_3+n-i+1)
\qquad
\hbox{if $i$ is even},
%\hbox{for $i=2,4,\ldots,n-t_1-t_3$},
\\
(i+2 k_2)(2 k_1+2 k_3+n-i+1)
\qquad
\hbox{if $i$ is odd}
%\hbox{for $i=1,3,\ldots,n-t_1-t_3-1$}.
\end{array}
\right.
\qquad 
(1\leq i\leq n-t_1-t_3).
\end{align*}

\item $\D(p_i)=0$ for all $i=0,1,\ldots,n-t_1-t_3$.
\end{enumerate}
\end{prop}
\begin{proof}
(i): Let $i$ be an integer with $0\leq i\leq n-t_1-t_3$.
By construction the coefficient of $x_2^h$ in $p_i$ is zero for all integers $h$ with $i<h\leq n-t_1-t_3$. Observe that the coefficient of $x_1^{n-t_3-i}x_2^i x_3^{t_3}$ in $p_i$ is 
\begin{gather}\label{coeff:n>=t1+t3(even)}
(-1)^{\frac{n-t_3}{2}-i}
\prod_{h=n-t_3-i+1}^{n-t_3} m_1^{(h)}
\prod_{h=t_3+1}^{n-t_1} m_3^{(h)}
\times 
\left\{
\begin{array}{ll}
\sigma_1
\qquad 
&\hbox{if $i$ is odd},
\\
\sigma_2 
\qquad 
&\hbox{if $i$ is even}.
\end{array}
\right.
\end{gather}
Clearly $\prod\limits_{h=t_3+1}^{n-t_1} m_3^{(h)}$ is nonzero.
Since $i\leq n-t_1-t_3$ the scalar $\prod\limits_{h=n-t_3-i+1}^{n-t_3} m_1^{(h)}$ is nonzero. 
Hence the matrix (\ref{coeff:n>=t1+t3(even)}) is nonsingular. By the above comments the part (i) follows.

(ii): 
Let $i$ be a nonnegative integer and let $j$ be an integer with $t_1\leq j\leq n-t_3-i-1$. 
Using Theorem \ref{thm:BImodule_Mn}(i) yields that the coefficient of 
$[x_1]^{n-t_3}_j [x_2]^{i+1}_{i+1} [x_3]^{n-t_1}_{n-i-j-1} $ in $(X-\theta_i)p_i$ is equal to 
$$
(-1)^{-\frac{i+j+1}{2}}\sigma_1 c_{i,i,j}=(-1)^{\frac{-i+j-1}{2}}c_{i+1,i+1,j}.
$$
Now let $h,i,j$ denote three integers with $0\leq h\leq i\leq n-t_1-t_3$ and $t_1\leq j\leq n-t_3-h$. Using Theorem \ref{thm:BImodule_Mn}(i) yields that the coefficient of 
$[x_1]^{n-t_3}_j [x_2]^i_h [x_3]^{n-t_1}_{n-h-j}$ in $(X-\theta_i)p_i$ is equal to $(-1)^{-\frac{h+j}{2}}$ times the sum of 
\begin{align*}
&
m_2^{(h)}
{\left\lfloor \frac{n-t_3-i-j}{2}\right\rfloor+\left\lfloor\frac{i-h+1}{2}\right\rfloor \choose \left\lfloor\frac{i-h+1}{2}\right\rfloor}
 \sigma_1 c_{h-1,i,j},
\\
&
m_3^{(n-h-j)}
{\left\lfloor \frac{n-t_3-i-j}{2}\right\rfloor+\left\lfloor\frac{i-h-1}{2}\right\rfloor \choose \left\lfloor\frac{i-h-1}{2}\right\rfloor}
 \sigma_1 c_{h+1,i,j},
\\
&
{\textstyle(
(-1)^h k_2
+
(-1)^{h+j} k_3
-
(-1)^j\theta_i
+
\frac{1}{2}
)}
{\left\lfloor \frac{n-t_3-i-j}{2}\right\rfloor+\left\lfloor\frac{i-h}{2}\right\rfloor \choose \left\lfloor\frac{i-h}{2}\right\rfloor}
c_{h,i,j}.
\end{align*}
It is straightforward to verify that the sum of the above three terms is equal to 
$$
(-1)^j
m_2^{(i+1)}
{\left\lfloor \frac{n-t_3-i-j-1}{2}\right\rfloor+\left\lfloor\frac{i-h+1}{2}\right\rfloor \choose \left\lfloor\frac{i-h+1}{2}\right\rfloor}
c_{h,i+1,j}.  
$$
By the above comments the equations given in (\ref{e:n>=t1+t3(even)-1}) follow.

Using Theorem \ref{thm:BImodule_Mn}(i) yields that the coefficient of 
$[x_1]^{n-t_3}_j [x_2]^i_h [x_3]^{n-t_1}_{n-h-j}$ in $(Y-\theta_i^*)p_i$ is equal to $(-1)^\frac{h+j}{2}$ times the sum of 
\begin{align*}
&
m_1^{(j)}
{\left\lfloor \frac{n-t_3-i-j+1}{2}\right\rfloor+\left\lfloor\frac{i-h}{2}\right\rfloor \choose \left\lfloor\frac{i-h}{2}\right\rfloor}
 \sigma_2 c_{h,i,j-1},
\\
&
m_3^{(n-h-j)}
{\left\lfloor \frac{n-t_3-i-j-1}{2}\right\rfloor+\left\lfloor\frac{i-h}{2}\right\rfloor \choose \left\lfloor\frac{i-h}{2}\right\rfloor}
 \sigma_2 c_{h,i,j+1},
\\
&
{\textstyle(
(-1)^j k_1
+
(-1)^{h+j} k_3
-
(-1)^h
\theta_i^*
+
\frac{1}{2}
)}
{\left\lfloor \frac{n-t_3-i-j}{2}\right\rfloor+\left\lfloor\frac{i-h}{2}\right\rfloor \choose \left\lfloor\frac{i-h}{2}\right\rfloor}
c_{h,i,j}.
\end{align*}
It is straightforward to verify that the sum of the above three terms is equal to $(-1)^h$ times 
$$
\left\{
\begin{array}{ll}
\begin{pmatrix}
0 &0
\\
0 &0
\end{pmatrix}
\qquad &\hbox{if $h=i$},
\\
(2k_3+n-i+1)
\displaystyle{\left\lfloor \frac{n-t_3-i-j+1}{2}\right\rfloor+\left\lfloor\frac{i-h-1}{2}\right\rfloor \choose \left\lfloor\frac{i-h-1}{2}\right\rfloor}
c_{h,i-1,j}
\qquad &\hbox{if $h<i$ and $i$ is even},
\\
(2k_1+2k_3+n-i+1)
\displaystyle{\left\lfloor \frac{n-t_3-i-j+1}{2}\right\rfloor+\left\lfloor\frac{i-h-1}{2}\right\rfloor \choose \left\lfloor\frac{i-h-1}{2}\right\rfloor}
c_{h,i-1,j}
\qquad &\hbox{if $h<i$ and $i$ is odd}.
\end{array}
\right.
$$
By the above comments the equations given in (\ref{e:n>=t1+t3(even)-2}) follow.

(iii): It is routine to verify that $\D(p_0)=0$. Combined with (ii) the statement (iii) follows.
\end{proof}

\begin{lem}\label{lem:basis:n>=t1+t3(even)}
Suppose that $n$ is an even integer with $n\geq t_1+t_2+t_3$. 
Let $\{p_i\}_{i=0}^{n-t_1-t_3}$ be as in Proposition \ref{prop:n>=t1+t3(even)}. Then the following hold:
\begin{enumerate}
\item $\M_n(x_1)\cap \M_n(x_2)\cap \M_n(x_3)$ has the basis 
\begin{align}
p_i\cdot 
\begin{pmatrix}
1
\\
0
\end{pmatrix}
\otimes
1 
\qquad 
(t_2\leq i\leq n-t_1-t_3), 
\label{M123basis:n>=t1+t3(even)-1}
\\ 
p_i\cdot 
\begin{pmatrix}
0
\\
1
\end{pmatrix}
\otimes
1 
\qquad 
(t_2\leq i\leq n-t_1-t_3).
\label{M123basis:n>=t1+t3(even)-2}
\end{align}

\item $\M_n/\M_n(x_2)$ has the basis
\begin{align}
p_i\cdot 
\begin{pmatrix}
1
\\
0
\end{pmatrix}
\otimes
1 
+
\M_n(x_2)
\qquad 
(0\leq i\leq t_2-1), 
\label{M/M2basis:n>=t1+t3(even)-1}
\\ 
p_i\cdot 
\begin{pmatrix}
0
\\
1
\end{pmatrix}
\otimes
1 
+
\M_n(x_2)
\qquad 
(0\leq i\leq t_2-1).
\label{M/M2basis:n>=t1+t3(even)-2}
\end{align}
\end{enumerate}
\end{lem}
\begin{proof}
(i): The linear independence of (\ref{M123basis:n>=t1+t3(even)-1}) and (\ref{M123basis:n>=t1+t3(even)-2}) follows from Proposition \ref{prop:n>=t1+t3(even)}(i). Since (\ref{M123basis:n>=t1+t3(even)-1}) and (\ref{M123basis:n>=t1+t3(even)-2}) are in 
$
\C^2\otimes\left(\bigoplus\limits_{i=t_1}^n x_1^i\cdot \R[x_2,x_3]_{n-i}
\cap
\bigoplus\limits_{i=t_2}^n x_2^i\cdot \R[x_1,x_3]_{n-i}
\cap 
\bigoplus\limits_{i=t_3}^n x_3^i\cdot \R[x_1,x_2]_{n-i}
\right)$,  
it follows from Proposition \ref{prop:n>=t1+t3(even)}(iii) that (\ref{M123basis:n>=t1+t3(even)-1}) and (\ref{M123basis:n>=t1+t3(even)-2}) are in $\M_n(x_1)\cap \M_n(x_2)\cap \M_n(x_3)$. Combined with Proposition \ref{prop:dimM(x123)} the statement (i) follows.

(ii): The linear independence of (\ref{M/M2basis:n>=t1+t3(even)-1}) and (\ref{M/M2basis:n>=t1+t3(even)-2}) follows from Proposition \ref{prop:n>=t1+t3(even)}(i). 
By Proposition \ref{prop:n>=t1+t3(even)}(iii) the cosets (\ref{M/M2basis:n>=t1+t3(even)-1}) and (\ref{M/M2basis:n>=t1+t3(even)-2}) are in 
$
\M_n/\M_n(x_2).
$  
By  Theorem \ref{thm:dimM=2(n+1)} and Proposition \ref{prop:dimM(x1)}(ii) the dimension of $\M_n/\M_n(x_2)$ is $2t_2$. The statement (ii) follows.
\end{proof}

\begin{prop}\label{prop:n>=t1+t2(even)}
Suppose that $n$ is an even integer with $n\geq t_1+t_2$. 
Let 
$$
q_i
=
\sum_{h=0}^i
(-1)^{-\frac{h}{2}}
\sum_{j=t_2}^{n-t_1-h}
(-1)^{\frac{j}{2}}
{\left\lfloor \frac{n-t_1-i-j}{2}\right\rfloor+\left\lfloor\frac{i-h}{2}\right\rfloor \choose \left\lfloor\frac{i-h}{2}\right\rfloor}
c_{h,i,j}
\otimes
[x_2]^{n-t_1}_j[x_3]^i_h[x_1]^{n-t_2}_{n-h-j}
$$
for all $i=0,1,\ldots,n-t_1-t_2$ where
\begin{align*}
c_{h,i,j}&=
\left\{
\begin{array}{ll}
\sigma_3^j \sigma_2 
\qquad &\hbox{if $h$ is odd and $i$ is odd},
\\
\sigma_3^j
\qquad &\hbox{if $h$ is even and $i$ is even},
\\
(-1)^\frac{1}{2} \sigma_1
\qquad &\hbox{if $h$ is odd, $i$ is even and $j$ is odd},
\\
\begin{pmatrix}
1 &0
\\
0 &1
\end{pmatrix}
\qquad &\hbox{if $h$ is even, $i$ is odd and $j$ is even},
\\
\begin{pmatrix}
0&0
\\
0 &0
\end{pmatrix}
\qquad &\hbox{else}
\end{array}
\right.
\end{align*}
for any integers $h,i,j$. 
Then the following hold:
\begin{enumerate}
\item $\{q_i\}_{i=0}^{n-t_1-t_2}$ are linearly independent over ${\rm Mat}_2(\C)$.

\item The following equations hold:
\begin{align}
(Y-\theta_i)q_i &=q_{i+1} \qquad (0\leq i\leq n-t_1-t_2-1),
\qquad 
(Y-\theta_{n-t_1-t_2}) q_{n-t_1-t_2}=0,
\label{e:n>=t1+t2(even)-1}
\\
(Z-\theta_i^*) q_i &=\varphi_i q_{i-1} \qquad (1\leq i\leq n-t_1-t_2),
\qquad 
(Z-\theta_0^*) q_0=0,
\label{e:n>=t1+t2(even)-2}
\end{align}
where
\begin{align*}
\theta_i &=(-1)^i
\textstyle(
k_1-k_3-i-\frac{1}{2}
)
\qquad 
(0\leq i\leq n-t_1-t_2),
\\
\theta_i^* &=(-1)^i
\textstyle(
k_1+k_2+n-i+\frac{1}{2}
)
\qquad 
(0\leq i\leq n-t_1-t_2),
\\
\varphi_i
&=
\left\{
\begin{array}{ll}
i(2 k_1+n-i+1)
\qquad
\hbox{if $i$ is even},
%\hbox{for $i=2,4,\ldots,n-t_1-t_2$},
\\
(i+2 k_3)(2 k_2+2 k_1+n-i+1)
\qquad
\hbox{if $i$ is odd}
%\hbox{for $i=1,3,\ldots,n-t_1-t_2-1$}.
\end{array}
\right.
\qquad (1\leq i\leq n-t_1-t_2).
\end{align*}

\item $\D(q_i)=0$ for all $i=0,1,\ldots,n-t_1-t_2$.
\end{enumerate}
\end{prop}
\begin{proof}
(i): Let $i$ be an integer with $0\leq i\leq n-t_1-t_2$.
By construction the coefficient of $x_3^h$ in $q_i$ is zero for all integers $h$ with $i<h\leq n-t_1-t_2$. Observe that the coefficient of $x_2^{n-t_1-i}x_3^i x_1^{t_1}$ in $q_i$ is 
\begin{gather}\label{coeff:n>=t1+t2(even)}
(-1)^{\frac{n-t_1}{2}-i}
\prod_{h=n-t_1-i+1}^{n-t_1} m_2^{(h)}
\prod_{h=t_1+1}^{n-t_2} m_1^{(h)}
\times 
\left\{
\begin{array}{ll}
\sigma_2
\qquad 
&\hbox{if $i$ is odd},
\\
\sigma_3 
\qquad 
&\hbox{if $i$ is even}.
\end{array}
\right.
\end{gather}
Clearly $\prod\limits_{h=t_1+1}^{n-t_2} m_1^{(h)}$ is nonzero.
Since $i\leq n-t_1-t_2$ the scalar $\prod\limits_{h=n-t_1-i+1}^{n-t_1} m_2^{(h)}$ is nonzero. 
Hence the matrix (\ref{coeff:n>=t1+t2(even)}) is nonsingular. By the above comments the part (i) follows.

(ii): 
Let $i$ be a nonnegative integer and let $j$ be an integer with $t_2\leq j\leq n-t_1-i-1$. 
Using Theorem \ref{thm:BImodule_Mn}(i) yields that the coefficient of 
$[x_2]^{n-t_1}_j [x_3]^{i+1}_{i+1} [x_1]^{n-t_2}_{n-i-j-1} $ in $(Y-\theta_i)q_i$ is equal to 
$$
(-1)^{-\frac{i+j+1}{2}}\sigma_2 c_{i,i,j}=(-1)^{\frac{-i+j-1}{2}}c_{i+1,i+1,j}.
$$
Now let $h,i,j$ denote three integers with $0\leq h\leq i\leq n-t_1-t_2$ and $t_2\leq j\leq n-t_1-h$. Using Theorem \ref{thm:BImodule_Mn}(i) yields that the coefficient of 
$[x_2]^{n-t_1}_j [x_3]^i_h [x_1]^{n-t_2}_{n-h-j}$ in $(Y-\theta_i)q_i$ is equal to $(-1)^{-\frac{h+j}{2}}$ times the sum of 
\begin{align*}
&
m_3^{(h)}
{\left\lfloor \frac{n-t_1-i-j}{2}\right\rfloor+\left\lfloor\frac{i-h+1}{2}\right\rfloor \choose \left\lfloor\frac{i-h+1}{2}\right\rfloor}
 \sigma_2 c_{h-1,i,j},
\\
&
m_1^{(n-h-j)}
{\left\lfloor \frac{n-t_1-i-j}{2}\right\rfloor+\left\lfloor\frac{i-h-1}{2}\right\rfloor \choose \left\lfloor\frac{i-h-1}{2}\right\rfloor}
 \sigma_2 c_{h+1,i,j},
\\
&
{\textstyle(
(-1)^h k_3
+
(-1)^{h+j} k_1
-
(-1)^j\theta_i
+
\frac{1}{2}
)}
{\left\lfloor \frac{n-t_1-i-j}{2}\right\rfloor+\left\lfloor\frac{i-h}{2}\right\rfloor \choose \left\lfloor\frac{i-h}{2}\right\rfloor}
c_{h,i,j}.
\end{align*}
It is straightforward to verify that the sum of the above three terms is equal to 
$$
(-1)^j
m_3^{(i+1)}
{\left\lfloor \frac{n-t_1-i-j-1}{2}\right\rfloor+\left\lfloor\frac{i-h+1}{2}\right\rfloor \choose \left\lfloor\frac{i-h+1}{2}\right\rfloor}
c_{h,i+1,j}.  
$$
By the above comments the equations given in (\ref{e:n>=t1+t2(even)-1}) follow.

Using Theorem \ref{thm:BImodule_Mn}(i) yields that the coefficient of 
$[x_2]^{n-t_1}_j [x_3]^i_h [x_1]^{n-t_2}_{n-h-j}$ in $(Z-\theta_i^*)q_i$ is equal to $(-1)^\frac{h+j}{2}$ times the sum of 
\begin{align*}
&
m_2^{(j)}
{\left\lfloor \frac{n-t_1-i-j+1}{2}\right\rfloor+\left\lfloor\frac{i-h}{2}\right\rfloor \choose \left\lfloor\frac{i-h}{2}\right\rfloor}
 \sigma_3 c_{h,i,j-1},
\\
&
m_1^{(n-h-j)}
{\left\lfloor \frac{n-t_1-i-j-1}{2}\right\rfloor+\left\lfloor\frac{i-h}{2}\right\rfloor \choose \left\lfloor\frac{i-h}{2}\right\rfloor}
 \sigma_3 c_{h,i,j+1},
\\
&
{\textstyle(
(-1)^j k_2
+
(-1)^{h+j} k_1
-
(-1)^h
\theta_i^*
+
\frac{1}{2}
)}
{\left\lfloor \frac{n-t_1-i-j}{2}\right\rfloor+\left\lfloor\frac{i-h}{2}\right\rfloor \choose \left\lfloor\frac{i-h}{2}\right\rfloor}
c_{h,i,j}.
\end{align*}
It is straightforward to verify that the sum of the above three terms is equal to $(-1)^h$ times 
$$
\left\{
\begin{array}{ll}
\begin{pmatrix}
0 &0
\\
0 &0
\end{pmatrix}
\qquad &\hbox{if $h=i$},
\\
(2k_1+n-i+1)
\displaystyle{\left\lfloor \frac{n-t_1-i-j+1}{2}\right\rfloor+\left\lfloor\frac{i-h-1}{2}\right\rfloor \choose \left\lfloor\frac{i-h-1}{2}\right\rfloor}
c_{h,i-1,j}
\qquad &\hbox{if $h<i$ and $i$ is even},
\\
(2k_1+2k_2+n-i+1)
\displaystyle{\left\lfloor \frac{n-t_1-i-j+1}{2}\right\rfloor+\left\lfloor\frac{i-h-1}{2}\right\rfloor \choose \left\lfloor\frac{i-h-1}{2}\right\rfloor}
c_{h,i-1,j}
\qquad &\hbox{if $h<i$ and $i$ is odd}.
\end{array}
\right.
$$
By the above comments the equations given in (\ref{e:n>=t1+t2(even)-2}) follow.

(iii): It is routine to verify that $\D(q_0)=0$. Combined with (ii) the statement (iii) follows.
\end{proof}

\begin{lem}\label{lem:basis:n>=t1+t2(even)}
Suppose that $n$ is an even integer with $n\geq t_1+t_2+t_3$. 
Let $\{q_i\}_{i=0}^{n-t_1-t_2}$ be as in Proposition \ref{prop:n>=t1+t2(even)}. Then the following hold:
\begin{enumerate}
\item $\M_n(x_1)\cap \M_n(x_2)\cap \M_n(x_3)$ has the basis 
\begin{align}
q_i\cdot 
\begin{pmatrix}
1
\\
0
\end{pmatrix}
\otimes
1 
\qquad 
(t_3\leq i\leq n-t_1-t_2), 
\label{M123basis:n>=t1+t2(even)-1}
\\ 
q_i\cdot 
\begin{pmatrix}
0
\\
1
\end{pmatrix}
\otimes
1 
\qquad 
(t_3\leq i\leq n-t_1-t_2).
\label{M123basis:n>=t1+t2(even)-2}
\end{align}

\item $\M_n/\M_n(x_3)$ has the basis
\begin{align}
q_i\cdot 
\begin{pmatrix}
1
\\
0
\end{pmatrix}
\otimes
1 
+
\M_n(x_3)
\qquad 
(0\leq i\leq t_3-1), 
\label{M/M3basis:n>=t1+t2(even)-1}
\\ 
q_i\cdot 
\begin{pmatrix}
0
\\
1
\end{pmatrix}
\otimes
1 
+
\M_n(x_3)
\qquad 
(0\leq i\leq t_3-1).
\label{M/M3basis:n>=t1+t2(even)-2}
\end{align}
\end{enumerate}
\end{lem}
\begin{proof}
(i): The linear independence of  (\ref{M123basis:n>=t1+t2(even)-1}) and (\ref{M123basis:n>=t1+t2(even)-2}) follows from Proposition \ref{prop:n>=t1+t2(even)}(i). Since (\ref{M123basis:n>=t1+t2(even)-1}) and (\ref{M123basis:n>=t1+t2(even)-2}) are in 
$
\C^2\otimes\left(\bigoplus\limits_{i=t_1}^n x_1^i\cdot \R[x_2,x_3]_{n-i}
\cap
\bigoplus\limits_{i=t_2}^n x_2^i\cdot \R[x_1,x_3]_{n-i}
\cap 
\bigoplus\limits_{i=t_3}^n x_3^i\cdot \R[x_1,x_2]_{n-i}
\right)$,  
it follows from Proposition \ref{prop:n>=t1+t2(even)}(iii) that (\ref{M123basis:n>=t1+t2(even)-1}) and (\ref{M123basis:n>=t1+t2(even)-2}) are in $\M_n(x_1)\cap \M_n(x_2)\cap \M_n(x_3)$. Combined with Proposition \ref{prop:dimM(x123)} the statement (i) follows.

(ii): The linear independence of  (\ref{M/M3basis:n>=t1+t2(even)-1}) and (\ref{M/M3basis:n>=t1+t2(even)-2}) follows from Proposition \ref{prop:n>=t1+t2(even)}(i). 
By Proposition \ref{prop:n>=t1+t2(even)}(iii) the cosets (\ref{M/M3basis:n>=t1+t2(even)-1}) and (\ref{M/M3basis:n>=t1+t2(even)-2}) are in 
$
\M_n/\M_n(x_3).
$  
By  Theorem \ref{thm:dimM=2(n+1)} and Proposition \ref{prop:dimM(x1)}(iii) the dimension of $\M_n/\M_n(x_3)$ is $2t_3$. The statement (ii) follows.
\end{proof}

\begin{prop}\label{prop:n>=t2+t3(even)}
Suppose that $n$ is an even integer with $n\geq t_2+t_3$. 
Let 
$$
r_i
=
\sum_{h=0}^i
(-1)^{-\frac{h}{2}}
\sum_{j=t_3}^{n-t_2-h}
(-1)^{\frac{j}{2}}
{\left\lfloor \frac{n-t_2-i-j}{2}\right\rfloor+\left\lfloor\frac{i-h}{2}\right\rfloor \choose \left\lfloor\frac{i-h}{2}\right\rfloor}
c_{h,i,j}
\otimes
[x_3]^{n-t_2}_j[x_1]^i_h[x_2]^{n-t_3}_{n-h-j}
$$
for all $i=0,1,\ldots,n-t_2-t_3$ where
\begin{align*}
c_{h,i,j}&=
\left\{
\begin{array}{ll}
\sigma_1^j \sigma_3
\qquad &\hbox{if $h$ is odd and $i$ is odd},
\\
\sigma_1^j
\qquad &\hbox{if $h$ is even and $i$ is even},
\\
(-1)^\frac{1}{2} \sigma_2
\qquad &\hbox{if $h$ is odd, $i$ is even and $j$ is odd},
\\
\begin{pmatrix}
1 &0
\\
0 &1
\end{pmatrix}
\qquad &\hbox{if $h$ is even, $i$ is odd and $j$ is even},
\\
\begin{pmatrix}
0&0
\\
0 &0
\end{pmatrix}
\qquad &\hbox{else}
\end{array}
\right.
\end{align*}
for any integers $h,i,j$. 
Then the following hold:
\begin{enumerate}
\item $\{r_i\}_{i=0}^{n-t_2-t_3}$ are linearly independent over ${\rm Mat}_2(\C)$.

\item The following equations hold:
\begin{align}
(Z-\theta_i)r_i &=r_{i+1} \qquad (0\leq i\leq n-t_2-t_3-1),
\qquad 
(Z-\theta_{n-t_2-t_3}) r_{n-t_2-t_3}=0,
\label{e:n>=t2+t3(even)-1}
\\
(X-\theta_i^*) r_i &=\varphi_i r_{i-1} \qquad (1\leq i\leq n-t_2-t_3),
\qquad 
(X-\theta_0^*) r_0=0,
\label{e:n>=t2+t3(even)-2}
\end{align}
where
\begin{align*}
\theta_i &=(-1)^i
\textstyle(
k_2-k_1-i-\frac{1}{2}
)
\qquad 
(0\leq i\leq n-t_2-t_3),
\\
\theta_i^* &=(-1)^i
\textstyle(
k_2+k_3+n-i+\frac{1}{2}
)
\qquad 
(0\leq i\leq n-t_2-t_3),
\\
\varphi_i
&=
\left\{
\begin{array}{ll}
i(2 k_2+n-i+1)
\qquad
\hbox{if $i$ is even},
%\hbox{for $i=2,4,\ldots,n-t_2-t_3$},
\\
(i+2 k_1)(2 k_2+2 k_3+n-i+1)
\qquad
\hbox{if $i$ is odd}
%\hbox{for $i=1,3,\ldots,n-t_2-t_3-1$}.
\end{array}
\right.
\qquad 
(1\leq i\leq n-t_2-t_3).
\end{align*}

\item $\D(r_i)=0$ for all $i=0,1,\ldots,n-t_2-t_3$.
\end{enumerate}
\end{prop}
\begin{proof}
(i): Let $i$ be an integer with $0\leq i\leq n-t_2-t_3$.
By construction the coefficient of $x_1^h$ in $r_i$ is zero for all integers $h$ with $i<h\leq n-t_2-t_3$. Observe that the coefficient of $x_3^{n-t_2-i}x_1^i x_2^{t_2}$ in $r_i$ is 
\begin{gather}\label{coeff:n>=t2+t3(even)}
(-1)^{\frac{n-t_2}{2}-i}
\prod_{h=n-t_2-i+1}^{n-t_2} m_3^{(h)}
\prod_{h=t_2+1}^{n-t_3} m_2^{(h)}
\times 
\left\{
\begin{array}{ll}
\sigma_3
\qquad 
&\hbox{if $i$ is odd},
\\
\sigma_1 
\qquad 
&\hbox{if $i$ is even}.
\end{array}
\right.
\end{gather}
Clearly $\prod\limits_{h=t_2+1}^{n-t_3} m_2^{(h)}$ is nonzero.
Since $i\leq n-t_2-t_3$ the scalar $\prod\limits_{h=n-t_2-i+1}^{n-t_2} m_3^{(h)}$ is nonzero. 
Hence the matrix (\ref{coeff:n>=t2+t3(even)}) is nonsingular. By the above comments the part (i) follows.

(ii): 
Let $i$ be a nonnegative integer and let $j$ be an integer with $t_3\leq j\leq n-t_2-i-1$. 
Using Theorem \ref{thm:BImodule_Mn}(i) yields that the coefficient of 
$[x_3]^{n-t_2}_j [x_1]^{i+1}_{i+1} [x_2]^{n-t_3}_{n-i-j-1} $ in $(Z-\theta_i)r_i$ is equal to 
$$
(-1)^{-\frac{i+j+1}{2}}\sigma_3 c_{i,i,j}=(-1)^{\frac{-i+j-1}{2}}c_{i+1,i+1,j}.
$$
Now let $h,i,j$ denote three integers with $0\leq h\leq i\leq n-t_2-t_3$ and $t_3\leq j\leq n-t_2-h$. Using Theorem \ref{thm:BImodule_Mn}(i) yields that the coefficient of 
$[x_3]^{n-t_2}_j [x_1]^i_h [x_2]^{n-t_3}_{n-h-j}$ in $(Z-\theta_i)r_i$ is equal to $(-1)^{-\frac{h+j}{2}}$ times the sum of 
\begin{align*}
&
m_1^{(h)}
{\left\lfloor \frac{n-t_2-i-j}{2}\right\rfloor+\left\lfloor\frac{i-h+1}{2}\right\rfloor \choose \left\lfloor\frac{i-h+1}{2}\right\rfloor}
 \sigma_3 c_{h-1,i,j},
\\
&
m_2^{(n-h-j)}
{\left\lfloor \frac{n-t_2-i-j}{2}\right\rfloor+\left\lfloor\frac{i-h-1}{2}\right\rfloor \choose \left\lfloor\frac{i-h-1}{2}\right\rfloor}
 \sigma_3 c_{h+1,i,j},
\\
&
{\textstyle(
(-1)^h k_1
+
(-1)^{h+j} k_2
-
(-1)^j\theta_i
+
\frac{1}{2}
)}
{\left\lfloor \frac{n-t_2-i-j}{2}\right\rfloor+\left\lfloor\frac{i-h}{2}\right\rfloor \choose \left\lfloor\frac{i-h}{2}\right\rfloor}
c_{h,i,j}.
\end{align*}
It is straightforward to verify that the sum of the above three terms is equal to 
$$
(-1)^j
m_1^{(i+1)}
{\left\lfloor \frac{n-t_2-i-j-1}{2}\right\rfloor+\left\lfloor\frac{i-h+1}{2}\right\rfloor \choose \left\lfloor\frac{i-h+1}{2}\right\rfloor}
c_{h,i+1,j}.  
$$
By the above comments the equations given in (\ref{e:n>=t2+t3(even)-1}) follow.

Using Theorem \ref{thm:BImodule_Mn}(i) yields that the coefficient of 
$[x_3]^{n-t_2}_j [x_1]^i_h [x_2]^{n-t_3}_{n-h-j}$ in $(X-\theta_i^*)r_i$ is equal to $(-1)^\frac{h+j}{2}$ times the sum of 
\begin{align*}
&
m_3^{(j)}
{\left\lfloor \frac{n-t_2-i-j+1}{2}\right\rfloor+\left\lfloor\frac{i-h}{2}\right\rfloor \choose \left\lfloor\frac{i-h}{2}\right\rfloor}
 \sigma_1 c_{h,i,j-1},
\\
&
m_2^{(n-h-j)}
{\left\lfloor \frac{n-t_2-i-j-1}{2}\right\rfloor+\left\lfloor\frac{i-h}{2}\right\rfloor \choose \left\lfloor\frac{i-h}{2}\right\rfloor}
 \sigma_1 c_{h,i,j+1},
\\
&
{\textstyle(
(-1)^j k_3
+
(-1)^{h+j} k_2
-
(-1)^h
\theta_i^*
+
\frac{1}{2}
)}
{\left\lfloor \frac{n-t_2-i-j}{2}\right\rfloor+\left\lfloor\frac{i-h}{2}\right\rfloor \choose \left\lfloor\frac{i-h}{2}\right\rfloor}
c_{h,i,j}.
\end{align*}
It is straightforward to verify that the sum of the above three terms is equal to $(-1)^h$ times 
$$
\left\{
\begin{array}{ll}
\begin{pmatrix}
0 &0
\\
0 &0
\end{pmatrix}
\qquad &\hbox{if $h=i$},
\\
(2k_2+n-i+1)
\displaystyle{\left\lfloor \frac{n-t_2-i-j+1}{2}\right\rfloor+\left\lfloor\frac{i-h-1}{2}\right\rfloor \choose \left\lfloor\frac{i-h-1}{2}\right\rfloor}
c_{h,i-1,j}
\qquad &\hbox{if $h<i$ and $i$ is even},
\\
(2k_2+2k_3+n-i+1)
\displaystyle{\left\lfloor \frac{n-t_2-i-j+1}{2}\right\rfloor+\left\lfloor\frac{i-h-1}{2}\right\rfloor \choose \left\lfloor\frac{i-h-1}{2}\right\rfloor}
c_{h,i-1,j}
\qquad &\hbox{if $h<i$ and $i$ is odd}.
\end{array}
\right.
$$
By the above comments the equations given in (\ref{e:n>=t2+t3(even)-2}) follow.

(iii): It is routine to verify that $\D(r_0)=0$. Combined with (ii) the statement (iii) follows.
\end{proof}

\begin{lem}\label{lem:basis:n>=t2+t3(even)}
Suppose that $n$ is an even integer with $n\geq t_1+t_2+t_3$. 
Let $\{r_i\}_{i=0}^{n-t_2-t_3}$ be as in Proposition \ref{prop:n>=t2+t3(even)}. Then the following hold:
\begin{enumerate}
\item $\M_n(x_1)\cap \M_n(x_2)\cap \M_n(x_3)$ has the basis 
\begin{align}
r_i\cdot 
\begin{pmatrix}
1
\\
0
\end{pmatrix}
\otimes
1 
\qquad 
(t_1\leq i\leq n-t_2-t_3), 
\label{M123basis:n>=t2+t3(even)-1}
\\ 
r_i\cdot 
\begin{pmatrix}
0
\\
1
\end{pmatrix}
\otimes
1 
\qquad 
(t_1\leq i\leq n-t_2-t_3).
\label{M123basis:n>=t2+t3(even)-2}
\end{align}

\item $\M_n/\M_n(x_1)$ has the basis
\begin{align}
r_i\cdot 
\begin{pmatrix}
1
\\
0
\end{pmatrix}
\otimes
1 
+
\M_n(x_1)
\qquad 
(0\leq i\leq t_1-1), 
\label{M/M1basis:n>=t2+t3(even)-1}
\\ 
r_i\cdot 
\begin{pmatrix}
0
\\
1
\end{pmatrix}
\otimes
1 
+
\M_n(x_1)
\qquad 
(0\leq i\leq t_1-1).
\label{M/M1basis:n>=t2+t3(even)-2}
\end{align}
\end{enumerate}
\end{lem}
\begin{proof}
(i): The linear independence of  (\ref{M123basis:n>=t2+t3(even)-1}) and (\ref{M123basis:n>=t2+t3(even)-2}) follows from Proposition \ref{prop:n>=t2+t3(even)}(i). Since (\ref{M123basis:n>=t2+t3(even)-1}) and (\ref{M123basis:n>=t2+t3(even)-2}) are in 
$
\C^2\otimes\left(\bigoplus\limits_{i=t_1}^n x_1^i\cdot \R[x_2,x_3]_{n-i}
\cap
\bigoplus\limits_{i=t_2}^n x_2^i\cdot \R[x_1,x_3]_{n-i}
\cap 
\bigoplus\limits_{i=t_3}^n x_3^i\cdot \R[x_1,x_2]_{n-i}
\right)$,  
it follows from Proposition \ref{prop:n>=t2+t3(even)}(iii) that (\ref{M123basis:n>=t2+t3(even)-1}) and (\ref{M123basis:n>=t2+t3(even)-2}) are in $\M_n(x_1)\cap \M_n(x_2)\cap \M_n(x_3)$. Combined with Proposition \ref{prop:dimM(x123)} the statement (i) follows.

(ii): The linear independence of  (\ref{M/M1basis:n>=t2+t3(even)-1}) and (\ref{M/M1basis:n>=t2+t3(even)-2}) follows from Proposition \ref{prop:n>=t2+t3(even)}(i). 
By Proposition \ref{prop:n>=t2+t3(even)}(iii) the cosets (\ref{M/M1basis:n>=t2+t3(even)-1}) and (\ref{M/M1basis:n>=t2+t3(even)-2}) are in 
$
\M_n/\M_n(x_1).
$  
By  Theorem \ref{thm:dimM=2(n+1)} and Proposition \ref{prop:dimM(x1)}(i) the dimension of $\M_n/\M_n(x_1)$ is $2t_1$. The statement (ii) follows.
\end{proof}

\begin{lem}\label{lem:basis:n+1>=t1+t2+t3(even)}
Suppose that $n$ is an even integer with $n\geq t_1+t_2+t_3$. 
Let 
$$
\{p_i\}_{i=0}^{n-t_1-t_3},
\qquad 
\{q_i\}_{i=0}^{n-t_1-t_2},
\qquad 
\{r_i\}_{i=0}^{n-t_2-t_3}
$$ 
be as in Propositions \ref{prop:n>=t1+t3(even)}, \ref{prop:n>=t1+t2(even)}, \ref{prop:n>=t2+t3(even)} respectively. Then the following hold:
\begin{enumerate}
\item $\M_n(x_1)/\M_n(x_1)\cap \M_n(x_2)\cap \M_n(x_3)$ has the basis 
\begin{align}
p_i\cdot 
\begin{pmatrix}
1
\\
0
\end{pmatrix}
\otimes
1
+
\M_n(x_1)\cap \M_n(x_2)\cap \M_n(x_3)
\qquad 
(0\leq i\leq t_2-1), 
\label{M1/M123basis:n+1>=t1+t2+t3(even)-1}
\\ 
p_i\cdot 
\begin{pmatrix}
0
\\
1
\end{pmatrix}
\otimes
1 
+
\M_n(x_1)\cap \M_n(x_2)\cap \M_n(x_3)
\qquad 
(0\leq i\leq t_2-1),
\label{M1/M123basis:n+1>=t1+t2+t3(even)-2}
\\
q_i\cdot 
\begin{pmatrix}
1
\\
0
\end{pmatrix}
\otimes
1
+
\M_n(x_1)\cap \M_n(x_2)\cap \M_n(x_3)
\qquad 
(0\leq i\leq t_3-1), 
\label{M1/M123basis:n+1>=t1+t2+t3(even)-3}
\\
q_i\cdot 
\begin{pmatrix}
0
\\
1
\end{pmatrix}
\otimes
1 
+
\M_n(x_1)\cap \M_n(x_2)\cap \M_n(x_3)
\qquad 
(0\leq i\leq t_3-1).
\label{M1/M123basis:n+1>=t1+t2+t3(even)-4}
\end{align}

\item $\M_n(x_2)/\M_n(x_1)\cap \M_n(x_2)\cap \M_n(x_3)$ has the basis 
\begin{align}
q_i\cdot 
\begin{pmatrix}
1
\\
0
\end{pmatrix}
\otimes
1
+
\M_n(x_1)\cap \M_n(x_2)\cap \M_n(x_3)
\qquad 
(0\leq i\leq t_3-1), 
\label{M2/M123basis:n+1>=t1+t2+t3(even)-1}
\\ 
q_i\cdot 
\begin{pmatrix}
0
\\
1
\end{pmatrix}
\otimes
1 
+
\M_n(x_1)\cap \M_n(x_2)\cap \M_n(x_3)
\qquad 
(0\leq i\leq t_3-1),
\label{M2/M123basis:n+1>=t1+t2+t3(even)-2}
\\
r_i\cdot 
\begin{pmatrix}
1
\\
0
\end{pmatrix}
\otimes
1
+
\M_n(x_1)\cap \M_n(x_2)\cap \M_n(x_3)
\qquad 
(0\leq i\leq t_1-1), 
\label{M2/M123basis:n+1>=t1+t2+t3(even)-3}
\\
r_i\cdot 
\begin{pmatrix}
0
\\
1
\end{pmatrix}
\otimes
1 
+
\M_n(x_1)\cap \M_n(x_2)\cap \M_n(x_3)
\qquad 
(0\leq i\leq t_1-1).
\label{M2/M123basis:n+1>=t1+t2+t3(even)-4}
\end{align}

\item $\M_n(x_3)/\M_n(x_1)\cap \M_n(x_2)\cap \M_n(x_3)$ has the basis 
\begin{align}
r_i\cdot 
\begin{pmatrix}
1
\\
0
\end{pmatrix}
\otimes
1
+
\M_n(x_1)\cap \M_n(x_2)\cap \M_n(x_3)
\qquad 
(0\leq i\leq t_1-1), 
\label{M3/M123basis:n+1>=t1+t2+t3(even)-1}
\\ 
r_i\cdot 
\begin{pmatrix}
0
\\
1
\end{pmatrix}
\otimes
1 
+
\M_n(x_1)\cap \M_n(x_2)\cap \M_n(x_3)
\qquad 
(0\leq i\leq t_1-1),
\label{M3/M123basis:n+1>=t1+t2+t3(even)-2}
\\
p_i\cdot 
\begin{pmatrix}
1
\\
0
\end{pmatrix}
\otimes
1
+
\M_n(x_1)\cap \M_n(x_2)\cap \M_n(x_3)
\qquad 
(0\leq i\leq t_2-1), 
\label{M3/M123basis:n+1>=t1+t2+t3(even)-3}
\\
p_i\cdot 
\begin{pmatrix}
0
\\
1
\end{pmatrix}
\otimes
1 
+
\M_n(x_1)\cap \M_n(x_2)\cap \M_n(x_3)
\qquad 
(0\leq i\leq t_2-1).
\label{M3/M123basis:n+1>=t1+t2+t3(even)-4}
\end{align}
\end{enumerate}
\end{lem}
\begin{proof}
(i): 
Observe that (\ref{M1/M123basis:n+1>=t1+t2+t3(even)-1})--(\ref{M1/M123basis:n+1>=t1+t2+t3(even)-4}) are in 
$
\M_n(x_1)/\M_n(x_1)\cap \M_n(x_2)\cap \M_n(x_3).
$ 
The linear independence of  (\ref{M1/M123basis:n+1>=t1+t2+t3(even)-1})--(\ref{M1/M123basis:n+1>=t1+t2+t3(even)-4}) follows from Propositions \ref{prop:n>=t1+t3(even)}(i) and \ref{prop:n>=t1+t2(even)}(i). 
By Propositions \ref{prop:dimM(x1)}(i) and \ref{prop:dimM(x123)} the dimension of $\M_n(x_1)/\M_n(x_1)\cap \M_n(x_2)\cap \M_n(x_3)$ is $2(t_2+t_3)$.
The statement (i) follows.

(ii): 
Observe that (\ref{M2/M123basis:n+1>=t1+t2+t3(even)-1})--(\ref{M2/M123basis:n+1>=t1+t2+t3(even)-4}) are in 
$
\M_n(x_2)/\M_n(x_1)\cap \M_n(x_2)\cap \M_n(x_3).
$ 
The linear independence of  (\ref{M2/M123basis:n+1>=t1+t2+t3(even)-1})--(\ref{M2/M123basis:n+1>=t1+t2+t3(even)-4}) follows from Propositions \ref{prop:n>=t1+t2(even)}(i) and \ref{prop:n>=t2+t3(even)}(i). 
By Propositions \ref{prop:dimM(x1)}(ii) and \ref{prop:dimM(x123)} the dimension of $\M_n(x_2)/\M_n(x_1)\cap \M_n(x_2)\cap \M_n(x_3)$ is $2(t_1+t_3)$.
The statement (ii) follows.

(iii): 
Observe that (\ref{M3/M123basis:n+1>=t1+t2+t3(even)-1})--(\ref{M3/M123basis:n+1>=t1+t2+t3(even)-4}) are in 
$
\M_n(x_3)/\M_n(x_1)\cap \M_n(x_2)\cap \M_n(x_3).
$ 
The linear independence of  (\ref{M3/M123basis:n+1>=t1+t2+t3(even)-1})--(\ref{M3/M123basis:n+1>=t1+t2+t3(even)-4}) follows from Propositions \ref{prop:n>=t2+t3(even)}(i) and \ref{prop:n>=t1+t3(even)}(i). 
By Propositions \ref{prop:dimM(x1)}(iii) and \ref{prop:dimM(x123)} the dimension of $\M_n(x_3)/\M_n(x_1)\cap \M_n(x_2)\cap \M_n(x_3)$ is $2(t_1+t_2)$.
The statement (iii) follows. 
\end{proof}

\begin{thm}\label{thm:n+1>=t1+t2+t3(even)}
Suppose that $n$ is an even  integer with $n\geq t_1+t_2+t_3$. Then the following hold: 
\begin{enumerate}
\item The $\BI$-module $\M_n(x_1)\cap \M_n(x_2)\cap \M_n(x_3)$ is isomorphic to a direct sum of two copies of 
\begin{gather}\label{M123(even)}
E_{n-t_1-t_2-t_3}
\textstyle(
k_1+\frac{n+1}{2},
k_2+\frac{n+1}{2},
k_3+\frac{n+1}{2}
).
\end{gather}

\item The $\BI$-module $\M_n/\M_n(x_1)$ is isomorphic to a direct sum of two copies of 
\begin{gather}\label{M/M1(even)}
O_{t_1-1}(k_1+k_2+k_3+n+1,k_3,k_2).
\end{gather}

\item The $\BI$-module $\M_n/\M_n(x_2)$ is isomorphic to a direct sum of two copies of 
\begin{gather}\label{M/M2(even)}
O_{t_2-1}(k_3,k_1+k_2+k_3+n+1,k_1).
\end{gather}

\item The $\BI$-module $\M_n/\M_n(x_3)$ is isomorphic to a direct sum of two copies of 
\begin{gather}\label{M/M3(even)}
O_{t_3-1}(k_2,k_1,k_1+k_2+k_3+n+1).
\end{gather}

\item The $\BI$-module $\M_n(x_1)/\M_n(x_1)\cap \M_n(x_2)\cap \M_n(x_3)$ 
is isomorphic to a direct sum of two copies of (\ref{M/M2(even)}) and (\ref{M/M3(even)}).

\item The $\BI$-module $\M_n(x_2)/\M_n(x_1)\cap \M_n(x_2)\cap \M_n(x_3)$ 
is isomorphic to a direct sum of two copies of (\ref{M/M1(even)}) and (\ref{M/M3(even)}).

\item The $\BI$-module $\M_n(x_3)/\M_n(x_1)\cap \M_n(x_2)\cap \M_n(x_3)$ 
is isomorphic to a direct sum of two copies of (\ref{M/M1(even)}) and (\ref{M/M2(even)}).
\end{enumerate}
Moreover the $\BI$-modules (\ref{M123(even)})--(\ref{M/M3(even)}) are irreducible.
\end{thm}
\begin{proof}
(i):  Let $V$ and $V'$ denote the subspaces of $\M_n$ spanned by (\ref{M123basis:n>=t1+t3(even)-1}) and (\ref{M123basis:n>=t1+t3(even)-2}), respectively. It follows from Lemma \ref{lem:basis:n>=t1+t3(even)}(i) that $\M_n(x_1)\cap \M_n(x_2)\cap \M_n(x_3)=V\oplus V'$. It follows from Theorem \ref{thm:BImodule_Mn}(ii) and Proposition \ref{prop:n>=t1+t3(even)}(i) that $V$ and $V'$ are two isomorphic $\BI$-modules. 
Compared with Proposition \ref{prop:Ed} both are isomorphic to 
\begin{gather}\label{M123'(even)}
E_{n-t_1-t_2-t_3}
(
\textstyle
k_1+\frac{n+1}{2},
-k_2-\frac{n+1}{2},
k_3+\frac{n+1}{2}
).
\end{gather}
Using Theorem \ref{thm:irr_E} yields that the $\BI$-module (\ref{M123'(even)}) is irreducible. By Theorem \ref{thm:onto2_E} the $\BI$-module (\ref{M123'(even)}) is isomorphic to (\ref{M123(even)}). The statement (i) follows.

(ii): Let $V$ and $V'$ denote the subspaces of $\M_n/\M_n(x_1)$ spanned by (\ref{M/M1basis:n>=t2+t3(even)-1}) and (\ref{M/M1basis:n>=t2+t3(even)-2}), respectively. 
It follows from Lemma \ref{lem:basis:n>=t2+t3(even)}(ii) that $\M_n/\M_n(x_1)=V\oplus V'$. It follows from Theorem \ref{thm:BImodule_Mn}(ii) and Proposition \ref{prop:n>=t2+t3(even)}(ii) that $V$ and $V'$ are two isomorphic $\BI$-modules. 
Compared with Proposition \ref{prop:Od} both are isomorphic to 
\begin{gather}\label{M/M1'(even)}
O_{t_1-1}(-k_2,-k_1-k_2-k_3-n-1,k_3)^{((-1,-1),(1\,2\,3))}.
\end{gather}
Using Theorem \ref{thm:irr_O} yields that the $\BI$-module (\ref{M/M1'(even)}) is irreducible. By Theorem \ref{thm:onto2_O} the $\BI$-module (\ref{M/M1'(even)}) is isomorphic to (\ref{M/M1(even)}). The statement (ii) follows.

(iii): Let $V$ and $V'$ denote the subspaces of $\M_n/\M_n(x_2)$ spanned by (\ref{M/M2basis:n>=t1+t3(even)-1}) and (\ref{M/M2basis:n>=t1+t3(even)-2}), respectively. 
It follows from Lemma \ref{lem:basis:n>=t1+t3(even)}(ii) that $\M_n/\M_n(x_2)=V\oplus V'$. It follows from Theorem \ref{thm:BImodule_Mn}(ii) and Proposition \ref{prop:n>=t1+t3(even)}(ii) that $V$ and $V'$ are two isomorphic $\BI$-modules. 
Compared with Proposition \ref{prop:Od} both are isomorphic to 
\begin{gather}\label{M/M2'(even)}
O_{t_2-1}(-k_3,-k_1-k_2-k_3-n-1,k_1)^{(-1,-1)}.
\end{gather}
Using Theorem \ref{thm:irr_O} yields that the $\BI$-module (\ref{M/M2'(even)}) is irreducible. By Theorem \ref{thm:onto2_O} the $\BI$-module (\ref{M/M2'(even)}) is isomorphic to (\ref{M/M2(even)}). The statement (iii) follows.

(iv): Let $V$ and $V'$ denote the subspaces of $\M_n/\M_n(x_3)$ spanned by (\ref{M/M3basis:n>=t1+t2(even)-1}) and (\ref{M/M3basis:n>=t1+t2(even)-2}), respectively. 
It follows from Lemma \ref{lem:basis:n>=t1+t2(even)}(ii) that $\M_n/\M_n(x_3)=V\oplus V'$. It follows from Theorem \ref{thm:BImodule_Mn}(ii) and Proposition \ref{prop:n>=t1+t2(even)}(ii) that $V$ and $V'$ are two isomorphic $\BI$-modules. 
Compared with Proposition \ref{prop:Ed} both are isomorphic to 
\begin{gather}\label{M/M3'(even)}
O_{t_3-1}(-k_1,-k_1-k_2-k_3-n-1,k_2)^{((-1,-1),(1\,3\,2))}.
\end{gather}
Using Theorem \ref{thm:irr_O} yields that the $\BI$-module (\ref{M/M3'(even)}) is irreducible. By Theorem \ref{thm:onto2_O} the $\BI$-module (\ref{M/M3'(even)}) is isomorphic to (\ref{M/M3(even)}). The statement (iv) follows.

(v): 
Let $V,V',W,W'$ denote the subspaces of $\M_n(x_1)/\M_n(x_1)\cap \M_n(x_2)\cap \M_n(x_3)$ spanned by (\ref{M1/M123basis:n+1>=t1+t2+t3(even)-1})--(\ref{M1/M123basis:n+1>=t1+t2+t3(even)-4}), respectively. 
It follows from Lemma \ref{lem:basis:n+1>=t1+t2+t3(even)}(i) that 
$$
\M_n(x_1)/\M_n(x_1)\cap \M_n(x_2)\cap \M_n(x_3)=V\oplus V'\oplus W\oplus W'.
$$ 
Similar to the proof of Theorem \ref{thm:n+1>=t1+t2+t3(even)}(iii) it follows that $V$ and $V'$ are isomorphic to the $\BI$-module (\ref{M/M2(even)}). Similar to the proof of Theorem \ref{thm:n+1>=t1+t2+t3(even)}(iv) it follows that $W$ and $W'$ are isomorphic to the $\BI$-module (\ref{M/M3(even)}). The statement (v) follows.

(vi): 
Let $V,V',W,W'$ denote the subspaces of $\M_n(x_1)/\M_n(x_1)\cap \M_n(x_2)\cap \M_n(x_3)$ spanned by (\ref{M2/M123basis:n+1>=t1+t2+t3(even)-1})--(\ref{M2/M123basis:n+1>=t1+t2+t3(even)-4}), respectively. 
It follows from Lemma \ref{lem:basis:n+1>=t1+t2+t3(even)}(ii) that 
$$
\M_n(x_2)/\M_n(x_1)\cap \M_n(x_2)\cap \M_n(x_3)=V\oplus V'\oplus W\oplus W'.
$$ 
Similar to the proof of Theorem \ref{thm:n+1>=t1+t2+t3(even)}(iv) it follows that $V$ and $V'$ are isomorphic to the $\BI$-module (\ref{M/M3(even)}). Similar to the proof of Theorem \ref{thm:n+1>=t1+t2+t3(even)}(ii) it follows that $W$ and $W'$ are isomorphic to the $\BI$-module (\ref{M/M1(even)}). The statement (vi) follows.

(vii): 
Let $V,V',W,W'$ denote the subspaces of $\M_n(x_3)/\M_n(x_1)\cap \M_n(x_2)\cap \M_n(x_3)$ spanned by (\ref{M3/M123basis:n+1>=t1+t2+t3(even)-1})--(\ref{M3/M123basis:n+1>=t1+t2+t3(even)-4}), respectively. 
It follows from Lemma \ref{lem:basis:n+1>=t1+t2+t3(even)}(iii) that 
$$
\M_n(x_3)/\M_n(x_1)\cap \M_n(x_2)\cap \M_n(x_3)=V\oplus V'\oplus W\oplus W'.
$$ 
Similar to the proof of Theorem \ref{thm:n+1>=t1+t2+t3(even)}(ii) it follows that $V$ and $V'$ are isomorphic to the $\BI$-module (\ref{M/M1(even)}). Similar to the proof of Theorem \ref{thm:n+1>=t1+t2+t3(even)}(iii) it follows that $W$ and $W'$ are isomorphic to the $\BI$-module (\ref{M/M2(even)}). The statement (vii) follows.
\end{proof}

\subsection*{Acknowledgements}

The research is supported by the Ministry of Science and Technology of Taiwan under the project MOST 110-2115-M-008-008-MY2.

\bibliographystyle{amsplain}
\bibliography{MP}
\end{document}